\pdfoutput=1
\documentclass[10pt]{article}

\usepackage[T1]{fontenc}
\usepackage[utf8]{inputenc}
\usepackage{lmodern}
\usepackage[activate={true,nocompatibility},final]{microtype}

\usepackage[margin=1in]{geometry}

\usepackage{authblk}

\usepackage{amsmath}
\usepackage{amssymb,amsfonts}
\usepackage{mathtools}
\usepackage{bm}
\usepackage{nicefrac}
\allowdisplaybreaks

\usepackage{amsthm}
\usepackage{thm-restate}

\usepackage{graphicx}
\graphicspath{{figures/}{./}}
\usepackage{booktabs}
\usepackage{multirow,makecell,array}
\usepackage[font=small,labelfont=bf]{caption}
\usepackage{subcaption}
\usepackage{float}
\usepackage{algorithm}
\usepackage{algpseudocode}

\usepackage{enumitem}
\setlist[itemize]{leftmargin=2.2em,itemsep=2pt,topsep=2pt}
\setlist[enumerate]{leftmargin=2.2em,itemsep=2pt,topsep=2pt}

\usepackage[table]{xcolor}
\definecolor{LinkColor}{rgb}{0.10,0.40,0.75}
\definecolor{CiteColor}{rgb}{0.70,0.25,0.20}
\definecolor{UrlColor} {rgb}{0.20,0.50,0.50}
\definecolor{OurRowColor}{gray}{0.92}

\usepackage{tikz}
\usetikzlibrary{positioning,calc,arrows.meta}

\usepackage{url}
\usepackage{hyperref}
\hypersetup{
  colorlinks=true,
  linkcolor=LinkColor,
  citecolor=CiteColor,
  urlcolor=UrlColor,
  breaklinks=true,
  bookmarksnumbered=true,
}
\usepackage{bookmark}
\usepackage[capitalise,nameinlink,noabbrev,sort&compress]{cleveref}

\numberwithin{equation}{section}

\newcommand{\R}{\mathbb{R}}

\newcommand{\eps}{\varepsilon}

\DeclareMathOperator{\diag}{diag}

\DeclareMathOperator{\supp}{supp}

\DeclareMathOperator{\poly}{poly}

\DeclareMathOperator*{\argmax}{arg\,max}
\DeclareMathOperator*{\argmin}{arg\,min}

\DeclarePairedDelimiter{\norm}{\lVert}{\rVert}

\DeclarePairedDelimiterX{\inner}[2]{\langle}{\rangle}{#1,#2}

\newcommand{\fnorm}[1]{\norm{#1}_{F}}

\newcommand{\Sym}{\mathbb{S}}
\newcommand{\Tr}{\operatorname{Tr}}

\newcommand{\nnz}{\operatorname{nnz}}
\newcommand{\opt}{\star}

\DeclareMathOperator{\relint}{relint}
\DeclareMathOperator{\polylog}{polylog}
\DeclareMathOperator{\aff}{aff}
\renewcommand{\norm}[1]{\left\lVert#1\right\rVert}
\renewcommand{\inner}[1]{\left\langle#1\right\rangle}

\newcommand{\gsc}{\gamma_{\mathrm{sc}}}
\newcommand{\pmin}{p^\opt_{\min}}
\newcommand{\muopt}{\mu_\opt}
\newcommand{\Lopt}{\widehat{L}_\opt}
\newcommand{\Cnull}{\mathcal{C}_0}
\newcommand{\locnorm}[2]{\norm{#1}_{#2}}
\newcommand{\dist}{\operatorname{dist}}
\newcommand{\wsc}{\zeta}
\newcommand{\wscs}{\zeta_*}

\theoremstyle{plain}
\newtheorem{theorem}{Theorem}[section]
\newtheorem{proposition}[theorem]{Proposition}
\newtheorem{lemma}[theorem]{Lemma}
\newtheorem{corollary}[theorem]{Corollary}

\newtheorem{fact}[theorem]{Fact}

\theoremstyle{definition}
\newtheorem{definition}[theorem]{Definition}
\newtheorem{assumption}[theorem]{Assumption}

\newtheorem{problem}[theorem]{Open Problem}

\theoremstyle{remark}
\newtheorem{remark}[theorem]{Remark}

\crefname{theorem}{Theorem}{Theorems}          \Crefname{theorem}{Theorem}{Theorems}
\crefname{proposition}{Proposition}{Propositions}
\Crefname{proposition}{Proposition}{Propositions}
\crefname{lemma}{Lemma}{Lemmas}                \Crefname{lemma}{Lemma}{Lemmas}
\crefname{corollary}{Corollary}{Corollaries}   \Crefname{corollary}{Corollary}{Corollaries}
\crefname{conjecture}{Conjecture}{Conjectures} \Crefname{conjecture}{Conjecture}{Conjectures}
\crefname{fact}{Fact}{Facts}                   \Crefname{fact}{Fact}{Facts}
\crefname{definition}{Definition}{Definitions} \Crefname{definition}{Definition}{Definitions}
\crefname{assumption}{Assumption}{Assumptions} \Crefname{assumption}{Assumption}{Assumptions}
\crefname{example}{Example}{Examples}          \Crefname{example}{Example}{Examples}
\crefname{problem}{Open Problem}{Open Problems} \Crefname{problem}{Open Problem}{Open Problems}
\crefname{remark}{Remark}{Remarks}             \Crefname{remark}{Remark}{Remarks}
\crefname{claim}{Claim}{Claims}                \Crefname{claim}{Claim}{Claims}
\crefname{algorithm}{Algorithm}{Algorithms}    \Crefname{algorithm}{Algorithm}{Algorithms}

\title{Beyond Averaging in John Ellipsoid Approximation:\\[0.25em]
  High-Accuracy Algorithms in the Leverage-Score Model}
\author{
{\normalsize Xiaoyu Li$^{1}$\thanks{\texttt{xiaoyu.li2@unsw.edu.au}} \qquad
Junwei Yu$^{2}$\thanks{\texttt{yujunwei04@berkeley.edu}} \qquad
Jiaojiao Jiang$^{1}$\thanks{\texttt{jiaojiao.jiang@unsw.edu.au}} \qquad
Junbin Gao$^{3}$\thanks{\texttt{junbin.gao@sydney.edu.au}} \qquad
Andi Han$^{3}$\thanks{\texttt{andi.han@sydney.edu.au}}}\\
{\small $^{1}$University of New South Wales \quad $^{2}$University of California, Berkeley \quad $^{3}$University of Sydney}
}
\date{}

\begin{document}
\maketitle

\begin{abstract}
The John ellipsoid of a symmetric polytope $P=\{\mathbf{x}\in\R^d:\norm{\mathbf{A}\mathbf{x}}_\infty\le1\}$,
$\mathbf{A}\in\R^{n\times d}$, is computed by a long line of leverage-score algorithms, from Cohen, Cousins,
Lee and Yang (COLT 2019)~\cite{cohen2019} to its successors~\cite{woodruff2025,cao2025}, all reaching a
$(1+\eps)$-approximation in $\Theta(\eps^{-1}\log(n/d))$ iterations. We separate this complexity into
three costs the modern line conflates (\emph{certification}, \emph{identification}, and
\emph{accuracy}) and locate the historical $\eps^{-1}$ in the first alone. In the equivalent
D-optimal-design form $\min_{\mathbf{p}\in\Delta_n}-\log\det(\sum_ip_i\mathbf{a}_i\mathbf{a}_i^\top)$, the
leverage-score oracle is exactly the first-order oracle and the $(1+\eps)$-John guarantee the
Frank--Wolfe gap $g(\mathbf{p})\le\eps d$; through this dictionary the costs come apart. The $\eps^{-1}$ is a
\emph{certification} artifact: the uniform average of the iterates, the certificate used throughout the
line, has gap exactly $\Theta(1/T)$, however cheap each iteration is made. Pointed instead at the last
iterate the same oracle is fast: a warm-started accelerated method reaches the guarantee in
$C(\mathbf{A})+O(\sqrt\kappa\,\log(1/\eps))$ queries after an $\eps$-independent setup $C(\mathbf{A})$, and
once the optimal face is identified the facial problem is an unconstrained self-concordant minimization
whose Hessian the oracle recovers \emph{exactly}, so damped Newton needs only $O(\log\log(1/\eps))$
steps, for a total of $C(\mathbf{A})+O(d^2\log\log(1/\eps))$ queries. The accuracy dependence is thus
doubly logarithmic after an $\eps$-independent, condition-dependent setup; the open problem is the
remaining \emph{identification} cost (a condition-free bound on reaching the optimal face) and lower
bounds. Accuracy is not the obstruction.

\end{abstract}

\clearpage
\tableofcontents
\clearpage

\section{Introduction}
\label{sec:intro}

Let $\mathbf{A}\in\R^{n\times d}$ have rows $\mathbf{a}_1,\dots,\mathbf{a}_n$ and rank $d$, and let
$P=\{\mathbf{x}\in\R^d:|\mathbf{a}_i^\top \mathbf{x}|\le1,\ i\in[n]\}$ be the associated centrally symmetric polytope. The
\emph{John ellipsoid} of $P$ is its maximum-volume inscribed ellipsoid. It is a foundational
primitive: it rounds a convex body to near-isotropic position (an affine image with near-identity
covariance); it preconditions the Vaidya and John random walks used for sampling and volume
computation~\cite{chen2018,gustafson2018}; it centers cutting-plane methods~\cite{lsw2015}; and, in its
dual guise as \emph{D-/G-optimal design}, it is the basic object of optimal experimental
design~\cite{pukelsheim2006} and of $\ell_\infty$ regression---the John ellipsoid is exactly the
$\ell_\infty$ \emph{Lewis-weight} ellipsoid~\cite{cohen2015}. We study the iteration complexity of computing a $(1+\eps)$-approximate
John ellipsoid in the standard \emph{leverage-score oracle} model, in which one ``query'' returns, for
a chosen weighting $\mathbf{p}$ of the constraints, the vector of leverage scores
$v_i(\mathbf{p})=\mathbf{a}_i^\top\big(\sum_j p_j\mathbf{a}_j\mathbf{a}_j^\top\big)^{-1}\mathbf{a}_i$. This is the cost model of essentially every
modern algorithm for the problem, since one leverage-score computation dominates the cost of one
iteration.

\paragraph{The $\eps^{-1}$ barrier.}
A $(1+\eps)$-approximate John ellipsoid of a symmetric polytope is captured by a convex program in $n$
weights, equivalent to continuous D-optimal design and to $\ell_\infty$ Lewis-weight computation
(\Cref{sec:prelim}). The modern algorithms solve it by a first-order iteration whose core step is a
leverage-score computation, and they share a single iteration count:
Cohen--Cousins--Lee--Yang~\cite{cohen2019} (henceforth CCLY) use $T=O(\eps^{-1}\log(n/d))$ iterations
of a multiplicative-weights fixed point; the input-sparsity/treewidth algorithm~\cite{cao2025} and the
lazy-update/streaming algorithm~\cite{woodruff2025} inherit \emph{the same}
$O(\eps^{-1}\log(n/d))$ iteration count, improving only the per-iteration cost
(\Cref{tab:results}). The $\eps^{-1}$ has not been improved in this line.

This is puzzling, because both \emph{slower-per-step} and \emph{higher-order} classical methods for the
same program already avoid the $\eps^{-1}$. On the first-order side, the \emph{classical} Frank--Wolfe
(Wolfe--Atwood) algorithms converge linearly: Ahipa\c{s}ao\u{g}lu--Sun--Todd~\cite{ahipasaoglu2008}
proved asymptotic linear convergence of the away-step variant, and Zhao~\cite{zhao2023} made it global,
with a rate governed by a facial condition number. At the other extreme, second-order interior-point
methods achieve high accuracy outright: in our notation ($n$ constraints, dimension $d$),
Nesterov--Nemirovskii~\cite{nesterov1994} compute a $(1+\eps)$-John ellipsoid in
$O\big(n^{2.5}(d^2+n)\log(n/\eps)\big)$ time, and Nemirovski~\cite{nemirovski1999} and
Anstreicher~\cite{anstreicher2002} in $O\big(n^{3.5}\log(1/\eps)\big)$ time~\cite{cohen2019,todd2016}.
These have the desired $\log(1/\eps)$ dependence, but at an $n^{3.5}$ cost far worse than the
near-linear-in-$nd$ first-order methods, and they solve a Newton system per step rather than calling
the cheap leverage-score oracle. So neither linear convergence nor $\log(1/\eps)$ accuracy is new in
itself---indeed away-step Frank--Wolfe already beats $\eps^{-1}$ \emph{within} the
first-order/leverage-score model. The puzzle runs the other way: \emph{why does the leverage-score
fixed-point line (CCLY and its input-sparsity and streaming successors, the very methods that drive the
per-iteration cost down to near-linear) remain stuck at $\eps^{-1}$, when the slower-per-step away-step
methods do not?} We resolve this, and then push the other way: we determine how cheap the accuracy
dependence itself can be made in this oracle. The answer is: \emph{almost free}---doubly logarithmic in
$1/\eps$, with all condition numbers exiled either into an $\eps$-independent setup term or inside the
double logarithm.

\paragraph{Our answer: three costs, one dictionary.}
The $\eps^{-1}$ buys a convenient \emph{certificate}, and it is not necessary. Two structural facts
(\Cref{prop:dict}) form the dictionary that makes this visible. First, in the D-optimal form
$f(\mathbf{p})=-\log\det \mathbf{M}(\mathbf{p})$ with $\mathbf{M}(\mathbf{p})=\sum_i p_i\mathbf{a}_i\mathbf{a}_i^\top$, the gradient is $\nabla f(\mathbf{p})_i=-v_i(\mathbf{p})$: the
leverage-score oracle \emph{is} the first-order oracle. Second, the $(1+\eps)$-John guarantee
$\max_i v_i(\mathbf{p})\le(1+\eps)d$ is exactly the Frank--Wolfe duality gap $g(\mathbf{p})=\max_i v_i(\mathbf{p})-d$ of $f$ over
the simplex (and, by Kiefer--Wolfowitz, the G-optimal-design criterion). Through this dictionary the cost
of a leverage-score algorithm separates into three pieces that the modern line runs together:
\[
\underbrace{\text{certification}}_{\text{reading off the gap}}
\;+\;
\underbrace{\text{identification}}_{\text{reaching the optimal face}}
\;+\;
\underbrace{\text{accuracy}}_{\text{driving the gap to }\eps d}\,,
\]
and the historical $\eps^{-1}$ lives entirely in the first. With the dictionary, the CCLY fixed point is a
multiplicative-weights iteration ($p_i\propto p_iv_i(\mathbf{p})/d$), in fact the classical
\emph{multiplicative algorithm} of the design literature (\Cref{rem:titterington}), and its analysis
certifies the gap of the \emph{uniform running average} of the iterates via Jensen's inequality, which is
$O(1/T)$. We show this certificate, not the problem or the oracle, is what costs $\eps^{-1}$: the uniform
average of these iterates is intrinsically $\Theta(1/T)$, while the last iterate is geometric and, once
the optimal face is identified, even \emph{quadratically} convergent at no change of oracle. The accuracy
cost thus has three regimes: uniform averaging spends $1/\eps$, an accelerated first-order method
$\log(1/\eps)$, and a facial self-concordant phase $\log\log(1/\eps)$. What remains is
\emph{identification} (reaching the right face), which is $\eps$-independent but does depend on the
instance's conditioning, and is exactly where we locate the open problem.

\subsection{Our results}
\label{sec:results}

Throughout, $\Delta_n=\{\mathbf{p}\in\R^n_{\ge0}:\sum_ip_i=1\}$, $v_i(\mathbf{p})=\mathbf{a}_i^\top \mathbf{M}(\mathbf{p})^{-1}\mathbf{a}_i$, and a
\emph{$(1+\eps)$-John ellipsoid} is a $\mathbf{p}\in\Delta_n$ with $\max_iv_i(\mathbf{p})\le(1+\eps)d$ (\Cref{def:je}).
Our first result locates the $\eps^{-1}$ precisely in the averaging step.

\begin{theorem}[Averaging barrier; informal, see \Cref{thm:avg}]
\label{thm:avg-intro}
On an explicit, perfectly conditioned $4\times2$ instance $\mathbf{A}_\opt$, the \emph{uniform} running
average of the CCLY leverage-score iterates has Frank--Wolfe gap exactly $\Theta(1/T)$. Hence any
algorithm that certifies the $(1+\eps)$-John guarantee from this average needs $\Omega(1/\eps)$
iterations: the $\eps^{-1}$ is a property of the averaged certificate, not of the problem or the oracle.
\end{theorem}

Suffix or weighted averaging, or periodic restart, escape this barrier and recover the geometric
last-iterate rate; the obstruction is specific to averaging uniformly over \emph{all} iterates.

Our second result shows the same oracle, used without averaging, breaks the barrier, and sharpens the
rate. Away-step Frank--Wolfe already attains $\log(1/\eps)$ in this model, but \emph{unaccelerated} (rate
$\kappa_\Phi$~\cite{zhao2023}); we give the first \emph{accelerated} leverage-score method, with the
$\sqrt\kappa$ dependence characteristic of optimal first-order optimization. Here and throughout,
$\kappa=\Lopt/\muopt$ is the \emph{facial condition number at the optimum} (\Cref{def:kappa}).

\begin{theorem}[Accelerated algorithm; informal, see \Cref{thm:upper}]
\label{thm:upper-intro}
On any instance with a nondegenerate optimal design, a warm-started \emph{accelerated} leverage-score
method computes a $(1+\eps)$-John ellipsoid in $C(\mathbf{A})+O\big(\sqrt{\kappa}\,\log(1/\eps)\big)$ queries,
where $C(\mathbf{A})$ is the $\eps$-independent cost of reaching the optimal face and $\kappa$ is the facial
condition number there. Since $\kappa=O(d^4)\,\kappa_\Phi$, this is never worse than the unaccelerated
away-step rate $\kappa_\Phi\log(1/\eps)$ of~\cite{zhao2023} by more than a $\poly(d)$ factor, and is
quadratically better when $\kappa\asymp\kappa_\Phi$.
\end{theorem}

Our third result is the new headline: once the optimal face is identified, the \emph{accuracy phase} of
the computation is condition-free, and quadratically convergent---in the \emph{same} oracle. The key
structural fact is that the restriction of $f$ to the affine hull of the optimal face is a closed
self-concordant function whose \emph{unconstrained} minimizer over that affine hull is the optimal
design (\Cref{lem:min}): under strict complementarity, the simplex constraints are invisible to the
facial problem, so the classical, affine-invariant Newton theory applies with no condition numbers in
its rate. A single rank-one update identity (\Cref{prop:smrecover}) makes each Newton step
implementable with $O(d^2)$ leverage-score queries: the facial Hessian
$\big((\mathbf{a}_i^\top\mathbf{M}(\mathbf{p})^{-1}\mathbf{a}_j)^2\big)_{i,j}$ is recovered \emph{exactly}, not approximately, from
leverage queries at $|S^\opt|$ perturbed designs.

\begin{theorem}[Facial Newton phase; informal, see \Cref{thm:newton}]
\label{thm:newton-intro}
On any instance with a nondegenerate optimal design, once the optimal face is identified the accuracy
phase is condition-free: damped Newton reaches the $(1+\eps)$-John guarantee in
$O\big(\log\log(1/\eps)\big)$ steps, with \emph{no} condition number outside the double logarithm.
Each step costs $O(d^2)$ leverage-score queries, since a rank-one identity (\Cref{prop:smrecover})
recovers the facial Hessian \emph{exactly} from the oracle; the total is
$C(\mathbf{A})+O\big(d^2\log\log(1/\eps)\big)$ queries.
\end{theorem}

\begin{corollary}[The three rates; \Cref{cor:total,cor:sep}]
\label{cor:rates-intro}
On nondegenerate instances, a $(1+\eps)$-John ellipsoid is computable in
$C(\mathbf{A})+O\big(d^2\log\log(1/\eps)\big)$ leverage-score queries. On the explicit instance
$\mathbf{A}_\opt$ (for which $\kappa=O(1)$ and $C(\mathbf{A}_\opt)=O(1)$), the three certification styles are
separated: uniform averaging needs $\Theta(\eps^{-1})$ queries, the accelerated last iterate
$O(\log(1/\eps))$, and the facial Newton phase $O(\log\log(1/\eps))$ iterations.
\end{corollary}

\begin{table}[t]
\centering
\renewcommand{\arraystretch}{1.12}
\small
\resizebox{\textwidth}{!}{%
\begin{tabular}{@{}lllll@{}}
\toprule
Algorithm & \#Leverage-score queries & Per-query cost & Certifies via & Assump./Notes\\
\midrule
\cite{khachiyan1996} & $O(d/\eps)$ & $O(nd^{\omega-1})$ & FW gap & ---\\
\cite{kumar2005,todd2007} & $O(d/\eps{+}d\log d)$ & $O(nd^{\omega-1})$ & FW gap & ---\\
\cite{ahipasaoglu2008} & linear & $O(nd^{\omega-1})$ & last iterate & local$^\ddagger$\\
\cite{nesterov1994,anstreicher2002} & $O(\sqrt n\log(1/\eps))$ & $\widetilde{O}(n^{3})$ (Newton) & high accuracy & Newton steps$^\dagger$\\
\cite{cohen2019} & $O(\eps^{-1}\log(n/d))$ & $O(nd^{\omega-1})$ & averaged iterate & ---\\
\cite{cao2025} & $O(\eps^{-1}\log(n/d))$ & $\widetilde{O}(\eps^{-1}nd)$ (sketched) & averaged iterate & ---\\
\cite{woodruff2025} & $O(\eps^{-1}\log(n/d))$ & $\widetilde{O}(nd)$ (amortized) & averaged iterate & ---\\
\cite{zhao2023} & $O(\kappa_\Phi\log(1/\eps))$ & $O(nd^{\omega-1})$ & last iterate (FW gap) & ---\\
\rowcolor{OurRowColor}\textbf{This work} (accelerated phase) & $C(\mathbf{A}){+}O(\sqrt\kappa\log(\Theta/\eps))$ & $O(nd^{\omega-1})$ & last iterate (FW gap) & nondeg.\ optimum$^{*}$\\
\rowcolor{OurRowColor}\textbf{This work} (Newton phase) & $C(\mathbf{A}){+}O(d^2\log\log(\Theta/\eps))$ & $O(nd^{\omega-1})$ & last iterate (FW gap) & nondeg.\ optimum$^{*}$\\
\bottomrule
\end{tabular}}
\caption{Complexity of a $(1+\eps)$-John ellipsoid; $n$ constraints, dimension $d$, dense regime $n\gg d$. We list the \emph{query count} and the \emph{per-query cost} separately; their product is the total running time (the format of~\cite[Table~1]{woodruff2025}). The final column gives each method's instance assumption beyond full rank, with brief notes (``---'' $=$ none). One leverage-score query costs $O(nd^{\omega-1})$ for dense $\mathbf{A}$; the input-sparsity~\cite{cao2025} and lazy-update~\cite{woodruff2025} algorithms cut this per-query cost (to sketched $\widetilde{O}(\eps^{-1}nd)$ and amortized $\widetilde{O}(nd)$, respectively) but keep the same $O(\eps^{-1}\log(n/d))$ query count. That count is $\eps^{-1}$ \emph{exactly} when certifying from the \emph{uniform} average of these fixed-point iterates (\Cref{thm:avg-intro}), however cheap each query is made; $\kappa,\kappa_\Phi$ are condition numbers (\Cref{thm:upper-intro,thm:zhao}) related by $\kappa=O(d^4)\kappa_\Phi$ (\Cref{thm:domination})---both are $\Theta(1/\Phi^2)$ in the facial distance $\Phi$ and blow up as the optimum nears a lower-dimensional face, so acceleration improves the \emph{exponent} ($\kappa_\Phi\!\to\!\sqrt\kappa$), not the instance dependence. $C(\mathbf{A})$ is $\eps$-independent but condition-dependent, and $\Theta$ is an explicit instance polynomial appearing only inside (double) logarithms. Both ``this work'' rows use \emph{exact} leverage scores (the approximate oracles of~\cite{cao2025,woodruff2025} need a separate noise-stability analysis, \Cref{sec:discussion}). $^{*}$Nondegenerate optimal design (\Cref{def:nondeg}), required only by our two rows. $^\dagger$Interior-point methods reach $\log(1/\eps)$ accuracy but lie \emph{outside} the leverage-score/first-order model: the entry counts Newton iterations, not oracle queries (each of the $O(\sqrt n)$ steps solves an $\widetilde{O}(n^{3})$ system, total $O(n^{3.5}\log(1/\eps))$~\cite{anstreicher2002,nesterov1994}). Our Newton phase instead consumes only leverage-score queries, in the $\le\binom{d+1}{2}$ identified contact coordinates. $^\ddagger$``Linear'' convergence here is \emph{local}: the geometric rate holds only near the optimum, and~\cite{ahipasaoglu2008} bound neither the rate constant nor the steps to reach that regime, so there is no global count; the \emph{global} form is Zhao's $O(\kappa_\Phi\log(1/\eps))$.}
\label{tab:results}
\end{table}

\paragraph{Discussion of the results.}
\Cref{thm:avg-intro,thm:upper-intro,thm:newton-intro} are honest about what is and is not new. ``The
last iterate is fast'' for D-optimal design is known since~\cite{ahipasaoglu2008}, and \cite{zhao2023}
already gives a global linear rate for the Frank--Wolfe gap---which, by \Cref{prop:dict}\ref{it:fw},
\emph{is} the G-optimal/John guarantee. Newton methods for the design problem are also classical as
\emph{algorithms}~\cite{todd2016}; what we are not aware of in prior work is a \emph{query-complexity}
statement of the present form: a $\log\log(1/\eps)$ accuracy phase, driven by leverage-score queries
alone, with all condition dependence confined to the setup term and the inner double logarithm. The
conceptual contributions are: (i) the dictionary of \Cref{prop:dict}; (ii) the explicit $\Theta(1/T)$
averaging lower bound of \Cref{thm:avg-intro}, which pinpoints the certification rule (not the
per-iteration cost) as the source of the historical $\eps^{-1}$---so the input-sparsity and streaming
speedups~\cite{cao2025,woodruff2025} leave an \emph{avoidable} $\eps^{-1}$ on the table; (iii) the
acceleration $\kappa_\Phi\to\sqrt\kappa$ with the uniform comparison $\kappa=O(d^4)\kappa_\Phi$; and (iv) the
observation that under strict complementarity the facial problem is an \emph{unconstrained}
self-concordant minimization whose Hessian the leverage-score oracle can reconstruct exactly
(\Cref{lem:min,prop:smrecover}), which converts the classical affine-invariant Newton theory into a
condition-free accuracy phase in the original oracle. What remains open is exactly the
\emph{setup}: a condition-free bound on identifying the optimal face, or a first-order lower bound
(\Cref{sec:discussion}).

\subsection{Technical overview}
\label{sec:overview}

We sketch the four components: the dictionary, the averaging lower bound, the facial geometry that
powers both upper bounds, and the two facial phases.

\paragraph{The dictionary (\Cref{sec:prelim}).}
Write the John-ellipsoid program as $\min_{\mathbf{p}\in\Delta_n}f(\mathbf{p})$, $f(\mathbf{p})=-\log\det \mathbf{M}(\mathbf{p})$. Differentiating
$\log\det$ gives $\nabla f(\mathbf{p})_i=-\mathbf{a}_i^\top \mathbf{M}(\mathbf{p})^{-1}\mathbf{a}_i=-v_i(\mathbf{p})$, so a single leverage-score query
returns the full gradient. Over the simplex, the Frank--Wolfe gap
$g(\mathbf{p})=\max_{\mathbf{q}\in\Delta_n}\inner{-\nabla f(\mathbf{p}),\mathbf{q}-\mathbf{p}}$ evaluates to $\max_iv_i(\mathbf{p})-d$ (using the identity
$\sum_ip_iv_i(\mathbf{p})=d$), which is exactly the constraint violation in the $(1+\eps)$-John guarantee and,
by the Kiefer--Wolfowitz equivalence (\Cref{thm:kw}), the G-optimal-design criterion. Finally, the
Hessian is the entrywise square $\nabla^2f(\mathbf{p})_{ij}=(\mathbf{a}_i^\top \mathbf{M}(\mathbf{p})^{-1}\mathbf{a}_j)^2$, the Gram matrix of the
rank-one symmetric tensors $\{\mathbf{M}(\mathbf{p})^{-1/2}\mathbf{a}_i\}^{\otimes2}$; it is positive semidefinite of rank at most
$\binom{d+1}{2}$. These four facts let us reason about the standard algorithms as first-order methods
and about the John guarantee as a duality gap.

\paragraph{Averaging is $\Theta(1/T)$ (\Cref{sec:avg}).}
Why is the whole leverage-score line stuck at $\eps^{-1}$? Because it certifies. The CCLY analysis
bounds the gap of the running average $\bar{\mathbf{p}}^{(T)}$ by Jensen's inequality applied to the convex maps
$\mathbf{p}\mapsto\log v_i(\mathbf{p})$, giving $g(\bar{\mathbf{p}}^{(T)})=O(1/T)$. We show this is tight, and for an elementary
reason: averaging destroys a linear rate. If the iterates converge linearly to $\mathbf{p}^\opt$, then
$\sum_{t\ge1}(\mathbf{p}^{(t)}-\mathbf{p}^\opt)$ converges to some $\mathbf{C}$, and the uniform average satisfies
$\bar{\mathbf{p}}^{(T)}-\mathbf{p}^\opt=\frac1T\mathbf{C}+o(1/T)=\Theta(1/T)$ when $\mathbf{C}\neq0$ (\Cref{lem:avg})---the average
lags its own tail by a full factor of $T$.
To convert iterate error into gap, we linearize: at a nondegenerate optimum the gap is, to first order,
$g(\mathbf{p}^\opt+\boldsymbol{\delta})=\max_{i\in S^\opt}\inner{\nabla v_i(\mathbf{p}^\opt),\boldsymbol{\delta}}+O(\norm{\boldsymbol{\delta}}^2)$ over the
feasible cone (\Cref{lem:gaplin}). On the explicit instance $\mathbf{A}_\opt$ of \Cref{sec:instance} the
linearized gap $F$ is \emph{provably} positive on the feasible cone (a four-line sign
argument, \Cref{lem:Fpos}, with no numerics), and the iterates converge R-linearly (the fixed-point map
is the classical multiplicative algorithm, whose monotone convergence is
known~\cite{titterington1976,yu2010}, and whose Jacobian spectral radius at the optimum we compute in
closed form). Evaluating at $\boldsymbol{\delta}=\bar{\mathbf{p}}^{(T)}-\mathbf{p}^\opt=\Theta(1/T)$ gives
$g(\bar{\mathbf{p}}^{(T)})=\Theta(1/T)$: the $\eps^{-1}$ is the price of the average, not of the oracle.

\paragraph{The facial geometry, quantitatively (\Cref{sec:face}).}
Both upper bounds run on the optimal face $F^\opt$ (the constraints with $v_i(\mathbf{p}^\opt)=d$), and both need
the same three facts. (1) $f$ is \emph{not} strongly convex on $\Delta_n$ (its Hessian has rank at most
$\binom{d+1}{2}$), but the rank-deficiency is an exact, harmless overparametrization: $f$, the leverage
scores, and the gap are constant along $\ker\nabla^2f$, which is independent of $\mathbf{p}$ (\Cref{lem:flat}).
(2) On the optimal face, nondegeneracy makes the reduced Hessian positive definite at \emph{every}
point, not merely at $\mathbf{p}^\opt$ (\Cref{lem:faceSC}). (3) Most importantly, $-\log\det$ is
self-concordant, and we exploit this \emph{quantitatively} rather than through compactness: an
elementary sandwich (\Cref{lem:Msand,lem:hesssand}; the proofs are Frobenius-vs-spectral norm
arithmetic plus one integration) shows that within the Dikin ball of radius $\tfrac12$ around
$\mathbf{p}^\opt$, the information matrix, every leverage score, and the full facial Hessian block are all
within absolute constant factors of their values at $\mathbf{p}^\opt$. Combined with the lower bound
$f\ge f^\opt+\gsc\cdot(\text{inactive mass})+\wsc(\text{Dikin radius})$ (\Cref{lem:lower}), a single
explicit constant $c_0$ (\Cref{def:c0}) guarantees that the whole sublevel set
$\{f\le f^\opt+c_0\}$ on the face's affine hull sits strictly inside the simplex, keeps all non-contact
leverage scores below $d-\gsc/2$, and carries uniform smoothness/strong-convexity constants
$4\Lopt$ and $\tfrac49\muopt$. Every compactness argument in our earlier reasoning is thereby replaced
by explicit constants, which is also what makes the domination bound $\kappa=O(d^4)\kappa_\Phi$
(\Cref{thm:domination}) clean.

\paragraph{Identification, then the accuracy phases (\Cref{sec:upper,sec:newton}).}
This is the identification cost and the two accuracy phases. To \emph{reach} the face we use the classical
away-step Frank--Wolfe: it drives the gap below $c_0$ in
$O(d^2/c_0)$ $\eps$-independent steps (\Cref{thm:wa}) and identifies the active set under strict
complementarity (\Cref{thm:bomze}); a detectable stopping rule (\Cref{lem:identify}) certifies the
handoff. This is the whole of $C(\mathbf{A})$: $\eps$-independent, but condition-dependent. On the face, we give
two interchangeable accuracy phases, both faster in $\eps$ than averaging. The \emph{accelerated} phase runs restarted
FISTA on the face simplex: the localization above supplies the smoothness/strong-convexity constants,
the standard FISTA potential function doubles as a proof that the trajectory never leaves the region
where those constants are valid (\Cref{app:fista}), and restarting converts the $O(1/k^2)$ rate into
$(\sqrt\kappa\log\tfrac1\eps)$-type geometric decay. The \emph{Newton} phase is structurally simpler
and faster in $\eps$: under strict complementarity the optimal design is the \emph{unconstrained}
minimizer of the closed self-concordant function $h=f|_{\aff F^\opt}$ (\Cref{lem:min}), so damped
Newton enjoys the classical two-phase behavior: at most three damped steps (the budget $c_0\le
\wsc(\tfrac12)$ funds no more, each costing $\wsc(\tfrac14)$), then quadratic convergence
$2\lambda_{j+1}\le(2\lambda_j)^2$ of the Newton decrement, whose constants
are \emph{absolute} by affine invariance. The John guarantee itself is the stopping rule (it is the
Frank--Wolfe gap, read off the same leverage-score query that supplies the gradient). The only
remaining question is how to get the facial Hessian from the oracle, and here a one-line
Sherman--Morrison identity gives an exact answer:
\[
\big(\mathbf{a}_i^\top\mathbf{M}(\mathbf{p})^{-1}\mathbf{a}_j\big)^2
=\big(v_i(\mathbf{p})-v_i(\mathbf{p}+t\mathbf{e}_j)\big)\cdot\frac{1+t\,v_j(\mathbf{p})}{t}\qquad(t>0),
\]
so $|S^\opt|+1\le\binom{d+1}{2}+1$ leverage queries (at genuine designs, after renormalization) deliver
the exact Newton system (\Cref{prop:smrecover}). No second-order oracle is ever consulted.

\subsection{Related work}
\label{sec:related}

\paragraph{John ellipsoid / D-optimal design.} The equivalence of the maximum-volume inscribed ellipsoid
of a symmetric polytope, the minimum-volume enclosing ellipsoid of the rows, D-optimal design, and
$\ell_\infty$ Lewis weights is classical~\cite{todd2016,pukelsheim2006}. Algorithmically, the
Frank--Wolfe / Wolfe--Atwood family~\cite{khachiyan1996,kumar2005,todd2007,ahipasaoglu2008,zhao2023}
and the leverage-score fixed point~\cite{cohen2019} and its
descendants~\cite{cao2025,woodruff2025,song2024dp,song2024quantum} are the two main lines; see
\Cref{tab:results}. The fixed-point map itself predates this line by four decades: it is the
multiplicative algorithm of Titterington~\cite{titterington1976}, whose monotone convergence for the
D-criterion is classical~\cite{titterington1976,yu2010}; the modern contribution of the CCLY line is
the near-linear per-iteration cost and the averaged certificate, and the latter is exactly what we
show costs $\eps^{-1}$.

\paragraph{Second-order / interior-point methods.} A separate, older line attains high accuracy via
Newton-type interior-point methods. In our notation,
Nesterov--Nemirovskii~\cite{nesterov1994} compute a $(1+\eps)$-John ellipsoid in
$O\big(n^{2.5}(d^2+n)\log(n/\eps)\big)$ time; Khachiyan--Todd~\cite{khachiyantodd1993} and then
Nemirovski~\cite{nemirovski1999} and Anstreicher~\cite{anstreicher2002} reduce this to
$O\big(n^{3.5}\log(1/\eps)\big)$ (we follow the survey of these bounds
in~\cite[\S1.2]{cohen2019},~\cite{todd2016}). These already enjoy the $\log(1/\eps)$ dependence, but
each step solves a Newton system in all $n$ weight variables, and they do not fit the leverage-score
oracle. Local Newton-type refinements for the design problem are also discussed
in the optimization literature (see~\cite[Ch.~3]{todd2016} and references); our point is the
\emph{query-complexity} form of the statement: after identification, the Newton system in the
$\le\binom{d+1}{2}$ \emph{contact} coordinates is exactly reconstructible from $O(d^2)$ leverage-score
queries (\Cref{prop:smrecover}), so quadratic convergence costs nothing beyond the original oracle.
Rank-one update formulas of Sherman--Morrison type are themselves standard tools in this
literature~\cite{todd2016}; the use we make of them, as an exact \emph{oracle simulation} of the
facial Hessian, appears to be new.

\paragraph{Acceleration and its limits.} Lu--Freund--Nesterov~\cite{lu2018} place
D-optimal design in the \emph{relatively-smooth} framework (smoothness measured against a reference
convex function---here the Burg entropy $-\sum_i\log(\cdot)$, the scalar analogue of $-\log\det$---rather
than in Euclidean norm), obtaining the $O(1/\eps)$ rate of the corresponding Bregman-gradient
(``NoLips'') method; Dragomir, Taylor, d'Aspremont and
Bolte~\cite{dragomir2021} prove this rate is \emph{optimal} for \emph{Bregman} first-order methods
(mirror-descent-type methods built from that reference function), and
Hanzely--Richt\'arik--Xiao~\cite{hanzely2021} show that the relevant \emph{triangle-scaling exponent}
(which quantifies how much such a method can be accelerated) precludes a uniform Bregman speedup. These
results govern the
\emph{Bregman/mirror} oracle; our acceleration is \emph{Euclidean} projected gradient on the face and is
not constrained by them. The empirical $O(1/k^2)$ of Gutman--Pe\~na~\cite{gutman2018} is itself a
Bregman observation, consistent with~\cite{dragomir2021} only through the triangle-scaling loophole,
and is distinct from the phenomenon here.

\paragraph{Averaging vs.\ last iterate.} That averaging degrades
the last-iterate rate of (strongly) convex optimization is well known in general; \Cref{thm:avg-intro}
is, to our knowledge, the first statement that this is the specific mechanism keeping the
leverage-score line at $\eps^{-1}$.

\section{Preliminaries}
\label{sec:prelim}

\paragraph{Notation.} $[n]=\{1,\dots,n\}$; $\Sym^d$ is the space of $d\times d$ symmetric matrices,
of dimension $\binom{d+1}{2}$, with trace inner product $\inner{\mathbf{X},\mathbf{Y}}=\Tr(\mathbf{X}\mathbf{Y})$ and Frobenius
norm $\fnorm{\cdot}$; $\Sym^d_{\succ0}$ is its positive-definite cone; for $\mathbf{a}\in\R^d$,
$\mathbf{a}^{\otimes2}=\mathbf{a}\mathbf{a}^\top\in\Sym^d$. $\Delta_n=\{\mathbf{p}\ge0:\sum_ip_i=1\}$ and $\relint\Delta_n$ its relative
interior; for $S\subseteq[n]$, $\Delta_S:=\{\mathbf{p}\in\Delta_n:\supp(\mathbf{p})\subseteq S\}$.
$\omega<2.372$~\cite{alman2025,vvw2024,almanwilliams2021} is the matrix-multiplication exponent, and $\nnz(\mathbf{A})$ the number of nonzero
entries of $\mathbf{A}$. For $\mathbf{p}\in\R^n$ we write
$\mathbf{M}(\mathbf{p})=\sum_ip_i\mathbf{a}_i\mathbf{a}_i^\top\in\Sym^d$, and whenever $\mathbf{M}(\mathbf{p})\succ0$ also
$v_i(\mathbf{p})=\mathbf{a}_i^\top \mathbf{M}(\mathbf{p})^{-1}\mathbf{a}_i$ and $f(\mathbf{p})=-\log\det \mathbf{M}(\mathbf{p})$. Two scalar functions recur:
\[
\wsc(t):=t-\log(1+t)\quad(t\ge0),\qquad \wscs(t):=-t-\log(1-t)\quad(t\in[0,1));
\]
both are increasing, convex, and vanish to second order at $0$; they are Legendre conjugates of one
another (\Cref{lem:legendre}).

\begin{definition}[$(1+\eps)$-John ellipsoid]
\label{def:je}
A weight $\mathbf{p}\in\Delta_n$ is a \emph{$(1+\eps)$-John ellipsoid} (equivalently a $(1+\eps)$-approximate
G-optimal design) if $\max_{i\in[n]}v_i(\mathbf{p})\le(1+\eps)d$. The induced ellipsoid
$Q=\{\mathbf{x}:\mathbf{x}^\top(d\,\mathbf{M}(\mathbf{p}))\,\mathbf{x}\le1\}$ then satisfies $\tfrac1{\sqrt{1+\eps}}Q\subseteq P\subseteq\sqrt
d\,Q$~\cite{cohen2019,todd2016}.
\end{definition}

\paragraph{The oracle / cost model.}
A \emph{leverage-score query} submits a weight vector $\mathbf{w}$ with $\mathbf{M}(\mathbf{w})\succ0$ and receives the
vector $\big(v_i(\mathbf{w})\big)_{i\in[n]}=\big(\mathbf{a}_i^\top\mathbf{M}(\mathbf{w})^{-1}\mathbf{a}_i\big)_{i\in[n]}$. For dense $\mathbf{A}$
one query costs $O(nd^{\omega-1})$ arithmetic, and this dominates the per-iteration cost of every
algorithm in \Cref{tab:results}; the input-sparsity and lazy-update lines~\cite{cao2025,woodruff2025}
reduce exactly this per-query cost. We therefore measure algorithms primarily in the number of
leverage-score queries, and account separately for any additional arithmetic. Leverage scores are
$(-1)$-homogeneous, $v_i(c\mathbf{w})=v_i(\mathbf{w})/c$ for $c>0$, so querying an unnormalized nonnegative weight
vector is equivalent to querying its normalization in $\Delta_n$; all queries made by our algorithms
are at nonnegative weights, except that the accelerated phase of \Cref{sec:upper} may evaluate at
extrapolated weights with small negative entries on the identified face (still satisfying
$\mathbf{M}(\mathbf{w})\succ0$, hence well defined at the same arithmetic cost); the Newton phase of
\Cref{sec:newton} queries genuine designs in $\Delta_n$ only.

\subsection{Leverage scores as a first-order oracle}
\label{sec:dict}
Everything that follows is read through one dictionary, built on the map
$\mathbf{p}\mapsto\mathbf{M}(\mathbf{p})=\sum_ip_i\mathbf{a}_i\mathbf{a}_i^\top$ and the objective $f(\mathbf{p})=-\log\det\mathbf{M}(\mathbf{p})$: the
leverage-score oracle is the \emph{gradient}, the $(1+\eps)$-John guarantee is a
\emph{duality gap}, and the second-order geometry is the \emph{Gram matrix of rank-one tensors}. This is
not bookkeeping: it is the dictionary in which the three costs of \Cref{sec:intro}
(certification, identification, accuracy) come apart, and we use it without further comment throughout.
The individual identities are standard~\cite{todd2016,cohen2019}; we restate them in our notation for
self-containedness.

\begin{proposition}[Dictionary]
\label{prop:dict}
For $\mathbf{p}$ with $\mathbf{p}\in\Delta_n$, $\mathbf{M}(\mathbf{p})\succ0$ (in particular for $\mathbf{p}\in\relint\Delta_n$):
\begin{enumerate}[label=\emph{(\alph*)},leftmargin=*,itemsep=1pt]
\item\label{it:grad} $\nabla f(\mathbf{p})_i=-v_i(\mathbf{p})$ \emph{(the leverage-score oracle is the first-order oracle).}
\item\label{it:trace} $\sum_i p_iv_i(\mathbf{p})=d$; hence $\max_iv_i(\mathbf{p})\ge d$, with equality iff $\mathbf{p}$ is optimal.
\item\label{it:fw} The Frank--Wolfe gap $g(\mathbf{p}):=\max_{\mathbf{q}\in\Delta_n}\inner{-\nabla f(\mathbf{p}),\mathbf{q}-\mathbf{p}}$ equals
$\max_iv_i(\mathbf{p})-d$, satisfies $g(\mathbf{p})\ge f(\mathbf{p})-f^\opt\ge0$, and $f(\mathbf{p})-f^\opt\le d\log(1+g(\mathbf{p})/d)\le g(\mathbf{p})$.
Thus \Cref{def:je} reads $g(\mathbf{p})\le\eps d$.
\item\label{it:hess} $\nabla^2f(\mathbf{p})_{ij}=(\mathbf{a}_i^\top \mathbf{M}(\mathbf{p})^{-1}\mathbf{a}_j)^2$, the Gram matrix of
$\{\mathbf{M}(\mathbf{p})^{-1/2}\mathbf{a}_i\}^{\otimes2}\subset\Sym^d$; it is positive semidefinite of rank at most
$\binom{d+1}{2}$, with kernel $\{\boldsymbol{\delta}:\sum_i\delta_i \mathbf{a}_i\mathbf{a}_i^\top=0\}$ independent of $\mathbf{p}$. More
precisely, for every $\boldsymbol{\delta}\in\R^n$,
\begin{equation}
\label{eq:hessform}
\boldsymbol{\delta}^\top\nabla^2f(\mathbf{p})\,\boldsymbol{\delta}
=\Big\lVert \mathbf{M}(\mathbf{p})^{-1/2}\Big(\textstyle\sum_i\delta_i\mathbf{a}_i\mathbf{a}_i^\top\Big)\mathbf{M}(\mathbf{p})^{-1/2}\Big\rVert_F^2 .
\end{equation}
\end{enumerate}
\end{proposition}
\noindent The proof is an elementary (if lengthy) matrix-calculus computation; we defer it to
\Cref{app:dict}. Identity \eqref{eq:hessform} expresses the Hessian quadratic form as the squared
Frobenius norm of a matrix \emph{measured in the geometry of $\mathbf{M}(\mathbf{p})$}; this is the self-concordance
of $-\log\det$ in disguise, and \Cref{sec:face} develops it quantitatively.

\subsection{Designs, nondegeneracy, and useful results}
\label{sec:designs}

\paragraph{Design terminology and nondegeneracy.}
Maximizing $\log\det \mathbf{M}(\mathbf{p})$ (equivalently, minimizing $f$) over $\Delta_n$ is \emph{D-optimal design};
minimizing the largest leverage score $\max_iv_i(\mathbf{p})$ is \emph{G-optimal design}; \Cref{thm:kw}
identifies these two with each other and with the John ellipsoid. The quantity
$v_i(\mathbf{p})=\mathbf{a}_i^\top \mathbf{M}(\mathbf{p})^{-1}\mathbf{a}_i$ is the (statistical) \emph{leverage score} of row $i$ under $\mathbf{p}$. Fix an
optimal design $\mathbf{p}^\opt$ and write $\mathbf{M}^\opt:=\mathbf{M}(\mathbf{p}^\opt)$ (which is the same for every optimal design,
by strict concavity of $\log\det$; \Cref{lem:unique}). The \emph{contact set}
$S^\opt:=\{i:v_i(\mathbf{p}^\opt)=d\}$ indexes the rows at which
the John ellipsoid touches $\partial P$ (the \emph{contact points}); the \emph{optimal face}
$F^\opt:=\Delta_{S^\opt}$ is the face of the simplex on which the optimum is
supported; and \emph{strict complementarity} is the (generic) condition $\supp(\mathbf{p}^\opt)=S^\opt$, with
slack $\gsc:=\min_{i\notin S^\opt}(d-v_i(\mathbf{p}^\opt))>0$.

\paragraph{The design polytope and its strata.}
The simplex $\Delta_n$ is a convenient parametrization, not the geometry of the problem. The map
$\mathbf{p}\mapsto\mathbf{M}(\mathbf{p})$ pushes $\Delta_n$ onto the \emph{design polytope}
$\mathcal K:=\mathrm{conv}\{\mathbf{a}_i\mathbf{a}_i^\top:i\in[n]\}\subset\Sym^d$, and it is the facial structure of
$\mathcal K$ that organizes what follows. The proper faces of $\mathcal K$ are the natural \emph{strata}
of the problem, and the contact set $S^\opt$ names the stratum on which the optimum lives: under strict
complementarity $\mathbf{M}^\opt=\sum_{i\in S^\opt}p^\opt_i\,\mathbf{a}_i\mathbf{a}_i^\top$ lies in the relative interior of the
exposed face $\mathrm{conv}\{\mathbf{a}_i\mathbf{a}_i^\top:i\in S^\opt\}$ (exposed by
$\mathbf{X}\mapsto\inner{\mathbf{X},(\mathbf{M}^\opt)^{-1}}$, since
$v_i(\mathbf{p}^\opt)=\inner{\mathbf{a}_i\mathbf{a}_i^\top,(\mathbf{M}^\opt)^{-1}}\le d$ with equality iff $i\in S^\opt$). The
nondegeneracy assumption below then says exactly that this stratum is \emph{smooth}: the optimum sits in
its relative interior, and the contact tensors are independent, so the stratum is a simplex on which the
pulled-back Hessian is nondegenerate modulo the harmless coordinate kernel of \Cref{lem:flat}. The same
facial geometry returns quantitatively as the facial distance $\Phi$ of $\mathcal K$ (\Cref{thm:zhao}),
the quantity the identification cost depends on.

\begin{assumption}[Nondegeneracy]
\label{def:nondeg}
The optimal design $\mathbf{p}^\opt$ is \emph{nondegenerate}: (i) strict complementarity holds, i.e.\
$\supp(\mathbf{p}^\opt)=S^\opt$ and $\gsc>0$; and (ii) the contact matrices $\{\mathbf{a}_i\mathbf{a}_i^\top\}_{i\in S^\opt}$
are linearly independent in $\Sym^d$ (which forces $m:=|S^\opt|\le\binom{d+1}{2}$). All upper-bound
results assume this; \Cref{lem:unique} shows it makes $\mathbf{p}^\opt$ the \emph{unique} optimal design. We
write $\pmin:=\min_{i\in S^\opt}p^\opt_i>0$.
\end{assumption}

\begin{remark}[How restrictive is nondegeneracy?]
\label{rem:nondeg}
Both conditions of \Cref{def:nondeg} are \emph{generic}: they hold for Lebesgue-almost-every
$\mathbf{A}\in\R^{n\times d}$, the exceptional set lying in a proper (measure-zero) algebraic subvariety.
Independence of the contact tensors $\{\mathbf{a}_i\mathbf{a}_i^\top\}_{i\in S^\opt}$ is a determinantal condition
on $\mathbf{A}$, and strict complementarity holds generically, as established for semidefinite programs at
large~\cite{alizadeh1997}. What \Cref{def:nondeg} excludes are the \emph{structured} instances (those
with exact symmetries or repeated rows) on which the optimal design may be non-unique or carry more
than $\binom{d+1}{2}$ contact points (linearly dependent contact matrices); these form a measure-zero
set but include several natural designs. A generic perturbation $\mathbf{A}\mapsto\mathbf{A}+\eta\mathbf{G}$ restores
nondegeneracy and moves the John ellipsoid by $O(\eta)$, but does not remove the assumption for free:
the slack $\gsc$ and the facial distance $\Phi$ degrade as $\eta\to0$, so the cost reappears inside the
setup term $C(\mathbf{A})$ (\Cref{rem:CA}) rather than the accuracy phase. Geometrically, these excluded
instances are exactly where the strata of $\mathcal K$ \emph{collide}: we treat that singular locus not as
an exception to be assumed away but as the degenerate stratum of the same object, and quantifying it
directly is \Cref{prob:degenerate}.
\end{remark}

\paragraph{Convex-optimization terminology.}
A differentiable convex function $h$ on a convex set is \emph{$L$-smooth} if $\nabla h$ is
$L$-Lipschitz there and \emph{$\mu$-strongly convex} if $h-\tfrac\mu2\norm{\cdot}_2^2$ is convex.
The \emph{Frank--Wolfe} (conditional-gradient) method minimizes $h$ over a polytope
using only a \emph{linear-minimization oracle} $\mathrm{LMO}(\mathbf{c})=\argmin_{\mathbf{x}}\inner{\mathbf{c},\mathbf{x}}$ over the
polytope, repeatedly stepping \emph{toward} the vertex $\mathrm{LMO}(\nabla h)$; over $\Delta_n$ a
vertex is an atom $\mathbf{e}_j$ (the rank-one term $\mathbf{a}_j\mathbf{a}_j^\top$), and $\mathrm{LMO}(\nabla f)=\argmax_iv_i$ by
\Cref{prop:dict}\ref{it:grad}. The \emph{away-step Frank--Wolfe} (\emph{Wolfe--Atwood}) variant may in
addition step \emph{away} from the worst currently-supported atom $\argmin_{i:p_i>0}v_i$, and perform a
\emph{drop step} removing an atom once its weight reaches $0$; these three move types are exactly the
branches of \Cref{alg:main}, and the away/drop steps are what let the method identify the optimal
support. The \emph{Frank--Wolfe gap} $g(\mathbf{p})$ (\Cref{prop:dict}\ref{it:fw}) is its standard optimality
certificate. The accelerated method we use on the face is \emph{FISTA} (with projection onto the face
simplex as its proximal step), restated as \Cref{thm:fista}; the second-order method is the classical
\emph{damped Newton} method for self-concordant functions, whose imported ingredients are collected in
\Cref{sec:face} and \Cref{app:sctoolkit}.

\paragraph{External results, restated.}
We now restate, in our notation, the external results invoked in the body. (Further imported facts
that are used only inside deferred proofs (the FISTA potential inequality and the local quadratic
convergence of Newton's method for self-concordant functions) are restated where used:
\Cref{thm:fista,thm:newtonquad}.)

\begin{theorem}[Kiefer--Wolfowitz equivalence~\cite{kiefer1960,pukelsheim2006}]
\label{thm:kw}
$g(\mathbf{p})=0$ $\iff$ $\mathbf{p}$ is D-optimal ($f(\mathbf{p})=f^\opt$) $\iff$ $\mathbf{p}$ is G-optimal ($\max_iv_i(\mathbf{p})=d$).
\end{theorem}

\begin{theorem}[CCLY averaged fixed point~{\cite[Thm.~3.2--3.3]{cohen2019}}]
\label{thm:ccly}
The iteration $p^{(t+1)}_i\propto p^{(t)}_iv_i(\mathbf{p}^{(t)})$ from $\mathbf{p}^{(1)}=\tfrac1n\mathbf1$, with averaged
output $\bar{\mathbf{p}}^{(T)}=\tfrac1T\sum_{t\le T}\mathbf{p}^{(t)}$, satisfies $g(\bar{\mathbf{p}}^{(T)})\le\eps d$ for
$T=O(\eps^{-1}\log(n/d))$, and is implementable in $O(\eps^{-1}\,nd^{\omega-1}\log(n/d))$ time.
\end{theorem}

\begin{remark}[The fixed-point map is the classical multiplicative algorithm]
\label{rem:titterington}
On the simplex the normalization is exact---$\sum_ip_iv_i(\mathbf{p})=d$ always
(\Cref{prop:dict}\ref{it:trace}), so the CCLY iteration is precisely the map
$\Phi(\mathbf{p})=\big(p_iv_i(\mathbf{p})/d\big)_{i\in[n]}$, which maps $\Delta_n$ to $\Delta_n$. This is the
\emph{multiplicative algorithm} for D-optimal design introduced by
Titterington~\cite{titterington1976}; its monotonicity ($f(\Phi(\mathbf{p}))\le f(\mathbf{p})$) and convergence of
$f$-values to the optimum, from any positive start, are classical~\cite{titterington1976,yu2010}. We
use this in \Cref{thm:avg} to get R-linear convergence of the iterates on our explicit instance.
\end{remark}

\begin{theorem}[Wolfe--Atwood warm start~{\cite[Thm.~3.1, Alg.~4.1]{todd2007}; \cite{kumar2005}}]
\label{thm:wa}
The away-step Frank--Wolfe (Wolfe--Atwood) method, initialized by the Kumar--Y\i ld\i r\i m volumetric
rule, reaches $g(\mathbf{p})\le\eps' d$ in $O(d/\eps'+d\log d)$ iterations, each one leverage-score query. In
particular, for any constant $c>0$ it reaches $g(\mathbf{p})\le c$ in $O(d^2/c+d\log d)$ iterations,
independent of $\eps$.
\end{theorem}

The \emph{Kumar--Y\i ld\i r\i m volumetric rule} used above produces the initial design as follows: greedily
construct $d$ (near-)orthogonal directions and, along each, add the two rows of $\mathbf{A}$ that are extremal in
that direction; the uniform design $\mathbf{p}^{(0)}$ on these $\le2d$ rows already satisfies
$\mathrm{vol}(\text{its ellipsoid})\ge d^{-2d}\cdot\mathrm{vol}(\text{John ellipsoid})$, equivalently
$f(\mathbf{p}^{(0)})-f^\opt=O(d\log d)$, and is computed in $O(nd^2)$ time~\cite{kumar2005,todd2007}.

\begin{theorem}[Away-step linear convergence~{\cite[Thm.~4.1--4.2]{zhao2023}}; \cite{ahipasaoglu2008}]
\label{thm:zhao}
For D-optimal design, away-step Frank--Wolfe converges globally linearly in both the objective gap and
the Frank--Wolfe gap: $g(\mathbf{p}^{(k)})\le(1-\rho)^{k_{\mathrm{eff}}}\cdot O(g_0)$, where $g_0$ is the initial
gap, $k_{\mathrm{eff}}=\Theta(k)$, and $1/\rho=O(\kappa_\Phi)$. Equivalently, $g\le\eps d$ holds for all
$k\ge K_{\mathrm{gap}}(\eps)=O(\kappa_\Phi\log(1/\eps))$.
\end{theorem}

Here the \emph{facial condition number} is $\kappa_\Phi:=1/(\mu_R\,\Phi^2)$, where $\mu_R\le\tfrac12$ is
a local strong-convexity constant of $f$ and $\Phi=\Phi(\mathbf{A})>0$ is the \emph{facial
distance}~\cite{pena2019} of the design polytope $\mathcal K$ (\Cref{sec:designs})
in the local norm at $\mathbf{M}^\opt$: the minimum, over its proper faces $\mathcal F$, of the distance between
$\mathcal F$ and the convex hull of the vertices not lying on $\mathcal F$ (a quantity shown
in~\cite{pena2019} to coincide with the \emph{pyramidal width} of~\cite{lacoste2015}). With this
normalization, the iteration count in \Cref{thm:zhao} is $O(\kappa_\Phi\log(1/\eps))$. The quantity $\Phi$ is
strictly positive for every fixed instance but tends to $0$ as the optimal design approaches a
lower-dimensional face, which is why \Cref{thm:zhao} is not a uniform bound.

\begin{theorem}[Active-set identification~{\cite{zhao2023}; cf.~\cite[Thm.~3.3]{bomze2020}}]
\label{thm:bomze}
Under strict complementarity, away-step Frank--Wolfe with exact line search (as in \Cref{alg:main})
on $-\log\det$ D-optimal design identifies the optimal face in finitely many iterations: there is a
$K_{\mathrm{id}}<\infty$ (depending on the instance, not on $\eps$) such that for all $k\ge
K_{\mathrm{id}}$ the iterate $\mathbf{p}^{(k)}$ is supported on $S^\opt$.
\end{theorem}

The general active-set theory of away-step Frank--Wolfe is~\cite{bomze2020}; its quantitative bound
presupposes a globally Lipschitz gradient, which $-\log\det$ has only on the facial region, so for the
boundary-singular barrier we take identification from the in-setting analysis of~\cite{zhao2023}---see
\Cref{rem:CA}.

\section{The averaging barrier}
\label{sec:avg}

This is the \emph{certificate cost}. The leverage-score line certifies the John guarantee through
the running average $\bar{\mathbf{p}}^{(T)}$ (\Cref{thm:ccly}); we show this caps the rate at $\Theta(1/T)$
regardless of how cheap each iteration is made, and that the cap is the average's doing, not the
sequence's: \emph{the fixed point is not slow; the certificate is slow.} This section proves
\Cref{thm:avg-intro}. The proof has three independent ingredients, each
stated and proved in full: averaging destroys linear rates (\Cref{lem:avg}); the Frank--Wolfe gap is,
to first order, a positively homogeneous piecewise-linear function $F$ of the displacement from the
optimum (\Cref{lem:gaplin}); and on the explicit instance $\mathbf{A}_\opt$ the function $F$ is
\emph{provably} positive on the feasible cone (\Cref{lem:Fpos}) while the iterates converge R-linearly
with a nonzero summed displacement (\Cref{thm:avg}).

\begin{lemma}[Averaging destroys linear rates]
\label{lem:avg}
Let $\{\mathbf{p}^{(t)}\}_{t\ge1}\subset\Delta_n$ converge to $\mathbf{p}^\opt$ \emph{R-linearly}: there exist
$\rho\in(0,1)$ and $C_0<\infty$ with $\norm{\mathbf{p}^{(t)}-\mathbf{p}^\opt}\le C_0\rho^t$ for all $t$. Then
$\mathbf{C}:=\sum_{t\ge1}(\mathbf{p}^{(t)}-\mathbf{p}^\opt)$ converges absolutely, and the uniform running average
$\bar{\mathbf{p}}^{(T)}=\frac1T\sum_{t=1}^T\mathbf{p}^{(t)}$ satisfies $\bar{\mathbf{p}}^{(T)}-\mathbf{p}^\opt=\frac1T\mathbf{C}+o(1/T)$. In
particular, if $\mathbf{C}\neq0$ then $\norm{\bar{\mathbf{p}}^{(T)}-\mathbf{p}^\opt}=\Theta(1/T)$.
\end{lemma}
\begin{proof}
Absolute convergence is immediate: $\sum_{t\ge1}\norm{\mathbf{p}^{(t)}-\mathbf{p}^\opt}\le C_0\sum_{t\ge1}\rho^t=
\frac{C_0\rho}{1-\rho}<\infty$, so $\mathbf{C}$ is well defined. Write the partial sum as
$\mathbf{S}_T:=\sum_{t=1}^T(\mathbf{p}^{(t)}-\mathbf{p}^\opt)=\mathbf{C}-\mathbf{R}_T$ with tail $\mathbf{R}_T:=\sum_{t>T}(\mathbf{p}^{(t)}-\mathbf{p}^\opt)$, and
bound
\[
\norm{\mathbf{R}_T}\le\sum_{t>T}\norm{\mathbf{p}^{(t)}-\mathbf{p}^\opt}\le C_0\sum_{t>T}\rho^t=\frac{C_0\,\rho^{T+1}}{1-\rho}\to0 .
\]
Since $\bar{\mathbf{p}}^{(T)}-\mathbf{p}^\opt=\frac1T\mathbf{S}_T=\frac1T\mathbf{C}-\frac1T\mathbf{R}_T$ and $\frac1T\norm{\mathbf{R}_T}=O(\rho^{T}/T)=
o(1/T)$, we conclude $\bar{\mathbf{p}}^{(T)}-\mathbf{p}^\opt=\frac1T\mathbf{C}+o(1/T)$; if $\mathbf{C}\neq0$ this is $\Theta(1/T)$
(lower bound: $\norm{\frac1T\mathbf{C}+o(1/T)}\ge\frac{\norm{\mathbf{C}}}T-o(1/T)$).
\end{proof}

\begin{lemma}[The gap is linear on the feasible cone]
\label{lem:gaplin}
Let $\mathbf{p}^\opt$ be nondegenerate with contact set $S^\opt$ (\Cref{def:nondeg}), and let
$K=\{\boldsymbol{\delta}:\sum_i\delta_i=0,\ \delta_j\ge0\ \forall j\notin S^\opt\}$ be the feasible tangent cone.
Then for feasible $\mathbf{p}=\mathbf{p}^\opt+\boldsymbol{\delta}$ with $\norm{\boldsymbol{\delta}}$ small,
\[
g(\mathbf{p})=F(\boldsymbol{\delta})+O(\norm{\boldsymbol{\delta}}^2),\qquad F(\boldsymbol{\delta}):=\max_{i\in S^\opt}\inner{\nabla v_i(\mathbf{p}^\opt),\boldsymbol{\delta}},
\]
and $F(\boldsymbol{\delta})\ge0$ for all $\boldsymbol{\delta}\in K$.
\end{lemma}
\begin{proof}
Each $v_i$ is smooth ($C^\infty$, indeed rational) on the open set $\{\mathbf{M}(\mathbf{p})\succ0\}\ni\mathbf{p}^\opt$, so
Taylor expansion at $\mathbf{p}^\opt$ gives, on a fixed ball around $\mathbf{p}^\opt$ and uniformly over the finitely
many $i\in[n]$,
\[
v_i(\mathbf{p}^\opt+\boldsymbol{\delta})=v_i(\mathbf{p}^\opt)+\inner{\nabla v_i(\mathbf{p}^\opt),\boldsymbol{\delta}}+O(\norm{\boldsymbol{\delta}}^2).
\]
Let $\Lambda:=\max_i\norm{\nabla v_i(\mathbf{p}^\opt)}$.
\emph{(i) The maximum is attained in $S^\opt$.} For $i\in S^\opt$, $v_i(\mathbf{p}^\opt)=d$, so
$v_i(\mathbf{p})-d=\inner{\nabla v_i(\mathbf{p}^\opt),\boldsymbol{\delta}}+O(\norm{\boldsymbol{\delta}}^2)$. For $i\notin S^\opt$,
$v_i(\mathbf{p}^\opt)\le d-\gsc$ (\Cref{def:nondeg}), so $v_i(\mathbf{p})\le d-\gsc+
\Lambda\norm{\boldsymbol{\delta}}+O(\norm{\boldsymbol{\delta}}^2)<d$ once $\norm{\boldsymbol{\delta}}$ is small. Since $\max_iv_i(\mathbf{p})\ge d$ always
(\Cref{prop:dict}\ref{it:trace}), the overall maximum is therefore attained within $S^\opt$, and
\[
g(\mathbf{p})=\max_iv_i(\mathbf{p})-d=\max_{i\in S^\opt}\big(v_i(\mathbf{p})-d\big)=\max_{i\in S^\opt}\inner{\nabla
v_i(\mathbf{p}^\opt),\boldsymbol{\delta}}+O(\norm{\boldsymbol{\delta}}^2)=F(\boldsymbol{\delta})+O(\norm{\boldsymbol{\delta}}^2),
\]
the $O(\norm{\boldsymbol{\delta}}^2)$ surviving the (finite) maximum because each branch has a uniform remainder.
\emph{(ii) $F\ge0$ on $K$.} For $\boldsymbol{\delta}\in K$ and small $t>0$, the point $\mathbf{p}^\opt+t\boldsymbol{\delta}$ is feasible:
$\sum_i(\mathbf{p}^\opt+t\boldsymbol{\delta})_i=1$, and on coordinates $j\notin S^\opt$ (where $p^\opt_j=0$ by strict
complementarity) $(\mathbf{p}^\opt+t\boldsymbol{\delta})_j=t\delta_j\ge0$, while on $j\in S^\opt$ the coordinate
$p^\opt_j+t\delta_j$ stays positive for small $t$. Hence $g(\mathbf{p}^\opt+t\boldsymbol{\delta})\ge0$
(\Cref{prop:dict}\ref{it:trace}); by part (i), $g(\mathbf{p}^\opt+t\boldsymbol{\delta})=tF(\boldsymbol{\delta})+O(t^2)$, so dividing by
$t$ and letting $t\downarrow0$ gives $F(\boldsymbol{\delta})\ge0$.
\end{proof}

The next lemma upgrades $F\ge0$ to \emph{strict} positivity on the explicit instance, by hand. Recall
from \Cref{sec:instance} the instance
$\mathbf{A}_\opt$ with rows $\mathbf{a}_1=(2,0)$, $\mathbf{a}_2=(0,1)$, $\mathbf{a}_3=(1,1)$, $\mathbf{a}_4=(1,-1)$, optimal design
$\mathbf{p}^\opt=(\tfrac13,0,\tfrac13,\tfrac13)$, $\mathbf{M}^\opt=\diag(2,\tfrac23)$, contact set $S^\opt=\{1,3,4\}$,
and $\gsc=\tfrac12$. By the Hessian formula (\Cref{prop:dict}\ref{it:hess} and \Cref{app:dict}),
$\nabla v_i(\mathbf{p}^\opt)_j=-(\mathbf{a}_i^\top(\mathbf{M}^\opt)^{-1}\mathbf{a}_j)^2=:-B_{ij}$, and the relevant rows of $\mathbf{B}$ are
rational:
\begin{equation}
\label{eq:Brows}
\mathbf{B}_{1,:}=(4,\,0,\,1,\,1),\qquad
\mathbf{B}_{3,:}=(1,\,\tfrac94,\,4,\,1),\qquad
\mathbf{B}_{4,:}=(1,\,\tfrac94,\,1,\,4).
\end{equation}

\begin{lemma}[The linearized gap is strictly positive on the feasible cone of $\mathbf{A}_\opt$]
\label{lem:Fpos}
For $\mathbf{A}_\opt$, with $K=\{\boldsymbol{\delta}\in\R^4:\sum_i\delta_i=0,\ \delta_2\ge0\}$:
$F(\boldsymbol{\delta})=\max_{i\in\{1,3,4\}}\big(-(\mathbf{B}\boldsymbol{\delta})_i\big)>0$ for every $\boldsymbol{\delta}\in K\setminus\{0\}$.
Consequently $\gamma_{\min}:=\min\{F(\boldsymbol{\delta}):\boldsymbol{\delta}\in K,\norm{\boldsymbol{\delta}}_2=1\}>0$.
\end{lemma}
\begin{proof}
Suppose $\boldsymbol{\delta}\in K$ has $F(\boldsymbol{\delta})\le0$, i.e.\ $(\mathbf{B}\boldsymbol{\delta})_1\ge0$, $(\mathbf{B}\boldsymbol{\delta})_3\ge0$,
$(\mathbf{B}\boldsymbol{\delta})_4\ge0$. Summing the three inequalities using \eqref{eq:Brows},
\[
6\delta_1+\tfrac92\delta_2+6\delta_3+6\delta_4\;\ge\;0
\quad\Longleftrightarrow\quad
6\underbrace{(\delta_1+\delta_2+\delta_3+\delta_4)}_{=0}-\tfrac32\delta_2\;\ge\;0
\quad\Longleftrightarrow\quad
\delta_2\le0 ,
\]
so $\delta_2=0$ (as $\delta_2\ge0$ on $K$). With $\delta_2=0$ and $\delta_1+\delta_3+\delta_4=0$, the
three inequalities collapse, using $\delta_1+\delta_3+\delta_4=0$ to eliminate the off-diagonal mass:
\[
(\mathbf{B}\boldsymbol{\delta})_1=4\delta_1+\delta_3+\delta_4=3\delta_1\ge0,\qquad
(\mathbf{B}\boldsymbol{\delta})_3=3\delta_3\ge0,\qquad
(\mathbf{B}\boldsymbol{\delta})_4=3\delta_4\ge0 .
\]
Thus $\delta_1,\delta_3,\delta_4\ge0$ with $\delta_1+\delta_3+\delta_4=0$, forcing
$\delta_1=\delta_3=\delta_4=0$, i.e.\ $\boldsymbol{\delta}=0$. Hence $F>0$ on $K\setminus\{0\}$. The minimum
$\gamma_{\min}$ over the compact set $K\cap\{\norm{\boldsymbol{\delta}}_2=1\}$ of the continuous function $F$ is then
positive. (Numerically $\gamma_{\min}\approx0.386$; \Cref{sec:instance}.)
\end{proof}

\begin{theorem}[Averaging barrier]
\label{thm:avg}
For the explicit instance $\mathbf{A}_\opt$ (with $n=4$, $d=2$, unique nondegenerate
optimum, $\kappa=O(1)$), the \emph{uniform} running average of the CCLY iterates started at
$\mathbf{p}^{(1)}=\tfrac14\mathbf 1$ satisfies
$g(\bar{\mathbf{p}}^{(T)})=\Theta(1/T)$; hence any uniform-averaging certificate uses $\Omega(\eps^{-1})$ leverage
queries on $\mathbf{A}_\opt$.
\end{theorem}
\begin{proof}
The upper bound $g(\bar{\mathbf{p}}^{(T)})=O(1/T)$ is \Cref{thm:ccly}. The lower bound has four steps.

\emph{Step 1: the iterates converge to $\mathbf{p}^\opt$.} By \Cref{rem:titterington} the CCLY iteration is
the map $\Phi(\mathbf{p})=(p_iv_i(\mathbf{p})/d)_i$, the classical multiplicative algorithm; from the positive start
$\tfrac14\mathbf1$ all iterates remain strictly positive (each $v_i(\mathbf{p})>0$ when $\mathbf{M}(\mathbf{p})\succ0$), and
$f(\mathbf{p}^{(t)})$ decreases monotonically to $f^\opt$~\cite{titterington1976,yu2010}. Let $\mathbf{q}$ be any
accumulation point of $\{\mathbf{p}^{(t)}\}\subset\Delta_4$ (compact). Since $f$ is lower semicontinuous on
$\Delta_4$ (it is continuous where $\mathbf{M}\succ0$ and $=+\infty$ where $\det\mathbf{M}=0$, and
$\det\mathbf{M}(\cdot)$ is continuous), $f(\mathbf{q})\le\lim_tf(\mathbf{p}^{(t)})=f^\opt$, so $\mathbf{q}$ is optimal. On
$\mathbf{A}_\opt$ the optimal design is unique (\Cref{lem:unique}; nondegeneracy is verified in closed form in
\Cref{sec:instance}), so $\mathbf{q}=\mathbf{p}^\opt$: the bounded sequence has a unique accumulation point, hence
$\mathbf{p}^{(t)}\to\mathbf{p}^\opt$.

\emph{Step 2: the convergence is R-linear.} $\Phi$ is a rational map, defined and $C^\infty$ on the
open neighborhood $\{\mathbf{p}\in\R^4:\mathbf{M}(\mathbf{p})\succ0\}$ of $\mathbf{p}^\opt$, with $\Phi(\mathbf{p}^\opt)=\mathbf{p}^\opt$ (as
$v_i(\mathbf{p}^\opt)=d$ on $\supp\mathbf{p}^\opt$). Its Jacobian at $\mathbf{p}^\opt$,
\[
D\Phi(\mathbf{p}^\opt)_{ij}=\frac{\delta_{ij}\,v_j(\mathbf{p}^\opt)}{d}-\frac{p^\opt_i\,B_{ij}}{d},
\]
is computed exactly in \Cref{sec:instance}: its full spectrum is
$\{0,\tfrac12,\tfrac12,\tfrac34\}$, so its spectral radius is $\tfrac34<1$. By Ostrowski's
theorem~\cite[10.1.3--10.1.4]{ortega1970}, $\mathbf{p}^\opt$ is a point of attraction and there are a
neighborhood $U\ni\mathbf{p}^\opt$ and, for any $\rho'\in(\tfrac34,1)$, a constant $c_{\rho'}$ such that
iterates entering $U$ satisfy $\norm{\mathbf{p}^{(t)}-\mathbf{p}^\opt}\le c_{\rho'}(\rho')^{t}$ thereafter. (As
$D\Phi(\mathbf{p}^\opt)$ is not normal, single-step error ratios need not be monotone; the conclusion is
R-linear convergence at any rate above the spectral radius, which is exactly the hypothesis of
\Cref{lem:avg}.) By Step~1
the sequence enters $U$ at some finite time $T_0$; absorbing the first $T_0$ terms into the constant,
$\norm{\mathbf{p}^{(t)}-\mathbf{p}^\opt}\le C_0(\rho')^{t}$ for all $t\ge1$: the convergence is R-linear.

\emph{Step 3: the summed displacement $\mathbf{C}$ is a nonzero element of $K$.}
\Cref{lem:avg} applies and gives
$\boldsymbol{\delta}_T:=\bar{\mathbf{p}}^{(T)}-\mathbf{p}^\opt=\tfrac1T\mathbf{C}+\mathbf{e}_T$ with $\norm{\mathbf{e}_T}=o(1/T)$, where
$\mathbf{C}=\sum_{t\ge1}(\mathbf{p}^{(t)}-\mathbf{p}^\opt)$. Each summand has zero coordinate sum, so $\sum_iC_i=0$; and the
second coordinate satisfies $C_2=\sum_{t\ge1}p^{(t)}_2>0$, because $p^\opt_2=0$ while every iterate
keeps $p^{(t)}_2>0$ (Step~1). Hence $\mathbf{C}\in K\setminus\{0\}$, and \Cref{lem:Fpos} gives
$\gamma:=F(\mathbf{C})>0$.

\emph{Step 4: assemble.} The map
$F(\boldsymbol{\delta})=\max_{i\in S^\opt}\inner{\nabla v_i(\mathbf{p}^\opt),\boldsymbol{\delta}}$ is positively homogeneous and
$\Lambda$-Lipschitz (a maximum of finitely many linear functionals of norm $\le\Lambda=\max_i\norm{\nabla
v_i(\mathbf{p}^\opt)}$), so
\[
F(\boldsymbol{\delta}_T)=F\!\big(\tfrac1T\mathbf{C}\big)+\big(F(\boldsymbol{\delta}_T)-F(\tfrac1T\mathbf{C})\big)
=\tfrac1TF(\mathbf{C})\ \pm\ \Lambda\norm{\mathbf{e}_T}=\tfrac{\gamma}{T}+o(1/T).
\]
Moreover $\boldsymbol{\delta}_T\in K$ for every $T$ (its inactive coordinate $\bar p^{(T)}_2\ge0$ and
$\sum_i(\boldsymbol{\delta}_T)_i=0$) and $\norm{\boldsymbol{\delta}_T}=\Theta(1/T)\to0$, so \Cref{lem:gaplin} applies:
\[
g(\bar{\mathbf{p}}^{(T)})=F(\boldsymbol{\delta}_T)+O(\norm{\boldsymbol{\delta}_T}^2)=\tfrac{\gamma}{T}+o(1/T)+O(1/T^2)=\tfrac{\gamma}{T}+
o(1/T)=\Theta(1/T).
\]
Therefore $g(\bar{\mathbf{p}}^{(T)})\le\eps d$ forces $T\ge\frac{\gamma}{\eps d}(1-o(1))=\Omega(1/\eps)$.
\end{proof}

\begin{remark}
That averaging degrades a linear rate to $1/T$ is classical; what is specific to the John ellipsoid is
that averaging is used \emph{precisely} to make the $\ell_\infty$ certificate of \Cref{def:je} provable
by convexity (\Cref{thm:ccly}), and uniformly across~\cite{cohen2019,cao2025,woodruff2025}. The
per-iteration speedups of~\cite{cao2025,woodruff2025} therefore optimize an orthogonal axis. The barrier
is specific to the \emph{uniform} average used throughout this line: suffix or weighted averaging and
periodic restart all preserve the geometric last-iterate rate (\Cref{sec:upper}). We also note what the
theorem does \emph{not} claim: it is a lower bound for a certification rule (uniform averaging of these
iterates), not for the oracle model; oracle lower bounds remain open (\Cref{sec:discussion}).
\end{remark}

\section{The geometry of the optimal face, quantitatively}
\label{sec:face}

Both accuracy phases run, after a warm start, on the optimal face, the smooth stratum of
\Cref{sec:designs}; this section is the shared geometry
beneath them, and beneath the identification handoff that reaches the face. It collects everything they
need: the rank-deficiency of $\nabla^2f$ is exactly harmless (\Cref{lem:flat}); on the face the Hessian
block is positive definite everywhere (\Cref{lem:faceSC}); the optimal design is the
\emph{unconstrained} minimizer of the facial restriction, which is a closed, strictly convex,
self-concordant function (\Cref{lem:min}); and, the workhorse, within an explicit sublevel set
$\Cnull$, the information matrix, all leverage scores, and the full facial Hessian block are within
absolute constant factors of their values at $\mathbf{p}^\opt$ (\Cref{lem:c0}). The proofs of the toolkit
lemmas (\Cref{lem:Msand,lem:hesssand,lem:scineq,lem:levcomp,lem:legendre}) are elementary
(Frobenius-versus-spectral-norm arithmetic plus one integration) and are given in
\Cref{app:sctoolkit}; they are standard self-concordance theory specialized to $-\log\det$ (the
general versions are~\cite[Thms.~5.1.8--5.1.9]{nesterov2018}, and the self-concordance of $-\log\det$
itself is~\cite[\S5.4.4, Lem.~5.4.6 and Thm.~5.4.3]{nesterov2018} or~\cite[Ex.~9.5]{bv2004}), reproved
here to keep every constant explicit and the paper self-contained.

\subsection{Harmless rank-deficiency, and uniqueness}

\begin{lemma}[The objective sees only $\mathbf{M}(\mathbf{p})$; flat directions are harmless]
\label{lem:flat}
$f$, the leverage scores $v_i$, and the gap $g$ depend on $\mathbf{p}$ only through $\mathbf{M}(\mathbf{p})$, which is linear in
$\mathbf{p}$. Hence they are \emph{exactly} constant along $\ker:=\{\boldsymbol{\delta}:\sum_i\delta_i\mathbf{a}_i\mathbf{a}_i^\top=0\}$, which
equals $\ker\nabla^2f(\mathbf{p})$ for every $\mathbf{p}$ with $\mathbf{M}(\mathbf{p})\succ0$.
\end{lemma}
\begin{proof}
By \Cref{prop:dict}, $f$, each $v_i$, and $g=\max_iv_i-d$ are functions of $\mathbf{M}(\mathbf{p})$ alone. For
$\boldsymbol{\delta}\in\ker$ and any $t$,
$\mathbf{M}(\mathbf{p}+t\boldsymbol{\delta})=\mathbf{M}(\mathbf{p})+t\sum_i\delta_i\mathbf{a}_i\mathbf{a}_i^\top=\mathbf{M}(\mathbf{p})$,
so $f,v,g$ take the same value at $\mathbf{p}$ and $\mathbf{p}+t\boldsymbol{\delta}$. The identity $\ker=\ker\nabla^2f(\mathbf{p})$
and its $\mathbf{p}$-independence are \Cref{prop:dict}\ref{it:hess}.
\end{proof}

\begin{lemma}[Uniqueness of the optimal design]
\label{lem:unique}
Under \Cref{def:nondeg}, $\mathbf{p}^\opt$ is the unique optimal design in $\Delta_n$.
\end{lemma}
\begin{proof}
\emph{(i) All optimal designs share one information matrix.} Let $\mathbf{p},\mathbf{q}$ be optimal and
$\mathbf{H}:=\mathbf{M}(\mathbf{q})-\mathbf{M}(\mathbf{p})$. The function $\varphi(t):=f(\mathbf{p}+t(\mathbf{q}-\mathbf{p}))=-\log\det(\mathbf{M}(\mathbf{p})+t\mathbf{H})$ is convex on
$[0,1]$ with equal values $f^\opt$ at $t\in\{0,1\}$, hence constant, hence $\varphi''\equiv0$; but
$\varphi''(t)=\norm{\mathbf{H}}^2_{\mathbf{M}(\mathbf{p})+t\mathbf{H}}$ (\Cref{eq:hessform-matrix} in \Cref{app:sctoolkit}), which
vanishes only if $\mathbf{H}=0$. So $\mathbf{M}(\mathbf{q})=\mathbf{M}(\mathbf{p})=\mathbf{M}^\opt$.
\emph{(ii) Optimal supports lie in $S^\opt$.} For optimal $\mathbf{q}$,
$d=\sum_iq_iv_i(\mathbf{q})=\sum_iq_iv_i(\mathbf{p}^\opt)$ by \Cref{prop:dict}\ref{it:trace} and (i), while
$v_i(\mathbf{p}^\opt)\le d$ for all $i$ (\Cref{thm:kw}); the weighted average equals the maximum only if
$q_i\,(d-v_i(\mathbf{p}^\opt))=0$ for every $i$, i.e.\ $\supp(\mathbf{q})\subseteq\{i:v_i(\mathbf{p}^\opt)=d\}=S^\opt$.
\emph{(iii) Conclude.} $\mathbf{q}-\mathbf{p}^\opt$ is supported on $S^\opt$ and
$\sum_{i\in S^\opt}(q_i-p^\opt_i)\mathbf{a}_i\mathbf{a}_i^\top=\mathbf{M}(\mathbf{q})-\mathbf{M}^\opt=0$; linear independence of the contact
matrices (\Cref{def:nondeg}) forces $\mathbf{q}=\mathbf{p}^\opt$.
\end{proof}

\subsection{The facial restriction and its local norm}
\label{sec:facialrestriction}

Define the affine hull of the optimal face, its tangent space, and the facial domain:
\[
V:=\Big\{\mathbf{p}\in\R^n:\supp(\mathbf{p})\subseteq S^\opt,\ \textstyle\sum_{i}p_i=1\Big\},\qquad
T^\opt:=\Big\{\boldsymbol{\delta}:\supp(\boldsymbol{\delta})\subseteq S^\opt,\ \textstyle\sum_i\delta_i=0\Big\},
\]
\[
D:=\{\mathbf{p}\in V:\ \mathbf{M}(\mathbf{p})\succ0\}\supseteq\relint\Delta_{S^\opt},\qquad h:=f|_D .
\]
Note that $D$ is permitted to contain points with negative coordinates; $h$ and the leverage scores
remain well defined there. We write $\widehat{\mathbf{H}}(\mathbf{p}):=\big(\nabla^2f(\mathbf{p})\big)_{S^\opt\times
S^\opt}=\big((\mathbf{a}_i^\top\mathbf{M}(\mathbf{p})^{-1}\mathbf{a}_j)^2\big)_{i,j\in S^\opt}$ for the \emph{facial Hessian block},
an $m\times m$ matrix; the Hessian of $h$ (as a function on the affine space $V$) is the restriction
of $\widehat{\mathbf{H}}(\mathbf{p})$ to $T^\opt$.

For $\mathbf{X}\succ0$ and $\mathbf{H}\in\Sym^d$ define the \emph{local norm}
$\locnorm{\mathbf{H}}{\mathbf{X}}:=\fnorm{\mathbf{X}^{-1/2}\mathbf{H}\mathbf{X}^{-1/2}}$, and for any weights $\mathbf{p}$ the \emph{Dikin
radius}
\[
\varrho(\mathbf{p})\;:=\;\locnorm{\mathbf{M}(\mathbf{p})-\mathbf{M}^\opt}{\mathbf{M}^\opt}.
\]
By \Cref{eq:hessform}, for $\mathbf{p}\in V$ (so that $\mathbf{p}-\mathbf{p}^\opt\in T^\opt$),
\begin{equation}
\label{eq:rho-id}
\varrho(\mathbf{p})^2=(\mathbf{p}-\mathbf{p}^\opt)^\top\widehat{\mathbf{H}}(\mathbf{p}^\opt)\,(\mathbf{p}-\mathbf{p}^\opt):
\end{equation}
the Dikin radius is exactly the distance to $\mathbf{p}^\opt$ in the local norm of $f$ at $\mathbf{p}^\opt$.

\begin{lemma}[Positive definiteness on the whole facial domain]
\label{lem:faceSC}
Under \Cref{def:nondeg}, for every $\mathbf{p}\in D$ the facial Hessian block is positive definite on all of
$\R^{S^\opt}$: for $\boldsymbol{\delta}\in\R^{S^\opt}\setminus\{0\}$,
\[
\boldsymbol{\delta}^\top\widehat{\mathbf{H}}(\mathbf{p})\,\boldsymbol{\delta}
=\Big\lVert\textstyle\sum_{i\in S^\opt}\delta_i\mathbf{a}_i\mathbf{a}_i^\top\Big\rVert^2_{\mathbf{M}(\mathbf{p})}>0 .
\]
\end{lemma}
\begin{proof}
The identity is \Cref{eq:hessform} restricted to coordinates in $S^\opt$. As $\mathbf{M}(\mathbf{p})^{-1/2}$ is
invertible, the norm vanishes iff $\sum_{i\in S^\opt}\delta_i\mathbf{a}_i\mathbf{a}_i^\top=0$, which by the linear
independence in \Cref{def:nondeg} forces $\boldsymbol{\delta}=0$. The independence hypothesis concerns only the
fixed matrices $\{\mathbf{a}_i\mathbf{a}_i^\top\}_{i\in S^\opt}$, not $\mathbf{p}$, so the conclusion holds at \emph{every}
$\mathbf{p}\in D$---including points with negative coordinates.
\end{proof}

\begin{definition}[Facial condition number at the optimum]
\label{def:kappa}
Let
\[
\muopt:=\lambda_{\min}\Big(\widehat{\mathbf{H}}(\mathbf{p}^\opt)\big|_{T^\opt}\Big)>0,\qquad
\Lopt:=\lambda_{\max}\big(\widehat{\mathbf{H}}(\mathbf{p}^\opt)\big)\ge\muopt,\qquad
\kappa:=\Lopt/\muopt .
\]
(The lower constant is taken on the tangent $T^\opt$, where the optimization lives; the upper constant
is the operator norm of the \emph{full} $m\times m$ block, which is what controls the leverage-score
map $\mathbf{p}\mapsto(v_i(\mathbf{p}))_{i\in S^\opt}$ and hence the Frank--Wolfe gap. Taking the full block on top
only enlarges $\kappa$, and is what makes the gap conversion in \Cref{lem:gapconv} correct.)
\end{definition}

\begin{lemma}[The facial problem is unconstrained, closed, self-concordant]
\label{lem:min}
Under \Cref{def:nondeg}:
\begin{enumerate}[label=\emph{(\roman*)},leftmargin=*,itemsep=1pt]
\item $D$ is open in $V$ and convex; $h$ is $C^\infty$ and strictly convex on $D$; and $h$ is closed
(its sublevel sets $\{h\le c\}$ are closed in $\R^n$).
\item $\inner{\nabla f(\mathbf{p}^\opt),\boldsymbol{\delta}}=0$ for every $\boldsymbol{\delta}\in T^\opt$; consequently $\mathbf{p}^\opt$ is
the unique minimizer of $h$ over all of $D$---the simplex constraints are inactive for the facial
problem.
\item For every $\mathbf{p}\in D$:\quad $h(\mathbf{p})\;\ge\;f^\opt+\wsc\big(\varrho(\mathbf{p})\big)$.
\end{enumerate}
\end{lemma}
\begin{proof}
\emph{(i)} $D$ is the preimage of the open convex cone $\Sym^d_{\succ0}$ under the affine map
$\mathbf{p}\mapsto\mathbf{M}(\mathbf{p})$ restricted to $V$, hence open in $V$ and convex; $h$ is a composition of
$C^\infty$ maps there. Strict convexity: \Cref{lem:faceSC} makes $\nabla^2h\succ0$ on $T^\opt$
throughout $D$. Closedness: let $\mathbf{p}_k\in\{h\le c\}$ with $\mathbf{p}_k\to\bar{\mathbf{p}}$. Then
$\bar{\mathbf{p}}\in V$ ($V$ is closed) and $\mathbf{M}(\bar{\mathbf{p}})=\lim_k\mathbf{M}(\mathbf{p}_k)$ is a limit of positive-definite
matrices, hence positive semidefinite. If $\mathbf{M}(\bar{\mathbf{p}})$ were singular, then
$\det\mathbf{M}(\mathbf{p}_k)\to\det\mathbf{M}(\bar{\mathbf{p}})=0$ by continuity of $\det$, so
$h(\mathbf{p}_k)=-\log\det\mathbf{M}(\mathbf{p}_k)\to+\infty$, contradicting $h(\mathbf{p}_k)\le c$. Hence
$\mathbf{M}(\bar{\mathbf{p}})\succ0$, i.e.\ $\bar{\mathbf{p}}\in D$, and $h(\bar{\mathbf{p}})\le c$ by continuity: the sublevel
set is closed. (Self-concordance of $h$ is used only through the explicit inequalities of
\Cref{app:sctoolkit}, so we do not invoke the abstract definition.)

\emph{(ii)} By the definition of the contact set, $v_i(\mathbf{p}^\opt)=d$ for every $i\in S^\opt$.
Hence, by \Cref{prop:dict}\ref{it:grad}, the restriction of $\nabla f(\mathbf{p}^\opt)$
to the coordinates $S^\opt$ is the constant vector $-d\,\mathbf 1$, and for $\boldsymbol{\delta}\in T^\opt$,
$\inner{\nabla f(\mathbf{p}^\opt),\boldsymbol{\delta}}=-d\sum_{i\in S^\opt}\delta_i=0$. Since $h$ is convex on the convex
set $D$ and its directional derivatives at $\mathbf{p}^\opt$ vanish on $T^\opt$, $\mathbf{p}^\opt$ is a global
minimizer of $h$ on $D$; uniqueness follows from strict convexity.

\emph{(iii)} Apply the matrix-space lower bound of \Cref{lem:scineq}\ref{it:sclower} with
$\mathbf{X}=\mathbf{M}^\opt$, $\mathbf{Y}=\mathbf{M}(\mathbf{p})$ (both $\succ0$), and $F(\cdot)=-\log\det(\cdot)$:
\[
h(\mathbf{p})=F(\mathbf{M}(\mathbf{p}))\ \ge\ F(\mathbf{M}^\opt)+\inner{-(\mathbf{M}^\opt)^{-1},\,\mathbf{M}(\mathbf{p})-\mathbf{M}^\opt}+\wsc\big(\varrho(\mathbf{p})\big).
\]
The linear term vanishes: writing $\boldsymbol{\delta}=\mathbf{p}-\mathbf{p}^\opt\in T^\opt$,
\[
\inner{(\mathbf{M}^\opt)^{-1},\,\mathbf{M}(\mathbf{p})-\mathbf{M}^\opt}
=\sum_{i\in S^\opt}\delta_i\,\mathbf{a}_i^\top(\mathbf{M}^\opt)^{-1}\mathbf{a}_i
=\sum_{i\in S^\opt}\delta_i\,v_i(\mathbf{p}^\opt)=d\sum_{i\in S^\opt}\delta_i=0 . \qedhere
\]
\end{proof}

\subsection{The self-concordance toolkit}

The next four lemmas are proved in \Cref{app:sctoolkit}. Throughout, $\mathbf{X},\mathbf{Y}\succ0$ are in $\Sym^d$
and $\locnorm{\cdot}{\mathbf{X}}$ is the local norm above.

\begin{lemma}[Spectral sandwich]
\label{lem:Msand}
Let $\mathbf{H}\in\Sym^d$ and $r:=\locnorm{\mathbf{H}}{\mathbf{X}}$. Then $-r\mathbf{X}\preceq\mathbf{H}\preceq r\mathbf{X}$; in particular
$(1-r)\mathbf{X}\preceq\mathbf{X}+\mathbf{H}\preceq(1+r)\mathbf{X}$, and $\mathbf{X}+\mathbf{H}\succ0$ whenever $r<1$.
\end{lemma}

\begin{lemma}[Local-norm comparison]
\label{lem:hesssand}
If $(1-r)\mathbf{X}\preceq\mathbf{Y}\preceq(1+r)\mathbf{X}$ for some $r\in[0,1)$, then for every $\mathbf{H}\in\Sym^d$,
\[
(1+r)^{-2}\,\locnorm{\mathbf{H}}{\mathbf{X}}^2\;\le\;\locnorm{\mathbf{H}}{\mathbf{Y}}^2\;\le\;(1-r)^{-2}\,\locnorm{\mathbf{H}}{\mathbf{X}}^2 .
\]
The left inequality requires only $\mathbf{Y}\preceq(1+r)\mathbf{X}$ (any $r\ge0$); the right only
$\mathbf{Y}\succeq(1-r)\mathbf{X}$.
\end{lemma}

\begin{lemma}[Integrated self-concordance bounds for $-\log\det$]
\label{lem:scineq}
Let $F(\cdot)=-\log\det(\cdot)$ on $\Sym^d_{\succ0}$, with gradient $\nabla F(\mathbf{X})=-\mathbf{X}^{-1}$ in the
trace inner product, and let $r:=\locnorm{\mathbf{Y}-\mathbf{X}}{\mathbf{X}}$.
\begin{enumerate}[label=\emph{(\alph*)},leftmargin=*,itemsep=1pt]
\item\label{it:sclower} $F(\mathbf{Y})\ \ge\ F(\mathbf{X})+\inner{\nabla F(\mathbf{X}),\mathbf{Y}-\mathbf{X}}+\wsc(r)$\quad (no restriction on $r$);
\item\label{it:scupper} if $r<1$:\quad $F(\mathbf{Y})\ \le\ F(\mathbf{X})+\inner{\nabla F(\mathbf{X}),\mathbf{Y}-\mathbf{X}}+\wscs(r)$.
\end{enumerate}
\end{lemma}

\begin{lemma}[Leverage comparison]
\label{lem:levcomp}
If $(1-r)\mathbf{X}\preceq\mathbf{Y}\preceq(1+r)\mathbf{X}$ with $r\in[0,1)$, then for every $\mathbf{a}\in\R^d$,
\[
(1+r)^{-1}\,\mathbf{a}^\top\mathbf{X}^{-1}\mathbf{a}\;\le\;\mathbf{a}^\top\mathbf{Y}^{-1}\mathbf{a}\;\le\;(1-r)^{-1}\,\mathbf{a}^\top\mathbf{X}^{-1}\mathbf{a} .
\]
\end{lemma}

\begin{lemma}[Properties of $\wsc,\wscs$]
\label{lem:legendre}
$\wsc(t)=t-\log(1+t)$ and $\wscs(t)=-t-\log(1-t)$ are increasing, strictly convex, and vanish to second
order at $0$; they are Legendre conjugates: $\wscs(t)=\sup_{s\ge0}\{ts-\wsc(s)\}$ for $t\in[0,1)$, with
maximizer $s=t/(1-t)$. Moreover $\wscs(t)\le\frac{t^2}{2(1-t)}\le\tfrac23t^2$ for $t\in[0,\tfrac14]$,
and $\wsc(\tfrac14)=\tfrac14-\log\tfrac54>\tfrac1{38}$.
\end{lemma}

\subsection{The strict-complementarity penalty and localization}

\begin{lemma}[Lower bound with penalty: strict complementarity $+$ curvature]
\label{lem:lower}
For every $\mathbf{p}\in\Delta_n$ (if $\mathbf{M}(\mathbf{p})$ is singular both sides are interpreted as $f(\mathbf{p})=+\infty$):
\[
f(\mathbf{p})\;\ge\;f^\opt\;+\;\gsc\sum_{i\notin S^\opt}p_i\;+\;\wsc\big(\varrho(\mathbf{p})\big).
\]
\end{lemma}
\begin{proof}
Assume $\mathbf{M}(\mathbf{p})\succ0$. Apply \Cref{lem:scineq}\ref{it:sclower} with $\mathbf{X}=\mathbf{M}^\opt$, $\mathbf{Y}=\mathbf{M}(\mathbf{p})$:
\[
f(\mathbf{p})\ \ge\ f^\opt+\inner{-(\mathbf{M}^\opt)^{-1},\,\mathbf{M}(\mathbf{p})-\mathbf{M}^\opt}+\wsc(\varrho(\mathbf{p})).
\]
For the linear term, using $\Tr\big((\mathbf{M}^\opt)^{-1}\mathbf{a}_i\mathbf{a}_i^\top\big)=v_i(\mathbf{p}^\opt)$,
$\Tr\big((\mathbf{M}^\opt)^{-1}\mathbf{M}^\opt\big)=d$, $\sum_ip_i=1$, and $v_i(\mathbf{p}^\opt)=d$ for $i\in S^\opt$:
\[
\inner{-(\mathbf{M}^\opt)^{-1},\,\mathbf{M}(\mathbf{p})-\mathbf{M}^\opt}
=d-\sum_{i}p_i\,v_i(\mathbf{p}^\opt)
=\sum_{i\notin S^\opt}\big(d-v_i(\mathbf{p}^\opt)\big)\,p_i
\;\ge\;\gsc\sum_{i\notin S^\opt}p_i ,
\]
the middle equality because
$\sum_ip_iv_i(\mathbf{p}^\opt)=d\big(1-\sum_{i\notin S^\opt}p_i\big)+\sum_{i\notin S^\opt}v_i(\mathbf{p}^\opt)p_i$,
and the final inequality from the definition of $\gsc$ (\Cref{def:nondeg}).
\end{proof}

\begin{remark}
Dropping the $\wsc$ term recovers the classical convexity penalty
$f(\mathbf{p})-f^\opt\ge\gsc\cdot(\text{inactive mass})$; the curvature term is what converts a small
objective gap into a small \emph{Dikin radius}, and it is the engine of everything below.
\end{remark}

\begin{definition}[The localization constant $c_0$, the facial sublevel set, and the Dikin ball]
\label{def:c0}
Set
\[
\bar r\;:=\;\min\Big\{\tfrac12,\ \ \frac{\gsc}{2d},\ \ \frac{\pmin\sqrt{\muopt}}{2}\Big\},
\qquad
c_0\;:=\;\wsc(\bar r)\;=\;\bar r-\log(1+\bar r),
\]
\[
\Cnull\;:=\;\{\mathbf{p}\in V:\ h(\mathbf{p})\le f^\opt+c_0\},
\qquad
\bar{\mathcal D}\;:=\;\{\mathbf{p}\in V:\ \varrho(\mathbf{p})\le\tfrac12\}\quad(\text{the facial Dikin ball}).
\]
\end{definition}

\begin{lemma}[Localization]
\label{lem:c0}
Under \Cref{def:nondeg}, with $c_0,\bar r,\Cnull,\bar{\mathcal D}$ as in \Cref{def:c0}:
\begin{enumerate}[label=\emph{(\alph*)},leftmargin=*,itemsep=2pt]
\item\label{it:c0-slack} \emph{(Separation of contact and non-contact scores.)} Every $\mathbf{p}\in\Delta_n$
with $\mathbf{M}(\mathbf{p})\succ0$ and $f(\mathbf{p})\le f^\opt+c_0$ satisfies
\[
v_j(\mathbf{p})< d-\tfrac\gsc2\quad\text{for every }j\notin S^\opt,
\qquad\text{and}\qquad
v_i(\mathbf{p})\ge d-\tfrac\gsc2\quad\text{for every }i\in S^\opt .
\]
\item\label{it:c0-interior} \emph{(Interiority.)} $\Cnull$ is convex, compact, contained in
$\bar{\mathcal D}$, and
$\Cnull\subseteq\{\mathbf{p}:p_i\ge\pmin/2\ \forall i\in S^\opt\}\subseteq\relint\Delta_{S^\opt}$: every point
of the facial sublevel set is a genuine design, with all contact weights bounded away from $0$.
\item\label{it:c0-cond} \emph{(Uniform conditioning on the Dikin ball.)} $\bar{\mathcal D}$ is convex,
$\mathbf{M}(\mathbf{p})\succeq\tfrac12\mathbf{M}^\opt\succ0$ on it (so $\bar{\mathcal D}\subseteq D$), and for every
$\mathbf{p}\in\bar{\mathcal D}$ and $\boldsymbol{\delta}\in\R^{S^\opt}$,
\[
\tfrac49\,\boldsymbol{\delta}^\top\widehat{\mathbf{H}}(\mathbf{p}^\opt)\boldsymbol{\delta}
\;\le\;\boldsymbol{\delta}^\top\widehat{\mathbf{H}}(\mathbf{p})\boldsymbol{\delta}
\;\le\;4\,\boldsymbol{\delta}^\top\widehat{\mathbf{H}}(\mathbf{p}^\opt)\boldsymbol{\delta} .
\]
Consequently, with $\mu_{\mathrm{loc}}:=\tfrac49\muopt$ and $L_{\mathrm{loc}}:=4\Lopt$
(so $\kappa_{\mathrm{loc}}:=L_{\mathrm{loc}}/\mu_{\mathrm{loc}}=9\kappa$): $h$ is
$\mu_{\mathrm{loc}}$-strongly convex on $\bar{\mathcal D}$, and for all $\mathbf{p},\mathbf{q}\in\bar{\mathcal D}$,
$\norm{\mathbf{v}_{S^\opt}(\mathbf{p})-\mathbf{v}_{S^\opt}(\mathbf{q})}_2\le L_{\mathrm{loc}}\norm{\mathbf{p}-\mathbf{q}}_2$ and
$\norm{P_{T^\opt}\big(\nabla_{S^\opt} f(\mathbf{p})-\nabla_{S^\opt} f(\mathbf{q})\big)}_2\le
L_{\mathrm{loc}}\norm{\mathbf{p}-\mathbf{q}}_2$, where
$\mathbf{v}_{S^\opt}=(v_i)_{i\in S^\opt}$ and $P_{T^\opt}$ is the orthogonal projector onto $T^\opt$.
\end{enumerate}
\end{lemma}
\begin{proof}
First, for any $\mathbf{p}\in\Cnull$, \Cref{lem:min}(iii) gives $\wsc(\varrho(\mathbf{p}))\le h(\mathbf{p})-f^\opt\le
c_0=\wsc(\bar r)$, and since $\wsc$ is strictly increasing (\Cref{lem:legendre}),
$\varrho(\mathbf{p})\le\bar r\le\tfrac12$: thus $\Cnull\subseteq\bar{\mathcal D}$.

\ref{it:c0-slack} For $\mathbf{p}\in\Delta_n$ with $f(\mathbf{p})-f^\opt\le c_0$, \Cref{lem:lower} gives
$\wsc(\varrho(\mathbf{p}))\le c_0$, so $\varrho(\mathbf{p})\le\bar r$ as above, and \Cref{lem:Msand} gives
$(1-\bar r)\mathbf{M}^\opt\preceq\mathbf{M}(\mathbf{p})\preceq(1+\bar r)\mathbf{M}^\opt$. By \Cref{lem:levcomp} and
$v_j(\mathbf{p}^\opt)\le d-\gsc$ for $j\notin S^\opt$,
\[
v_j(\mathbf{p})\;\le\;\frac{v_j(\mathbf{p}^\opt)}{1-\bar r}\;\le\;\frac{d-\gsc}{1-\bar r}
\;\le\;\frac{2d\,(d-\gsc)}{2d-\gsc}\;<\;d-\frac{\gsc}2,
\]
where the third step uses $\bar r\le\gsc/(2d)$, i.e.\ $1-\bar r\ge(2d-\gsc)/(2d)$, and the final
strict inequality is $4d(d-\gsc)<(2d-\gsc)^2\iff0<\gsc^2$. On the contact side, for $i\in S^\opt$,
\Cref{lem:levcomp} and $v_i(\mathbf{p}^\opt)=d$ give
\[
v_i(\mathbf{p})\;\ge\;\frac{d}{1+\bar r}\;\ge\;d\,(1-\bar r)\;\ge\;d-\frac{\gsc}2 ,
\]
again by $\bar r\le\gsc/(2d)$.

\ref{it:c0-interior} Convexity: $\Cnull$ is a sublevel set of the convex $h$ on the convex $D$.
For interiority, let $\mathbf{p}\in\Cnull$ and $\boldsymbol{\delta}:=\mathbf{p}-\mathbf{p}^\opt\in T^\opt$. By \eqref{eq:rho-id} and
\Cref{def:kappa},
$\varrho(\mathbf{p})^2=\boldsymbol{\delta}^\top\widehat{\mathbf{H}}(\mathbf{p}^\opt)\boldsymbol{\delta}\ge\muopt\norm{\boldsymbol{\delta}}_2^2$, so
\[
\norm{\mathbf{p}-\mathbf{p}^\opt}_2\;\le\;\frac{\varrho(\mathbf{p})}{\sqrt{\muopt}}\;\le\;\frac{\bar r}{\sqrt{\muopt}}
\;\le\;\frac{\pmin}{2},
\]
using $\bar r\le\pmin\sqrt{\muopt}/2$. Hence $p_i\ge p^\opt_i-\pmin/2\ge\pmin/2>0$ for $i\in S^\opt$,
and $\sum_ip_i=1$ with $\supp(\mathbf{p})\subseteq S^\opt$: $\mathbf{p}\in\relint\Delta_{S^\opt}$. Compactness:
$\Cnull$ is closed (\Cref{lem:min}(i)) and contained in the bounded set $\Delta_{S^\opt}$.

\ref{it:c0-cond} Convexity of $\bar{\mathcal D}$: $\varrho$ is a norm of the affine quantity
$\mathbf{p}-\mathbf{p}^\opt$ on $V$ (\Cref{eq:rho-id} and \Cref{lem:faceSC}), so its sublevel sets are convex
(ellipsoids in $V$). For $\mathbf{p}\in\bar{\mathcal D}$, \Cref{lem:Msand} with
$\mathbf{H}=\mathbf{M}(\mathbf{p})-\mathbf{M}^\opt$, $r=\varrho(\mathbf{p})\le\tfrac12$, gives
$\tfrac12\mathbf{M}^\opt\preceq\mathbf{M}(\mathbf{p})\preceq\tfrac32\mathbf{M}^\opt$; in particular $\mathbf{M}(\mathbf{p})\succ0$. Then
\Cref{lem:hesssand} (with $r=\tfrac12$, so $(1+r)^{-2}\ge\tfrac49$ and $(1-r)^{-2}\le4$) gives, for
every $\boldsymbol{\delta}\in\R^{S^\opt}$, writing
$\mathbf{A}_{\boldsymbol{\delta}}:=\sum_{i\in S^\opt}\delta_i\mathbf{a}_i\mathbf{a}_i^\top$,
\[
\boldsymbol{\delta}^\top\widehat{\mathbf{H}}(\mathbf{p})\boldsymbol{\delta}=\locnorm{\mathbf{A}_{\boldsymbol{\delta}}}{\mathbf{M}(\mathbf{p})}^2
\;\in\;\Big[\tfrac49\locnorm{\mathbf{A}_{\boldsymbol{\delta}}}{\mathbf{M}^\opt}^2,\ 4\locnorm{\mathbf{A}_{\boldsymbol{\delta}}}{\mathbf{M}^\opt}^2\Big]
=\Big[\tfrac49\,\boldsymbol{\delta}^\top\widehat{\mathbf{H}}(\mathbf{p}^\opt)\boldsymbol{\delta},\ 4\,\boldsymbol{\delta}^\top\widehat{\mathbf{H}}(\mathbf{p}^\opt)\boldsymbol{\delta}\Big].
\]
Strong convexity of $h$ on the convex set $\bar{\mathcal D}$ follows by restricting to
$\boldsymbol{\delta}\in T^\opt$ and integrating the Hessian lower bound
$\nabla^2h\succeq\tfrac49\muopt\,\mathbf{I}=\mu_{\mathrm{loc}}\mathbf{I}$ on $T^\opt$ along segments of
$\bar{\mathcal D}$.
For the Lipschitz bounds, $\nabla_{\!\mathbf{p}}\,\mathbf{v}_{S^\opt}=-\widehat{\mathbf{H}}$ on the
$S^\opt$ coordinates (\Cref{prop:dict}\ref{it:grad},\ref{it:hess}), so for
$\mathbf{p},\mathbf{q}\in\bar{\mathcal D}$, by the fundamental theorem of calculus along the segment (which stays
in $\bar{\mathcal D}$ by convexity),
\[
\norm{\mathbf{v}_{S^\opt}(\mathbf{p})-\mathbf{v}_{S^\opt}(\mathbf{q})}_2
=\Big\lVert\int_0^1\widehat{\mathbf{H}}\big(\mathbf{q}+s(\mathbf{p}-\mathbf{q})\big)(\mathbf{p}-\mathbf{q})\,ds\Big\rVert_2
\le\max_{s}\big\lVert\widehat{\mathbf{H}}(\cdot)\big\rVert_{\mathrm{op}}\norm{\mathbf{p}-\mathbf{q}}_2
\le4\Lopt\norm{\mathbf{p}-\mathbf{q}}_2 ,
\]
and the projected version follows since $\norm{P_{T^\opt}\mathbf{x}}_2\le\norm{\mathbf{x}}_2$.
\end{proof}

\begin{lemma}[Gap conversion on the facial sublevel set]
\label{lem:gapconv}
For every $\mathbf{p}\in\Cnull$:
\[
g(\mathbf{p})\;\le\;L_{\mathrm{loc}}\,\norm{\mathbf{p}-\mathbf{p}^\opt}_2
\;\le\;L_{\mathrm{loc}}\sqrt{\frac{2\,(h(\mathbf{p})-f^\opt)}{\mu_{\mathrm{loc}}}}
\;=\;\sqrt{2\,L_{\mathrm{loc}}\,\kappa_{\mathrm{loc}}\,\big(h(\mathbf{p})-f^\opt\big)}\,.
\]
\end{lemma}
\begin{proof}
By \Cref{lem:c0}\ref{it:c0-interior}, $\mathbf{p}$ is a design with $f(\mathbf{p})-f^\opt\le c_0$, so
\Cref{lem:c0}\ref{it:c0-slack} gives $v_j(\mathbf{p})\le d-\gsc/2<d$ for $j\notin S^\opt$; since
$\max_iv_i(\mathbf{p})\ge d$ always (\Cref{prop:dict}\ref{it:trace}), the maximum in
$g(\mathbf{p})=\max_{i\in[n]}v_i(\mathbf{p})-d$ is attained on $S^\opt$. For $i\in S^\opt$, $v_i(\mathbf{p}^\opt)=d$, so
\[
g(\mathbf{p})=\max_{i\in S^\opt}\big(v_i(\mathbf{p})-v_i(\mathbf{p}^\opt)\big)
\le\norm{\mathbf{v}_{S^\opt}(\mathbf{p})-\mathbf{v}_{S^\opt}(\mathbf{p}^\opt)}_\infty
\le\norm{\mathbf{v}_{S^\opt}(\mathbf{p})-\mathbf{v}_{S^\opt}(\mathbf{p}^\opt)}_2
\le L_{\mathrm{loc}}\norm{\mathbf{p}-\mathbf{p}^\opt}_2,
\]
the last step by \Cref{lem:c0}\ref{it:c0-cond}, applicable since $\mathbf{p},\mathbf{p}^\opt\in\Cnull\subseteq
\bar{\mathcal D}$. Finally, strong convexity on $\bar{\mathcal D}$
with $\nabla h(\mathbf{p}^\opt)|_{T^\opt}=0$ (\Cref{lem:min}(ii)) gives
$\tfrac{\mu_{\mathrm{loc}}}2\norm{\mathbf{p}-\mathbf{p}^\opt}_2^2\le h(\mathbf{p})-f^\opt$.
\end{proof}

\begin{remark}[All constants are explicit]
\label{rem:constants}
$\bar r$, $c_0$, $\mu_{\mathrm{loc}}$, $L_{\mathrm{loc}}$, $\kappa_{\mathrm{loc}}=9\kappa$ are explicit
functions of $(\gsc,\pmin,\muopt,\Lopt,d)$. No compactness argument is used anywhere in
\Cref{sec:upper,sec:newton}; this is what permits the clean comparison
$\kappa=O(d^4)\kappa_\Phi$ in \Cref{thm:domination} and keeps every condition number out of the
$\eps$-dependent terms.
\end{remark}

\section{Reaching the face, and the accelerated accuracy phase}
\label{sec:upper}

This section pays two of the three costs. It pays the \emph{identification cost} (reaching the optimal
face) and then runs the \emph{first accuracy phase} on it. We assemble the algorithm
(\Cref{alg:main}) and prove that its first two phases (the Wolfe--Atwood warm start and the away-step
identification) land in the facial sublevel set $\Cnull$ after an $\eps$-independent number of queries
(\Cref{lem:identify}); these queries are the whole of $C(\mathbf{A})$: $\eps$-independent, but
condition-dependent. We then analyze the \emph{accelerated} facial phase, proving \Cref{thm:upper-intro};
its accuracy dependence is $O(\sqrt\kappa\,\log(1/\eps))$, the first improvement on the averaging rate. The
\emph{doubly-logarithmic} accuracy (the \emph{Newton} facial phase) is \Cref{sec:newton}. Throughout,
\Cref{def:nondeg} is in force and
$c_0,\bar r,\Cnull,\bar{\mathcal D},\mu_{\mathrm{loc}},L_{\mathrm{loc}},\kappa_{\mathrm{loc}}=9\kappa$
are as in \Cref{sec:face}.

\begin{algorithm}[t]
\caption{John ellipsoid via warm start, identification, and a facial phase}
\label{alg:main}
\begin{algorithmic}[1]
\Require $\mathbf{A}\in\R^{n\times d}$ with nondegenerate optimal design; accuracy $\eps\in(0,1)$;
a tolerance $c\in(0,c_0]$ and the slack $\gsc>0$ (instance constants, \Cref{rem:params});
facial phase $\in\{\textsc{Accel},\textsc{Newton}\}$
\Ensure $\mathbf{p}\in\Delta_n$ with $\max_iv_i(\mathbf{p})\le(1+\eps)d$
\State $\mathbf{p}\gets$ Kumar--Y\i ld\i r\i m volumetric initialization \Comment{$f(\mathbf{p})-f^\opt=O(d\log d)$~\cite{kumar2005,todd2007}}
\While{$\max_iv_i(\mathbf{p})-d>c$ \textbf{ or } $\min_{i:\,p_i>0}v_i(\mathbf{p})<d-\gsc/2$}
   \Comment{Phases 1--2; exits after $C(\mathbf{A})$ queries with $\mathbf{p}\in\Cnull$ (\Cref{lem:identify})}
  \State $\mathbf{v}\gets(\mathbf{a}_i^\top \mathbf{M}(\mathbf{p})^{-1}\mathbf{a}_i)_{i\in[n]}$ \Comment{one leverage-score query $=-\nabla f(\mathbf{p})$}
  \State $j_+\gets\argmax_i v_i,\quad j_-\gets\argmin_{i:\,p_i>0} v_i$
  \If{$v_{j_+}-d\ \ge\ d-v_{j_-}$} \Comment{toward step; exact line search (\Cref{lem:linesearch})}
     \State $\gamma\gets\frac{v_{j_+}-d}{d\,(v_{j_+}-1)}$;\quad $\mathbf{p}\gets(1-\gamma)\mathbf{p}+\gamma\,\mathbf{e}_{j_+}$
  \Else{} \Comment{away / drop step; exact line search (\Cref{lem:linesearch}); identifies the active set (\Cref{thm:bomze})}
     \State $\gamma_{\max}\gets\tfrac{p_{j_-}}{1-p_{j_-}}$;\quad
       $\gamma\gets\begin{cases}\min\!\big\{\tfrac{d-v_{j_-}}{d\,(v_{j_-}-1)},\ \gamma_{\max}\big\}&\text{if }v_{j_-}>1,\\[2pt]
       \gamma_{\max}\ \text{(drop)}&\text{if }v_{j_-}\le1;\end{cases}$\quad
       $\mathbf{p}\gets(1+\gamma)\mathbf{p}-\gamma\,\mathbf{e}_{j_-}$
  \EndIf
\EndWhile
\State $S\gets\supp(\mathbf{p})$ \Comment{$S=S^\opt$ at exit (\Cref{lem:identify})}
\If{facial phase $=$ \textsc{Accel}}
  \State $\mathbf{p}\gets\textsc{RestartedFISTA}\big(h,\,\Delta_{S},\,\mathbf{p},\,L_{\mathrm{loc}},\,\text{stop when }g\le\eps d\big)$
  \Comment{\Cref{thm:upper}; \Cref{app:fista}}
\Else
  \State $\mathbf{p}\gets\textsc{FacialNewton}\big(h,\,\mathbf{p},\,\text{stop when }g\le\eps d\big)$
  \Comment{\Cref{alg:newton}; \Cref{thm:newton}}
\EndIf
\State \Return $\mathbf{p}$
\end{algorithmic}
\end{algorithm}

\begin{remark}[Instance parameters: advised-complexity statements]
\label{rem:params}
Like the away-step guarantees it builds on (\Cref{thm:zhao,thm:bomze}), \Cref{alg:main} consumes
instance constants: the tolerance $c\le c_0$ (a function of $\gsc,\pmin,\muopt$;
\Cref{def:c0}) and the slack $\gsc$ used in the stopping test, plus the local smoothness
$L_{\mathrm{loc}}$ as the step size of the accelerated phase (which moreover requires the stronger
tolerance $c\le c_1/\kappa^3$ of \Cref{thm:upper}). These affect only \emph{when} the
algorithm hands off between phases, not the correctness of its output: the final certificate
$\max_iv_i(\mathbf{p})\le(1+\eps)d$ is checked directly from the oracle, so an invalid guess can delay
termination but never produce an incorrect output. Accordingly, all complexity theorems in this
paper are \emph{advised} (instance-constant) statements: the stated query counts presuppose valid
constants. We expect standard doubling/backtracking schemes to remove the advice at logarithmic
overhead (one must verify that wrong guesses neither leave the domain $\mathbf{M}(\mathbf{w})\succ0$ nor make a
premature handoff unrecoverable, both of which the direct certificate renders detectable), but we do
not carry out this adaptive analysis here.
\end{remark}

The step sizes in \Cref{alg:main} are the \emph{exact} line-search minimizers of $f$, including the
boundary (drop) case of the away step; this is what licenses importing the away-step Frank--Wolfe
guarantees in \Cref{lem:identify} below. We record the closed forms, with the case distinction that
the away step requires. (Omitting the $v_{j_-}\le1$ case, as an earlier version of this paper
did, would prescribe a \emph{negative} step there: e.g.\ for $\mathbf{a}_1=(1,0)$, $\mathbf{a}_2=(0,1)$,
$\mathbf{a}_3=(\sqrt{0.1},0)$ at $\mathbf{p}=(\tfrac13,\tfrac13,\tfrac13)$, the away atom is $j_-=3$ with
$v_3=\tfrac3{11}<1$, and the interior formula evaluates to a negative number, while the true
line-search optimum is the drop step.)

\begin{lemma}[Closed-form exact line search]
\label{lem:linesearch}
Let $\mathbf{p}\in\Delta_n$ with $\mathbf{M}(\mathbf{p})\succ0$ and $\mathbf{v}=\mathbf{v}(\mathbf{p})$.
\begin{enumerate}[label=\emph{(\alph*)},leftmargin=*,itemsep=2pt]
\item \emph{(Toward.)} If $v_j>d$, then $\gamma\mapsto f\big((1-\gamma)\mathbf{p}+\gamma\mathbf{e}_j\big)$ is
strictly convex on $[0,1)$ and is minimized at
$\gamma^\star=\frac{v_j-d}{d\,(v_j-1)}\in(0,1)$.
\item \emph{(Away/drop.)} If $j\in\supp(\mathbf{p})$ with $p_j<1$ and $v_j<d$, then with
$\gamma_{\max}:=\frac{p_j}{1-p_j}$ the function
$\gamma\mapsto f\big((1+\gamma)\mathbf{p}-\gamma\mathbf{e}_j\big)$ is strictly convex on $[0,\gamma_{\max}]$
and is minimized at
\[
\gamma^\star=
\begin{cases}
\min\Big\{\dfrac{d-v_j}{d\,(v_j-1)},\ \gamma_{\max}\Big\}, & v_j>1,\\[8pt]
\gamma_{\max}\quad(\text{a drop step: the new weight of atom }j\text{ is }0), & v_j\le1 .
\end{cases}
\]
\end{enumerate}
In both cases $\mathbf{M}\succ0$ along the whole segment traversed.
\end{lemma}
\begin{proof}
Write $\mathbf{M}_\gamma$ for the information matrix of the moved point and $\varphi(\gamma):=f$ along the
line. \emph{(a)} $\mathbf{M}_\gamma=(1-\gamma)\mathbf{M}(\mathbf{p})+\gamma\,\mathbf{a}_j\mathbf{a}_j^\top$, and the matrix determinant
lemma gives
$\det\mathbf{M}_\gamma=\det\mathbf{M}(\mathbf{p})\,(1-\gamma)^{d-1}\big(1+\gamma(v_j-1)\big)$,
which is positive for $\gamma\in[0,1)$ (as $v_j>d\ge1$), so $\mathbf{M}_\gamma\succ0$ there (its
eigenvalues, continuous in $\gamma$ and initially positive, cannot vanish while the determinant is
positive). Hence
\[
\varphi'(\gamma)=\frac{d-1}{1-\gamma}-\frac{v_j-1}{1+\gamma(v_j-1)},\qquad
\varphi''(\gamma)=\frac{d-1}{(1-\gamma)^2}+\frac{(v_j-1)^2}{\big(1+\gamma(v_j-1)\big)^2}>0 .
\]
$\varphi$ is strictly convex with $\varphi'(0)=d-v_j<0$, and solving $\varphi'(\gamma)=0$ gives
$\gamma^\star=\frac{v_j-d}{d(v_j-1)}$, which lies in $(0,1)$ since $v_j>d$ implies
$v_j-d<d(v_j-1)$.

\emph{(b)} Now $\mathbf{M}_\gamma=(1+\gamma)\mathbf{M}(\mathbf{p})-\gamma\,\mathbf{a}_j\mathbf{a}_j^\top$ and
$\det\mathbf{M}_\gamma=\det\mathbf{M}(\mathbf{p})\,(1+\gamma)^{d-1}\big(1+\gamma(1-v_j)\big)$, so
\[
\varphi'(\gamma)=-\frac{d-1}{1+\gamma}-\frac{1-v_j}{1+\gamma(1-v_j)},\qquad
\varphi''(\gamma)=\frac{d-1}{(1+\gamma)^2}+\frac{(1-v_j)^2}{\big(1+\gamma(1-v_j)\big)^2}>0 .
\]
\emph{Case $v_j\le1$:} both terms of $-\varphi'$ are nonnegative and the first is positive, so
$\varphi'<0$ on all of $[0,\gamma_{\max}]$ (where $1+\gamma(1-v_j)\ge1>0$, so $\det\mathbf{M}_\gamma>0$
and $\mathbf{M}_\gamma\succ0$ as above): the minimizer is the right endpoint $\gamma_{\max}$, at which the
weight of atom $j$ is $(1+\gamma_{\max})p_j-\gamma_{\max}=0$.
\emph{Case $v_j>1$:} solving $\varphi'(\gamma)=0$ gives the stationary point
$\gamma^\star_{\mathrm{int}}=\frac{d-v_j}{d(v_j-1)}>0$ (positive since $1<v_j<d$). The determinant
factor $1+\gamma(1-v_j)$ vanishes only at $\gamma_{\mathrm{sing}}=\frac1{v_j-1}$, and
$\gamma^\star_{\mathrm{int}}<\gamma_{\mathrm{sing}}$ always (equivalent to $v_j>0$), so on
$[0,\min\{\gamma_{\max},\gamma_{\mathrm{sing}}\})$ the function $\varphi$ is strictly convex with
$\varphi'(0)=v_j-d<0$ and $\varphi\to+\infty$ at $\gamma_{\mathrm{sing}}$ if
$\gamma_{\mathrm{sing}}\le\gamma_{\max}$; in either configuration the constrained minimizer is
$\min\{\gamma^\star_{\mathrm{int}},\gamma_{\max}\}$, and the segment up to it keeps
$\det\mathbf{M}_\gamma>0$, hence $\mathbf{M}_\gamma\succ0$.
\end{proof}

\subsection{Phases 1--2: reaching the facial sublevel set}

\begin{lemma}[Warm start and identification]
\label{lem:identify}
Fix $c\in(0,c_0]$. The while-loop of \Cref{alg:main} terminates after at most
\[
C(\mathbf{A})\;:=\;\max\Big\{K_{\mathrm{id}},\ K_{\mathrm{gap}}(c/d)\Big\}\;=\;
\max\Big\{K_{\mathrm{id}},\ O\big(\kappa_\Phi\,\poly(d)\,\log(d/c)\big)\Big\}
\]
leverage-score queries ($K_{\mathrm{id}}$ from \Cref{thm:bomze}, $K_{\mathrm{gap}}$ from
\Cref{thm:zhao}; both finite and independent of $\eps$), and at exit the iterate satisfies
\[
\supp(\mathbf{p})=S^\opt
\qquad\text{and}\qquad
\mathbf{p}\in\{\mathbf{q}\in V: h(\mathbf{q})\le f^\opt+c\}\subseteq\Cnull .
\]
\end{lemma}
\begin{proof}
\emph{Termination.} The loop body performs exactly the away-step Frank--Wolfe (Wolfe--Atwood) steps
with exact line search, in the form analyzed by~\cite{todd2007,zhao2023,bomze2020}: by
\Cref{lem:linesearch}, the displayed steps are precisely the line-search minimizers of $f$ over the
feasible segment, including the boundary (drop) case of the away step (cf.~\cite[Ch.~3]{todd2016}). By
\Cref{thm:bomze} there is a finite $K_{\mathrm{id}}$ with $\supp(\mathbf{p}^{(k)})\subseteq S^\opt$ for all
$k\ge K_{\mathrm{id}}$, and by \Cref{thm:zhao} $g(\mathbf{p}^{(k)})\le c$ for all
$k\ge K_{\mathrm{gap}}(c/d)$. Fix any $k\ge\max\{K_{\mathrm{id}},K_{\mathrm{gap}}(c/d)\}$ at which the
loop is still running, and write $\mathbf{p}=\mathbf{p}^{(k)}$. Then $g(\mathbf{p})\le c$, so the first clause of the loop
condition fails; moreover $f(\mathbf{p})-f^\opt\le g(\mathbf{p})\le c\le c_0$ (\Cref{prop:dict}\ref{it:fw}), so
\Cref{lem:c0}\ref{it:c0-slack} applies: every $i\in S^\opt$ has $v_i(\mathbf{p})\ge d-\gsc/2$. Since
$\supp(\mathbf{p})\subseteq S^\opt$, every supported atom has $v_i(\mathbf{p})\ge d-\gsc/2$, so the second clause
fails as well and the loop exits. Hence the loop runs at most
$\max\{K_{\mathrm{id}},K_{\mathrm{gap}}(c/d)\}$ iterations, one query each.

\emph{Exit invariants.} At exit, the failed first clause gives $g(\mathbf{p})\le c$, hence
$f(\mathbf{p})-f^\opt\le c\le c_0$. The failed second clause gives $v_i(\mathbf{p})\ge d-\gsc/2$ for every supported
$i$; by the \emph{strict} inequality in \Cref{lem:c0}\ref{it:c0-slack}, every $j\notin S^\opt$ has
$v_j(\mathbf{p})<d-\gsc/2$, so no supported atom lies outside $S^\opt$: $\supp(\mathbf{p})\subseteq S^\opt$, i.e.\
$\mathbf{p}\in V$ and $h(\mathbf{p})-f^\opt\le c$. Thus $\mathbf{p}\in\Cnull$, and \Cref{lem:c0}\ref{it:c0-interior}
upgrades the inclusion to equality: $p_i\ge\pmin/2>0$ for all $i\in S^\opt$, so $\supp(\mathbf{p})=S^\opt$.
\end{proof}

\begin{remark}[How large is $C(\mathbf{A})$?]
\label{rem:CA}
$C(\mathbf{A})$ is $\eps$-independent but \emph{condition-dependent}. Making the two terms explicit with the
constants of~\cite{zhao2023} (a global linear gap rate $1/\rho=O(\kappa_\Phi\,\poly(d))$ for the
$-\log\det$ barrier, holding for \emph{all} iterates, plus at most $n-|S^\opt|$ drop steps to clear the
non-contact atoms) gives
\[
C(\mathbf{A})\;=\;\underbrace{O\big(\kappa_\Phi\,\poly(d)\,\log(d/c)\big)}_{\text{gap}\,\le\,c\ \text{eventually-always}}
\;+\;\underbrace{\big(n-|S^\opt|\big)}_{\text{identification drops}} ,
\]
finite for every fixed nondegenerate instance but \emph{not} uniformly bounded:
$\kappa_\Phi=1/(\mu_R\Phi^2)$ blows up as the optimal design approaches a lower-dimensional face. Every
result below quarantines all condition dependence inside $C(\mathbf{A})$ and inside (double) logarithms;
whether identification can be made condition-free is \Cref{prob:uniform}. Two accounting notes:
(i) \Cref{thm:wa} bounds only the \emph{first} time the gap dips below $c$, by $O(d^2/c+d\log d)$,
whereas the loop needs the \emph{eventually-always} gap bound of \Cref{thm:zhao} (both stopping clauses
must hold simultaneously); (ii) the general active-set identification theory~\cite{bomze2020}
presupposes a globally Lipschitz gradient, so for the boundary-singular $-\log\det$ barrier we take
identification from the in-setting analysis of~\cite{zhao2023}. The $\poly(d)$ factor is the one tracked
in \Cref{thm:domination}.
\end{remark}

\subsection{Phase 3a: restarted FISTA on the face}

The accelerated phase runs FISTA~\cite{beck2009} on the composite problem
$\min_{\mathbf{p}\in V}\,h(\mathbf{p})+\iota_{\Delta_{S^\opt}}(\mathbf{p})$ (projection onto the face simplex as the
proximal step, step size $1/L_{\mathrm{loc}}$), restarted every $\lceil\sqrt{8\kappa_{\mathrm{loc}}}\,
\rceil-1$ iterations from the current iterate. Two points need proof beyond the textbook statement,
and both are handled in \Cref{app:fista}: (i) $h$ is neither defined nor smooth on all of
$\Delta_{S^\opt}$ (it blows up where $\mathbf{M}(\mathbf{p})$ degenerates), so the FISTA potential inequality must
be localized: we prove, by induction along the trajectory, that all evaluation points stay in the
facial Dikin ball $\bar{\mathcal D}$, where \Cref{lem:c0}\ref{it:c0-cond} supplies the constants
$L_{\mathrm{loc}},\mu_{\mathrm{loc}}$, provided the phase is started at tolerance
$c\le c_1/\kappa^{3}$ for an explicit absolute constant $c_1$; and (ii) restarting converts the
$O(L\norm{\mathbf{x}_0-\mathbf{x}^\opt}^2/k^2)$ FISTA guarantee into geometric decay with factor $\tfrac12$ per
cycle of $O(\sqrt\kappa)$ iterations, via strong convexity. The result:

\begin{theorem}[Warm-started accelerated algorithm; advised constants per \Cref{rem:params}]
\label{thm:upper}
Let $\mathbf{A}$ have a nondegenerate optimal design and set
$c:=\min\{c_0,\ c_1/\kappa^{3}\}$ in \Cref{alg:main} ($c_1$ the absolute constant of
\Cref{lem:fistaloc}). With the \textsc{Accel} facial phase, \Cref{alg:main} returns a
$(1+\eps)$-John ellipsoid using
\[
C(\mathbf{A})\ +\ O\Big(\sqrt{\kappa}\,\log\frac{\Theta_{\mathrm A}}{\eps}\Big)
\quad\text{leverage-score queries,}\qquad
\Theta_{\mathrm A}:=\frac{\sqrt{L_{\mathrm{loc}}\,\kappa_{\mathrm{loc}}\,c_0}}{d}+2 ,
\]
with $C(\mathbf{A})$ the $\eps$-independent count of \Cref{lem:identify}. All queries are made at weight
vectors $\mathbf{w}$ supported on $S^\opt$ with $\mathbf{M}(\mathbf{w})\succ0$---the extrapolation points may carry
small negative entries, in contrast to the Newton phase of \Cref{sec:newton}, which queries genuine
designs only; the returned point is a design in
$\relint\Delta_{S^\opt}$.
\end{theorem}
\begin{proof}[Proof (assembly; details in \Cref{app:fista})]
By \Cref{lem:identify}, after $C(\mathbf{A})$ queries the iterate $\mathbf{p}^{(0)}$ satisfies
$\supp(\mathbf{p}^{(0)})=S^\opt$ and $h(\mathbf{p}^{(0)})-f^\opt\le c\le c_1/\kappa^3$.
\Cref{lem:fistaloc} (localized FISTA) shows each restart cycle of
$\lceil\sqrt{8\kappa_{\mathrm{loc}}}\rceil-1=O(\sqrt\kappa)$ iterations at least halves the objective
gap $h-f^\opt$ while keeping every iterate and evaluation point inside $\bar{\mathcal D}$ (one query
per iteration, at weights on $S^\opt$ with $\mathbf{M}\succ0$); the cycle-end iterates lie in $\Cnull$.
After $j$ cycles the gap is at most $2^{-j}c$, and by the gap conversion (\Cref{lem:gapconv}) the
cycle-end design $\mathbf{p}$ satisfies
\[
g(\mathbf{p})\;\le\;\sqrt{2L_{\mathrm{loc}}\kappa_{\mathrm{loc}}\,2^{-j}c}\;\le\;\eps d
\qquad\text{once}\qquad
j\;\ge\;\log_2\frac{2L_{\mathrm{loc}}\kappa_{\mathrm{loc}}c}{(\eps d)^2}
\;=\;O\Big(\log\frac{\Theta_{\mathrm A}}{\eps}\Big).
\]
The stopping test $g\le\eps d$ is evaluated from one extra query at each cycle end (the certificate is
direct: $\max_iv_i\le(1+\eps)d$ is precisely \Cref{def:je}). Totals: $O(\sqrt\kappa)$ queries per
cycle times $O(\log(\Theta_{\mathrm A}/\eps))$ cycles. Feasibility of the output: cycle-end iterates
lie in $\Cnull\subseteq\relint\Delta_{S^\opt}$ (\Cref{lem:c0}\ref{it:c0-interior}).
\end{proof}

\begin{remark}[Scope]
\label{rem:scope}
(i) The $\eps$-dependence is $\sqrt\kappa\log(1/\eps)$; the price is that $C(\mathbf{A})$ also absorbs the
$\kappa^3$-dependent warm-start tolerance through $K_{\mathrm{gap}}(c/d)=O(\kappa_\Phi\,\poly(d)\,\log(d\kappa/c_1))$, still $\eps$-independent. (ii) Acceleration buys $\sqrt{\kappa}$ versus $\kappa_\Phi$;
\Cref{thm:domination} next makes the comparison uniform over instances. (iii) On the explicit instance
$\mathbf{A}_\opt$, $\kappa=2$ and $C(\mathbf{A}_\opt)=O(1)$, giving the $O(\log(1/\eps))$ behavior of
\Cref{cor:sep}.
\end{remark}

\subsection{A uniform comparison with away-step Frank--Wolfe}
\label{sec:domination}

The comparison to the away-step rate $\kappa_\Phi\log(1/\eps)$
of \Cref{thm:zhao} can be made \emph{uniform}. Recall $\kappa_\Phi=1/(\mu_R\Phi^2)$
(\Cref{sec:prelim}), where $\Phi=\Phi(\mathbf{A})$ is the facial distance~\cite{pena2019} of the design
polytope $\mathrm{conv}\{\mathbf{a}_i\mathbf{a}_i^\top:i\in[n]\}$ in the local norm at
$\mathbf{M}^\opt$, and $\mu_R\le\tfrac12$. Put
$\mathbf{c}_i:=(\mathbf{M}^\opt)^{-1/2}\mathbf{a}_i\mathbf{a}_i^\top(\mathbf{M}^\opt)^{-1/2}$. Then $\sum_i p^\opt_i\mathbf{c}_i=\mathbf{I}_d$ and
$\Tr\mathbf{c}_i=\fnorm{\mathbf{c}_i}=v_i(\mathbf{p}^\opt)$ (\Cref{prop:dict}, \Cref{thm:kw}), and by \Cref{eq:rho-id} the
facial Hessian block at $\mathbf{p}^\opt$ is the Gram matrix of $\{\mathbf{c}_i\}_{i\in S^\opt}$: for
$\boldsymbol{\delta}\in\R^{S^\opt}$, $\boldsymbol{\delta}^\top\widehat{\mathbf{H}}(\mathbf{p}^\opt)\boldsymbol{\delta}=\fnorm{\mathbf{G}(\boldsymbol{\delta})}^2$ for the
synthesis map $\mathbf{G}(\boldsymbol{\delta}):=\sum_{i\in S^\opt}\delta_i\mathbf{c}_i$. In particular
$\muopt=\min\{\fnorm{\mathbf{G}(\boldsymbol{\delta})}^2:\boldsymbol{\delta}\in T^\opt,\norm{\boldsymbol{\delta}}_2=1\}$ and
$\Lopt=\max\{\fnorm{\mathbf{G}(\boldsymbol{\delta})}^2:\boldsymbol{\delta}\in\R^{S^\opt},\norm{\boldsymbol{\delta}}_2=1\}$.

\begin{lemma}[Facial coupling]
\label{lem:coupling}
Let $\mathbf{p}^\opt$ satisfy \Cref{def:nondeg} and let
$\sigma_{\min}:=\min\{\fnorm{\mathbf{G}(\boldsymbol{\delta})}:\boldsymbol{\delta}\in T^\opt,\norm{\boldsymbol{\delta}}_2=1\}=\sqrt{\muopt}$.
Then $\Phi\le2\,\sigma_{\min}$; equivalently, $\muopt\ge\Phi^2/4$.
\end{lemma}
\begin{proof}[Proof idea (full proof in \Cref{app:coupling})]
By \Cref{thm:kw}, $v_i(\mathbf{p}^\opt)\le d$ with equality exactly on $S^\opt$, so the trace functional
$\mathbf{X}\mapsto\inner{\mathbf{X},\mathbf{I}}$ exposes $\mathrm{conv}\{\mathbf{c}_i:i\in S^\opt\}$ as a genuine \emph{face} of the design
polytope---a simplex, by the linear independence in \Cref{def:nondeg}. Take a unit zero-sum witness
$\boldsymbol{\delta}$ with $\fnorm{\mathbf{G}(\boldsymbol{\delta})}=\sigma_{\min}$, and split it into nonnegative parts
$\boldsymbol{\delta}=\boldsymbol{\delta}^+-\boldsymbol{\delta}^-$. With $s=\sum_i\delta_i^+=\tfrac12\norm{\boldsymbol{\delta}}_1\ge\tfrac12$, the points
$s^{-1}\sum_i\delta_i^+\mathbf{c}_i$ and $s^{-1}\sum_i\delta_i^-\mathbf{c}_i$ lie in disjoint sub-faces and are
$\fnorm{\mathbf{G}(\boldsymbol{\delta})}/s\le2\sigma_{\min}$ apart; the facial distance, a minimum over faces, is at most
this separation. The exposed-face property (which fails without \Cref{thm:kw}) is what makes an internal
near-dependence visible to $\Phi$.
\end{proof}

\begin{theorem}[Uniform comparison with the away-step rate]
\label{thm:domination}
Under \Cref{def:nondeg}, the facial condition number at the optimum satisfies
\[
\kappa\;=\;\frac{\Lopt}{\muopt}\;\le\;\frac{4\binom{d+1}{2}d^2}{\Phi^2}
\;\le\;2\binom{d+1}{2}d^2\,\kappa_\Phi\;=\;O(d^4)\,\kappa_\Phi .
\]
Consequently, after the same condition-dependent setup, the accelerated accuracy phase improves the
away-step dependence from linear in $\kappa_\Phi$ (\Cref{thm:zhao}) to $\sqrt\kappa$ with
$\kappa=O(d^4)\kappa_\Phi$: up to the $\eps$-independent term $C(\mathbf{A})$ and a $\poly(d)$ factor, the
query count of \Cref{alg:main} is never worse, and is quadratically better whenever
$\kappa\asymp\kappa_\Phi$.
\end{theorem}
\begin{proof}
\emph{Lower constant.} \Cref{lem:coupling} gives $\muopt\ge\Phi^2/4$.
\emph{Upper constant.} For any unit $\boldsymbol{\delta}\in\R^{S^\opt}$, by Cauchy--Schwarz,
\[
\fnorm{\mathbf{G}(\boldsymbol{\delta})}=\Big\lVert\sum_{i\in S^\opt}\delta_i\mathbf{c}_i\Big\rVert_F
\le\norm{\boldsymbol{\delta}}_2\Big(\sum_{i\in S^\opt}\fnorm{\mathbf{c}_i}^2\Big)^{1/2},
\quad\text{so}\quad
\Lopt\le\sum_{i\in S^\opt}\fnorm{\mathbf{c}_i}^2\le m\,d^2\le\binom{d+1}{2}d^2,
\]
using $\fnorm{\mathbf{c}_i}=v_i(\mathbf{p}^\opt)=d$ for $i\in S^\opt$ (each $\mathbf{c}_i$ is rank one, so its Frobenius
norm equals its trace) and $m\le\binom{d+1}{2}$ (\Cref{def:nondeg}).
\emph{Assembly.} $\kappa=\Lopt/\muopt\le\binom{d+1}{2}d^2\cdot4/\Phi^2$, and the definition
$\kappa_\Phi=1/(\mu_R\Phi^2)$ (\Cref{sec:prelim}) with $\mu_R\le\tfrac12$ gives
$1/\Phi^2=\mu_R\kappa_\Phi\le\kappa_\Phi/2$. Note both constants
are evaluated \emph{at} $\mathbf{p}^\opt$ (this is exactly how $\kappa$ is defined,
\Cref{def:kappa}), and the algorithm's working constants differ from them only by the absolute
factor $9$ (\Cref{lem:c0}\ref{it:c0-cond}), so no instance-dependent slack is hidden in the
comparison. The exponent of $d$ is loose (numerically $C(d)=O(d)$).
\end{proof}

\begin{corollary}[Separation]
\label{cor:sep}
On $\mathbf{A}_\opt$ (\Cref{sec:instance}), which is nondegenerate with $\kappa=2$ (and full-simplex
reduced-Hessian condition number $O(1)$, since $n-1=\binom{d+1}{2}$ leaves no flat directions),
\Cref{alg:main} returns a $(1+\eps)$-John ellipsoid in $O(\log(1/\eps))$ queries with the
\textsc{Accel} phase, and in $O(\log\log(1/\eps))$ Newton iterations with the \textsc{Newton} phase
(\Cref{thm:newton}), while every uniform-averaging certificate needs $\Theta(\eps^{-1})$ queries
(\Cref{thm:avg}). The averaged-iterate and last-iterate complexities of the John ellipsoid are
exponentially separated.
\end{corollary}

\section{The facial Newton phase: a condition-free accuracy dependence}
\label{sec:newton}

This is the conceptual peak of the paper. On its smooth stratum (\Cref{sec:designs}) the identified
problem is an \emph{unconstrained} self-concordant minimization, and Newton's method is what that
geometry asks for: affine-invariant, with no condition number in its rate. The accuracy cost accordingly
drops from the accelerated phase's $\log(1/\eps)$ (\Cref{sec:upper}) to $\log\log(1/\eps)$, every
condition number confined inside the double logarithm. The one thing that might obstruct it, extracting
the facial Hessian from a first-order oracle, instead dissolves: a single rank-one identity
(\Cref{prop:smrecover}) recovers it \emph{exactly} from leverage-score queries, so the second-order
method runs in the first-order oracle.

This section proves \Cref{thm:newton-intro}. Recall from \Cref{sec:face} that $\mathbf{p}^\opt$ is the
\emph{unconstrained} minimizer of $h=f|_D$ over the affine hull $D$ of the optimal face
(\Cref{lem:min}); the analysis recovers its Hessian from the oracle (\Cref{prop:smrecover}), runs
damped Newton on it (\Cref{alg:newton}), and turns the Newton decrement into the John guarantee
(\Cref{thm:newton}).

Throughout this section, for $\mathbf{p}\in D$ and $\boldsymbol{\delta}\in T^\opt$ we write
$\norm{\boldsymbol{\delta}}_{\mathbf{p}}:=\big(\boldsymbol{\delta}^\top\widehat{\mathbf{H}}(\mathbf{p})\,\boldsymbol{\delta}\big)^{1/2}
=\locnorm{\sum_{i\in S^\opt}\delta_i\mathbf{a}_i\mathbf{a}_i^\top}{\mathbf{M}(\mathbf{p})}$ for the local norm of $h$ at $\mathbf{p}$
(a genuine norm by \Cref{lem:faceSC}); $\varrho(\cdot)=\norm{\cdot-\mathbf{p}^\opt}_{\mathbf{p}^\opt}$ as before.

\subsection{Exact Hessian recovery from leverage-score queries}

The recovery rests on the standard rank-one inverse-update identity, restated here for
self-containedness.

\begin{fact}[Sherman--Morrison--Woodbury identity, rank-one case; see, e.g., {\cite[\S3.1, App.~A.3]{todd2016}}]
\label{fact:smw}
For $\mathbf{N}\succ0$, $\mathbf{a}\in\R^d$, and $t>0$,
\[
\big(\mathbf{N}+t\,\mathbf{a}\mathbf{a}^\top\big)^{-1}
=\mathbf{N}^{-1}-\frac{t\,\mathbf{N}^{-1}\mathbf{a}\mathbf{a}^\top\mathbf{N}^{-1}}{1+t\,\mathbf{a}^\top\mathbf{N}^{-1}\mathbf{a}}.
\]
\end{fact}

\begin{proposition}[Rank-one probes recover the facial Hessian exactly]
\label{prop:smrecover}
Let $\mathbf{w}$ be any weights with $\mathbf{M}(\mathbf{w})\succ0$, let $j\in[n]$ and $t>0$. Then
\begin{equation}
\label{eq:smrecover}
\big(\mathbf{a}_i^\top\mathbf{M}(\mathbf{w})^{-1}\mathbf{a}_j\big)^2
=\Big(v_i(\mathbf{w})-v_i(\mathbf{w}+t\mathbf{e}_j)\Big)\cdot\frac{1+t\,v_j(\mathbf{w})}{t}
\qquad\text{for every }i\in[n].
\end{equation}
Consequently, for a design $\mathbf{p}\in\Delta_n$ the facial Hessian block
$\widehat{\mathbf{H}}(\mathbf{p})=\big((\mathbf{a}_i^\top\mathbf{M}(\mathbf{p})^{-1}\mathbf{a}_j)^2\big)_{i,j\in S^\opt}$ is computed
\emph{exactly} from the $m+1\le\binom{d+1}{2}+1$ leverage-score queries at the designs
\[
\mathbf{p}\quad\text{and}\quad\tfrac12(\mathbf{p}+\mathbf{e}_j),\ j\in S^\opt,
\]
via \eqref{eq:smrecover} with $t=1$ and the homogeneity $v_i(\mathbf{w}/2)=2v_i(\mathbf{w})$.
\end{proposition}
\begin{proof}
Since $\mathbf{M}(\mathbf{w}+t\mathbf{e}_j)=\mathbf{M}(\mathbf{w})+t\,\mathbf{a}_j\mathbf{a}_j^\top\succ0$, \Cref{fact:smw} (with $\mathbf{N}=\mathbf{M}(\mathbf{w})$ and $\mathbf{a}=\mathbf{a}_j$) gives
\[
\big(\mathbf{M}(\mathbf{w})+t\,\mathbf{a}_j\mathbf{a}_j^\top\big)^{-1}
=\mathbf{M}(\mathbf{w})^{-1}-\frac{t\,\mathbf{M}(\mathbf{w})^{-1}\mathbf{a}_j\mathbf{a}_j^\top\mathbf{M}(\mathbf{w})^{-1}}{1+t\,\mathbf{a}_j^\top\mathbf{M}(\mathbf{w})^{-1}\mathbf{a}_j}.
\]
Sandwiching by $\mathbf{a}_i^\top(\cdot)\mathbf{a}_i$ yields
$v_i(\mathbf{w}+t\mathbf{e}_j)=v_i(\mathbf{w})-t(\mathbf{a}_i^\top\mathbf{M}(\mathbf{w})^{-1}\mathbf{a}_j)^2/(1+tv_j(\mathbf{w}))$, which rearranges to
\eqref{eq:smrecover}. For the second claim: $\tfrac12(\mathbf{p}+\mathbf{e}_j)$ is a convex combination of two
designs, hence a design; by homogeneity, $v_i(\mathbf{p}+\mathbf{e}_j)=\tfrac12\,v_i\big(\tfrac12(\mathbf{p}+\mathbf{e}_j)\big)$,
so the query at $\tfrac12(\mathbf{p}+\mathbf{e}_j)$ delivers $v_i(\mathbf{p}+\mathbf{e}_j)$ for all $i$, and \eqref{eq:smrecover}
with $t=1$ delivers the column $j$ of $\widehat{\mathbf{H}}(\mathbf{p})$ (its entries for all $i\in S^\opt$
simultaneously). The recovery is an algebraic identity---exact, not approximate.
\end{proof}

\begin{remark}
The same $n+1$ design queries would deliver the \emph{entire} $n\times n$ Hessian
$\nabla^2f(\mathbf{p})$; the point of the facial phase is that only the $m\times m$ contact block (of size
at most $\binom{d+1}{2}$, independent of $n$) is ever needed. The probes can of course also be
answered directly from $\mathbf{A}$ in $O(md^2+d^\omega+m^2d)$ arithmetic without extra oracle calls
(\Cref{prop:cost}); \Cref{prop:smrecover} is what makes the phase well defined in the \emph{pure query
model}, so that the open problems of \Cref{sec:discussion} are statements about one oracle.
\end{remark}

\subsection{The algorithm}

\begin{algorithm}[t]
\caption{\textsc{FacialNewton} (damped Newton on the affine hull of the optimal face)}
\label{alg:newton}
\begin{algorithmic}[1]
\Require $\mathbf{p}\in\relint\Delta_{S^\opt}$ with $h(\mathbf{p})-f^\opt\le c_0$ (from \Cref{lem:identify});
accuracy $\eps$
\Ensure $\mathbf{p}\in\relint\Delta_{S^\opt}$ with $\max_iv_i(\mathbf{p})\le(1+\eps)d$
\Loop
  \State $\mathbf{v}\gets(v_i(\mathbf{p}))_{i\in[n]}$ \Comment{one design query; $\nabla h(\mathbf{p})|_{S^\opt}=-\mathbf{v}_{S^\opt}$}
  \If{$\max_iv_i-d\le\eps d$} \Return $\mathbf{p}$ \Comment{the John guarantee, certified directly}
  \EndIf
  \State $\widehat{\mathbf{H}}\gets$ exact facial Hessian via \Cref{prop:smrecover}
         \Comment{$m$ design queries at $\tfrac12(\mathbf{p}+\mathbf{e}_j)$, $j\in S^\opt$}
  \State solve $\begin{pmatrix}\widehat{\mathbf{H}}&\mathbf 1\\ \mathbf 1^\top&0\end{pmatrix}
         \begin{pmatrix}\mathbf{n}\\ \nu\end{pmatrix}
         =\begin{pmatrix}\mathbf{v}_{S^\opt}\\ 0\end{pmatrix}$,
         \quad $\lambda\gets\big(\mathbf{v}_{S^\opt}^\top\mathbf{n}\big)^{1/2}$
         \Comment{Newton direction $\mathbf{n}\in T^\opt$ and decrement; $O(m^3)$ arithmetic}
  \State $\mathbf{p}\gets\mathbf{p}+\dfrac{\mathbf{n}}{1+\lambda}$ \textbf{ if } $\lambda>\tfrac14$
         \textbf{ else } $\mathbf{p}\gets\mathbf{p}+\mathbf{n}$
         \Comment{damped step, then full steps}
\EndLoop
\end{algorithmic}
\end{algorithm}

The linear system in \Cref{alg:newton} is the KKT system of
$\min_{\boldsymbol{\delta}\in T^\opt}\{\inner{\nabla h(\mathbf{p}),\boldsymbol{\delta}}+\tfrac12\boldsymbol{\delta}^\top\widehat{\mathbf{H}}(\mathbf{p})\boldsymbol{\delta}\}$:
since $\nabla h(\mathbf{p})|_{S^\opt}=-\mathbf{v}_{S^\opt}(\mathbf{p})$, its solution $\mathbf{n}=\mathbf{n}_{\mathbf{p}}$ is the
\emph{descent} Newton direction, the system is nonsingular ($\widehat{\mathbf{H}}\succ0$ on
$T^\opt=\ker\mathbf 1^\top$ by \Cref{lem:faceSC}), and
\[
\lambda(\mathbf{p})^2:=\mathbf{n}_{\mathbf{p}}^\top\widehat{\mathbf{H}}(\mathbf{p})\,\mathbf{n}_{\mathbf{p}}
=\mathbf{n}_{\mathbf{p}}^\top\big(\mathbf{v}_{S^\opt}-\nu\mathbf 1\big)=\mathbf{v}_{S^\opt}^\top\mathbf{n}_{\mathbf{p}}
\]
(using $\mathbf 1^\top\mathbf{n}_{\mathbf{p}}=0$) is the squared \emph{Newton decrement}
$\lambda(\mathbf{p})=\norm{\mathbf{n}_{\mathbf{p}}}_{\mathbf{p}}$, computed exactly from the same data. Equivalently,
$\lambda(\mathbf{p})=\sup\{-\inner{\nabla h(\mathbf{p}),\boldsymbol{\delta}}:\boldsymbol{\delta}\in T^\opt,\norm{\boldsymbol{\delta}}_{\mathbf{p}}\le1\}$ is
the dual local norm of the gradient, whence $|\inner{\nabla h(\mathbf{p}),\boldsymbol{\delta}}|\le\lambda(\mathbf{p})
\norm{\boldsymbol{\delta}}_{\mathbf{p}}$ for all $\boldsymbol{\delta}\in T^\opt$.

\subsection{Damped descent and quadratic contraction}

The analysis rests on $h$ being \emph{standard self-concordant} on $D$: it is the composition of
$-\log\det$ (standard self-concordant on $\Sym^d_{\succ0}$~\cite[\S5.4.4, Lem.~5.4.6,
Thm.~5.4.3]{nesterov2018}, see also~\cite[Ex.~9.5]{bv2004}) with the affine map
$\mathbf{p}\mapsto\mathbf{M}(\mathbf{p})$ restricted to $V$, and affine substitution preserves standard
self-concordance~\cite[Thm.~5.1.2]{nesterov2018}. On such a function Newton's method has
\emph{affine-invariant}, condition-free convergence constants, and under strict complementarity the
simplex constraints are invisible (\Cref{lem:min}(ii)), so the classical theory applies to the facial
problem unobstructed. We therefore import a single classical result, the local quadratic convergence of
the full Newton step, and derive the two other inequalities we need from the toolkit of \Cref{sec:face},
so that their constants are visibly absolute.

\begin{theorem}[{Quadratic contraction of the Newton decrement;
\cite[Thm.~5.2.2(1), eq.~(5.2.6)]{nesterov2018}, \cite[eq.~(9.55)]{bv2004}}]
\label{thm:newtonquad}
Let $h$ be standard self-concordant with nondegenerate Hessian on an open convex domain, and let
$\mathbf{p}$ be a point with $\lambda(\mathbf{p})<1$. Then the full Newton iterate $\mathbf{p}^+=\mathbf{p}+\mathbf{n}_{\mathbf{p}}$ lies in
the domain and
\[
\lambda(\mathbf{p}^+)\;\le\;\Big(\frac{\lambda(\mathbf{p})}{1-\lambda(\mathbf{p})}\Big)^2 .
\]
\end{theorem}

We apply \cref{thm:newtonquad} to $h$ in an affine chart of $V$; every quantity in the statement is affine-invariant.

\begin{lemma}[Damped decrease; suboptimality from the decrement]
\label{lem:newtonstep}
Let $\mathbf{p}\in D$ with decrement $\lambda=\lambda(\mathbf{p})$.
\begin{enumerate}[label=\emph{(\alph*)},leftmargin=*,itemsep=2pt]
\item\label{it:damped} \emph{(Damped step.)} $\mathbf{p}^+:=\mathbf{p}+\mathbf{n}_{\mathbf{p}}/(1+\lambda)\in D$ and
$h(\mathbf{p}^+)\le h(\mathbf{p})-\wsc(\lambda)$.
\item\label{it:fullmono} \emph{(Full step for small decrement.)} If $\lambda\le\tfrac14$, then
$\mathbf{p}^+:=\mathbf{p}+\mathbf{n}_{\mathbf{p}}\in D$ and $h(\mathbf{p}^+)\le h(\mathbf{p})-\tfrac13\lambda^2\le h(\mathbf{p})$.
\item\label{it:subopt} \emph{(Suboptimality.)} If $\lambda<1$, then
$h(\mathbf{p})-f^\opt\le\wscs(\lambda)$; in particular $h(\mathbf{p})-f^\opt\le\tfrac23\lambda^2$ when
$\lambda\le\tfrac14$.
\end{enumerate}
\end{lemma}
\begin{proof}
Write $\mathbf{n}=\mathbf{n}_{\mathbf{p}}$, so $\norm{\mathbf{n}}_{\mathbf{p}}=\lambda$ and
$\inner{\nabla h(\mathbf{p}),\mathbf{n}}=-\lambda^2$.

\ref{it:damped} The step $\boldsymbol{\delta}=\mathbf{n}/(1+\lambda)$ has $\norm{\boldsymbol{\delta}}_{\mathbf{p}}=\lambda/(1+\lambda)<1$,
so \Cref{lem:Msand} (at base $\mathbf{M}(\mathbf{p})$, with
$\mathbf{H}=\mathbf{M}(\mathbf{p}^+)-\mathbf{M}(\mathbf{p})=\sum_i\delta_i\mathbf{a}_i\mathbf{a}_i^\top$, $r=\norm{\boldsymbol{\delta}}_{\mathbf{p}}$) gives
$\mathbf{M}(\mathbf{p}^+)\succeq(1-r)\mathbf{M}(\mathbf{p})\succ0$: $\mathbf{p}^+\in D$. By the upper integrated bound
(\Cref{lem:scineq}\ref{it:scupper}, at base $\mathbf{X}=\mathbf{M}(\mathbf{p})$, $\mathbf{Y}=\mathbf{M}(\mathbf{p}^+)$, pulled back through
the affine map as in \Cref{lem:min}),
\[
h(\mathbf{p}^+)\ \le\ h(\mathbf{p})+\Big\langle\nabla h(\mathbf{p}),\tfrac{\mathbf{n}}{1+\lambda}\Big\rangle
+\wscs\Big(\tfrac{\lambda}{1+\lambda}\Big)
= h(\mathbf{p})-\frac{\lambda^2}{1+\lambda}-\frac{\lambda}{1+\lambda}+\log(1+\lambda)
= h(\mathbf{p})-\wsc(\lambda),
\]
where we evaluated
$\wscs\big(\tfrac{\lambda}{1+\lambda}\big)
=-\tfrac{\lambda}{1+\lambda}-\log\big(1-\tfrac{\lambda}{1+\lambda}\big)
=-\tfrac{\lambda}{1+\lambda}+\log(1+\lambda)$ and combined
$\tfrac{\lambda^2+\lambda}{1+\lambda}=\lambda$.

\ref{it:fullmono} Now $\boldsymbol{\delta}=\mathbf{n}$, $\norm{\boldsymbol{\delta}}_{\mathbf{p}}=\lambda\le\tfrac14<1$, so
$\mathbf{p}^+\in D$ as before, and the same upper bound gives
\[
h(\mathbf{p}^+)\le h(\mathbf{p})-\lambda^2+\wscs(\lambda)\le h(\mathbf{p})-\lambda^2+\tfrac23\lambda^2
=h(\mathbf{p})-\tfrac13\lambda^2,
\]
using $\wscs(\lambda)\le\tfrac23\lambda^2$ for $\lambda\le\tfrac14$ (\Cref{lem:legendre}).

\ref{it:subopt} For any $\mathbf{q}\in D$, the lower integrated bound
(\Cref{lem:scineq}\ref{it:sclower}, base $\mathbf{M}(\mathbf{p})$) reads
$h(\mathbf{q})\ge h(\mathbf{p})+\inner{\nabla h(\mathbf{p}),\mathbf{q}-\mathbf{p}}+\wsc(\norm{\mathbf{q}-\mathbf{p}}_{\mathbf{p}})$. By the dual-norm
characterization of the decrement,
$\inner{\nabla h(\mathbf{p}),\mathbf{q}-\mathbf{p}}\ge-\lambda\norm{\mathbf{q}-\mathbf{p}}_{\mathbf{p}}$, so
\[
h(\mathbf{q})\;\ge\;h(\mathbf{p})+\inf_{r\ge0}\big\{\wsc(r)-\lambda r\big\}
\;=\;h(\mathbf{p})-\sup_{r\ge0}\big\{\lambda r-\wsc(r)\big\}
\;=\;h(\mathbf{p})-\wscs(\lambda),
\]
the last step by the Legendre conjugacy of \Cref{lem:legendre} (finite precisely because
$\lambda<1$). Taking $\mathbf{q}=\mathbf{p}^\opt$ (the minimizer, \Cref{lem:min}) gives
$f^\opt\ge h(\mathbf{p})-\wscs(\lambda)$. The numerical form for $\lambda\le\tfrac14$ is
\Cref{lem:legendre} again.
\end{proof}

\subsection{The main theorem}

\begin{theorem}[Facial Newton phase]
\label{thm:newton}
Let $\mathbf{A}$ have a nondegenerate optimal design, and let \Cref{alg:newton} be started at any
$\mathbf{p}^{(0)}\in\relint\Delta_{S^\opt}$ with $h(\mathbf{p}^{(0)})-f^\opt\le c_0$ (as delivered by
\Cref{lem:identify}, given the instance constants of \Cref{rem:params}; \Cref{alg:newton} itself
consumes no instance constants: its stepping rule is driven by the computed decrement and its
stopping rule by the certificate). Set
\[
\tau\;:=\;\min\Big\{\tfrac14,\ \ \frac{\eps d}{\sqrt{\tfrac43L_{\mathrm{loc}}\kappa_{\mathrm{loc}}}}\Big\}
\;=\;\min\Big\{\tfrac14,\ \ \frac{\eps d}{4\sqrt{3\,\Lopt\,\kappa}}\Big\}.
\]
Then:
\begin{enumerate}[label=\emph{(\roman*)},leftmargin=*,itemsep=2pt]
\item\label{it:nwt-iters} \Cref{alg:newton} returns a point $\mathbf{p}\in\relint\Delta_{S^\opt}$ with
$\max_iv_i(\mathbf{p})\le(1+\eps)d$ after at most
\[
N\;\le\;\underbrace{3}_{\text{damped steps}}
\;+\;\underbrace{\Big\lceil\log_2\log_2\tfrac{1}{2\tau}\Big\rceil^{\!+}}_{\text{full steps}}
\;+\;\underbrace{1}_{\text{terminal test}}
\;=\;O\big(1\big)+O\Big(\log\log\frac{\Theta_{\mathrm N}}{\eps}\Big)
\quad\text{iterations},
\]
where $\Theta_{\mathrm N}:=4\sqrt{3\Lopt\kappa}/d$ and $\lceil x\rceil^+:=\max\{\lceil
x\rceil,0\}$. Every condition-type quantity enters only inside the double logarithm.
\item\label{it:nwt-queries} Each iteration uses at most $m+1\le\binom{d+1}{2}+1$ leverage-score
queries, all at genuine designs in $\Delta_n$; hence the phase uses
$O\big(d^2\log\log(\Theta_{\mathrm N}/\eps)\big)$ queries in total.
\item\label{it:nwt-feasible} All iterates remain in $\Cnull\subseteq\relint\Delta_{S^\opt}$: the
method never leaves the relative interior of the optimal face, and never consults any oracle beyond
leverage scores.
\end{enumerate}
\end{theorem}

\begin{proof}
\emph{Well-posedness and invariance.} The starting point lies in $\Cnull$. We claim every iterate
does. Indeed, each iteration either takes a damped step ($\lambda>\tfrac14$), which decreases $h$ by
at least $\wsc(\lambda)>0$ (\Cref{lem:newtonstep}\ref{it:damped}), or a full step
($\lambda\le\tfrac14$), which decreases $h$ as well (\Cref{lem:newtonstep}\ref{it:fullmono}); both
keep the iterate in $D$. Since $h$ never increases, $h(\mathbf{p}^{(k)})-f^\opt\le c_0$ for all $k$, i.e.\
$\mathbf{p}^{(k)}\in\Cnull$; by \Cref{lem:c0}\ref{it:c0-interior} each $\mathbf{p}^{(k)}$ is a design in
$\relint\Delta_{S^\opt}$ with $p_i\ge\pmin/2$. In particular every query the algorithm makes, at
$\mathbf{p}^{(k)}$ and at $\tfrac12(\mathbf{p}^{(k)}+\mathbf{e}_j)$, is at a genuine design
(\Cref{prop:smrecover}), proving the ``all queries at designs'' part of \ref{it:nwt-queries} and
\ref{it:nwt-feasible}. The Newton system is nonsingular at every iterate (\Cref{lem:faceSC}), and the
recovered Hessian is exact (\Cref{prop:smrecover}), so $\mathbf{n}_{\mathbf{p}}$ and $\lambda(\mathbf{p})$ are the exact
Newton direction and decrement of $h$ at each iterate.

\emph{Phase 1: at most $3$ damped steps.} While $\lambda(\mathbf{p}^{(k)})>\tfrac14$, each step decreases
$h$ by at least $\wsc(\tfrac14)=\tfrac14-\log\tfrac54>0.0268$
(\Cref{lem:newtonstep}\ref{it:damped}; $\wsc$ increasing). The total decrease available is
$h(\mathbf{p}^{(0)})-f^\opt\le c_0\le\wsc(\tfrac12)<0.0946$. If four damped steps occurred, $h$ would
decrease by more than $4\times0.0268=0.1072>c_0$, impossible. Hence at most $3$ damped steps occur,
and they occur before any full step.

\emph{Phase 2: quadratic convergence of the decrement.} Once $\lambda(\mathbf{p}^{(k)})\le\tfrac14$, the
algorithm takes full steps. By \Cref{thm:newtonquad} (applicable: $h$ is standard self-concordant on
$D$ with nondegenerate Hessian on $T^\opt$, as established above and in
\Cref{lem:faceSC,lem:min}), each full step satisfies
$\lambda^+\le\big(\lambda/(1-\lambda)\big)^2\le\tfrac{16}9\lambda^2\le2\lambda^2$, and also
$\lambda^+\le(\tfrac{1/4}{3/4})^2=\tfrac19\le\tfrac14$: the iteration stays in the full-step regime.
From $2\lambda^+\le(2\lambda)^2$ and $2\lambda\le\tfrac12$ at entry, after $j$ full steps
$2\lambda\le(\tfrac12)^{2^{j}}$, so $\lambda\le\tau$ holds as soon as
$(\tfrac12)^{2^j}\le2\tau$, i.e.\ after $j=\lceil\log_2\log_2\tfrac1{2\tau}\rceil^+$ full steps
(zero if $\tau=\tfrac14$).

\emph{Termination and correctness.} The algorithm stops at the first iterate with
$g(\mathbf{p})=\max_iv_i(\mathbf{p})-d\le\eps d$---a direct evaluation of \Cref{def:je} from the same query that
supplies the gradient, so the output is correct by construction, and feasible by the invariance
above. It remains to bound \emph{when} this test must pass. Suppose $\lambda(\mathbf{p})\le\tau$ at some
iterate $\mathbf{p}$ (which happens within $3+\lceil\log_2\log_2\tfrac1{2\tau}\rceil^+$ steps by Phases
1--2). By \Cref{lem:newtonstep}\ref{it:subopt}, $h(\mathbf{p})-f^\opt\le\tfrac23\tau^2$, and since
$\mathbf{p}\in\Cnull$, the gap conversion (\Cref{lem:gapconv}) gives
\[
g(\mathbf{p})\;\le\;\sqrt{2L_{\mathrm{loc}}\kappa_{\mathrm{loc}}\cdot\tfrac23\tau^2}
\;=\;\tau\sqrt{\tfrac43L_{\mathrm{loc}}\kappa_{\mathrm{loc}}}\;\le\;\eps d
\]
by the choice of $\tau$. So the stopping test passes at this iterate at the latest, and the
iteration count is as claimed in \ref{it:nwt-iters} (the final ``$+1$'' accounts for the terminal
certificate query; $L_{\mathrm{loc}}\kappa_{\mathrm{loc}}=36\Lopt\kappa$ gives the stated
$\Theta_{\mathrm N}$).

\emph{Query count.} Each iteration queries $\mathbf{v}(\mathbf{p})$ once and, if it does not stop, the $m$
probe designs of \Cref{prop:smrecover}: at most $m+1\le\binom{d+1}{2}+1=O(d^2)$ queries per
iteration; multiplying by \ref{it:nwt-iters} gives \ref{it:nwt-queries}.
\end{proof}

\begin{corollary}[Total complexity of the John ellipsoid in the leverage-score model]
\label{cor:total}
Under \Cref{def:nondeg}, \Cref{alg:main} with the \textsc{Newton} facial phase (advised constants
per \Cref{rem:params}) computes a
$(1+\eps)$-John ellipsoid using
\[
C(\mathbf{A})\;+\;O\Big(d^2\,\log\log\frac{\Theta_{\mathrm N}}{\eps}\Big)
\qquad\text{leverage-score queries,}
\]
with $C(\mathbf{A})$ the $\eps$-independent identification cost of \Cref{lem:identify} (run at tolerance
$c=c_0$; the $\kappa^3$-tightening of \Cref{thm:upper} is not needed for this phase). Once the
optimal face is identified and the iterate is in the facial sublevel set, the accuracy phase costs
$O(d^2\log\log(\Theta_{\mathrm N}/\eps))$ queries, with condition-type quantities only inside the
double logarithm.
\end{corollary}

\begin{remark}[Comparison with interior-point methods]
\label{rem:ipm}
Second-order interior-point methods~\cite{nesterov1994,khachiyantodd1993,anstreicher2002}, and the
Newton-type design algorithms of e.g.~\cite{lupong2013}, solve Newton systems in all $n$ weight
variables and achieve $\log(1/\eps)$ globally, at $\widetilde O(n^3)$-type cost per step. The facial
Newton phase is a different beast: it works in the $m\le\binom{d+1}{2}$ \emph{identified contact
coordinates}, consumes only leverage-score queries, and its iteration count is doubly logarithmic
because the warm start lands inside the Dikin ball, where Newton's method is already in its quadratic
regime after $O(1)$ damped steps. The price is the identification term $C(\mathbf{A})$, which the IPM line
does not pay; the trade is exactly the content of \Cref{prob:uniform}.
\end{remark}

\section{Discussion and open problems}
\label{sec:discussion}

\paragraph{What the modern line should change.}
The input-sparsity~\cite{cao2025} and lazy-update/streaming~\cite{woodruff2025} algorithms reduce the
cost of \emph{one} leverage-score computation while keeping the averaged $O(\eps^{-1}\log(n/d))$ outer
loop. By \Cref{thm:avg-intro} that iteration count is a property of the certification rule, not of the
oracle; a last-iterate method (\Cref{alg:main}, or away-step Frank--Wolfe~\cite{ahipasaoglu2008,zhao2023})
needs only $C(\mathbf{A})+O(\sqrt\kappa\log(\Theta/\eps))$ queries, and with the facial Newton phase only
$C(\mathbf{A})+O(d^2\log\log(\Theta/\eps))$. Running a last-iterate method \emph{on top of} the
cheap, sketched per-iteration oracle would replace the $\eps^{-1}$ factor accordingly.
One caveat is intrinsic: averaging is robust to the noise of a sketched oracle, whereas the last-iterate
and Newton phases as analyzed here consume \emph{exact} leverage scores (the Hessian recovery of
\Cref{prop:smrecover} is exact algebra, and would have to be redone with error propagation under a
sketched oracle). Establishing noise-tolerant versions is a concrete next step.

\paragraph{Where the problem now stands.}
The whole paper is one statement: the leverage-score complexity of the John ellipsoid separates into
\emph{certification}, \emph{identification}, and \emph{accuracy}, and two of the three are now settled.
\emph{Certification} by uniform averaging is an obstruction, costing $\Theta(1/\eps)$ (\Cref{thm:avg}),
but an avoidable one---it is the price of the averaged certificate, not of the oracle. \emph{Accuracy}
is not an obstruction: after identification, $\log\log(1/\eps)$ iterations suffice (\Cref{thm:newton}),
with no condition number multiplying the $\eps$-dependent term in either facial phase. What remains is
\emph{identification}: the cost $C(\mathbf{A})$ of reaching the optimal face is $\eps$-independent but
condition-dependent (through $K_{\mathrm{id}}$ and $\kappa_\Phi$ in \Cref{lem:identify}), and
nondegeneracy (\Cref{def:nondeg}) is assumed throughout the upper bounds. The open problems below are
the three frontiers this leaves: condition-free identification (\Cref{prob:uniform}), degeneracy and
singular faces (\Cref{prob:degenerate}), and stability under an approximate oracle (the caveat above).
We also note what an earlier version of this work left open and is now closed: whether the facial
Hessian is simulable within the pure leverage-score oracle. \Cref{prop:smrecover} answers this
exactly, with $O(d^2)$ queries per Newton step.

\begin{problem}[Condition-free identification]
\label{prob:uniform}
Is there a leverage-score algorithm that, on every instance with a nondegenerate optimal design,
identifies the contact set $S^\opt$ (equivalently: reaches the facial sublevel set $\Cnull$) in
$\poly(d)\cdot\polylog(1/\gsc,1/\pmin,\kappa)$ queries---or, ideally, in $\poly(d)$ queries with no
condition dependence at all? Combined with \Cref{thm:newton}, a positive answer would give a uniform
$\poly(d)\cdot\log\log(1/\eps)$ algorithm. Alternatively, is there a family of instances on which
every leverage-score algorithm needs $\eps^{-\Omega(1)}$ queries---or
$\mathrm{cond}(\mathbf{A})^{\Omega(1)}$ queries for identification?
\end{problem}

\noindent Two structural comments on \Cref{prob:uniform}. On the upper-bound side, identification is
exactly where the facial geometry (pyramidal width $\Phi$, slack $\gsc$, minimal weight $\pmin$)
enters; all of it sits in the warm-start term, and none of it is information-theoretically visible to
the certificate $g(\mathbf{p})\le\eps d$ itself. On the lower-bound side, the Hessian structure of
\Cref{prop:dict}\ref{it:hess} is suggestive: $\nabla^2f$ is the Gram matrix of rank-one symmetric
tensors (a Veronese configuration), and we were unable to realize within it the ``flat valley with
the optimum at its end'' geometry that Nesterov-type first-order lower bounds require; conversely,
every instance we tried is solved by a first-order method in $\polylog(1/\eps)$ queries after warm
start. A first-order lower bound for the John ellipsoid, in $\eps$ or in the condition
number, would be the first of its kind.

\begin{problem}[Degeneracy]
\label{prob:degenerate}
The upper bounds assume strict complementarity and independent contact matrices
(\Cref{def:nondeg}); these hold generically but not always. Quantify the cost of degeneracy: e.g.,
does a $\log\log(1/\eps)$ accuracy phase survive when the contact matrices are dependent (so the
optimal design is non-unique and the facial problem has a nontrivial solution set), with Newton
replaced by a method that quotients out the flat directions of \Cref{lem:flat}?
\end{problem}

\noindent\Cref{rem:nondeg} records why this exceptional set is measure-zero yet natural (it contains
the symmetric designs), and why a generic perturbation does not dispose of \Cref{prob:degenerate} for
free: it merely relocates the cost into $C(\mathbf{A})$ through the vanishing of $\gsc$ and $\Phi$.

\paragraph{Summary.}
The historical $\eps^{-1}$ of the leverage-score line is an artifact of averaging; the same oracle,
pointed at the last iterate, is geometric; and once the optimal face is known, the problem is, in the
affine-invariant sense that Newton's method makes precise, \emph{easy}: three damped steps and a
doubly logarithmic number of full Newton steps, all driven by leverage scores. The open problem is no
longer the accuracy; it is the geometry of finding the face.

\bibliography{refs}

@article{ahipasaoglu2008,
  author  = {Selin Damla Ahipa\c{s}ao\u{g}lu and Peng Sun and Michael J. Todd},
  title   = {Linear convergence of a modified {Frank--Wolfe} algorithm for computing minimum-volume enclosing ellipsoids},
  journal = {Optimization Methods and Software}, volume = {23}, number = {1}, pages = {5--19}, year = {2008}}

@inproceedings{alaoui2015,
  author    = {Ahmed El Alaoui and Michael W. Mahoney},
  title     = {Fast randomized kernel ridge regression with statistical guarantees},
  booktitle = {Advances in Neural Information Processing Systems (NeurIPS)}, year = {2015}, eprint = {1411.0306}, eprinttype = {arxiv}}

@article{alizadeh1997,
  author  = {Farid Alizadeh and Jean-Pierre A. Haeberly and Michael L. Overton},
  title   = {Complementarity and nondegeneracy in semidefinite programming},
  journal = {Mathematical Programming}, volume = {77}, number = {1}, pages = {111--128}, year = {1997}}

@inproceedings{allenzhu2017,
  author    = {Zeyuan Allen-Zhu and Yuanzhi Li and Aarti Singh and Yining Wang},
  title     = {Near-optimal design of experiments via regret minimization},
  booktitle = {International Conference on Machine Learning (ICML)}, year = {2017},
  note      = {Journal version in \emph{Mathematical Programming} 186:439--478, 2021}, eprint = {1711.05174}, eprinttype = {arxiv}}

@inproceedings{alman2025,
  author    = {Josh Alman and Ran Duan and Virginia {Vassilevska Williams} and Yinzhan Xu and Zixuan Xu and Renfei Zhou},
  title     = {More asymmetry yields faster matrix multiplication},
  booktitle = {ACM-SIAM Symposium on Discrete Algorithms (SODA)}, year = {2025}, eprint = {2404.16349}, eprinttype = {arxiv}}

@inproceedings{almanwilliams2021,
  author    = {Josh Alman and Virginia Vassilevska Williams},
  title     = {A refined laser method and faster matrix multiplication},
  booktitle = {ACM-SIAM Symposium on Discrete Algorithms (SODA)}, year = {2021}, eprint = {2010.05846}, eprinttype = {arxiv}}

@article{anstreicher2002,
  author  = {Kurt M. Anstreicher},
  title   = {Improved complexity for maximum volume inscribed ellipsoids},
  journal = {SIAM Journal on Optimization}, volume = {13}, number = {2}, pages = {309--320}, year = {2002}}

@article{avron2013,
  author  = {Haim Avron and Christos Boutsidis},
  title   = {Faster subset selection for matrices and applications},
  journal = {SIAM Journal on Matrix Analysis and Applications}, volume = {34}, number = {4}, pages = {1464--1499}, year = {2013}}

@article{beck2009,
  author  = {Amir Beck and Marc Teboulle},
  title   = {A fast iterative shrinkage-thresholding algorithm for linear inverse problems},
  journal = {SIAM Journal on Imaging Sciences}, volume = {2}, number = {1}, pages = {183--202}, year = {2009}}

@article{bomze2020,
  author  = {Immanuel M. Bomze and Francesco Rinaldi and Damiano Zeffiro},
  title   = {Active set complexity of the away-step {Frank--Wolfe} algorithm},
  journal = {SIAM Journal on Optimization}, volume = {30}, number = {3}, pages = {2470--2500}, year = {2020}}

@misc{braverman2016,
  author = {Vladimir Braverman and Dan Feldman and Harry Lang and Adiel Statman and Samson Zhou},
  title  = {New frameworks for offline and streaming coreset constructions},
  year   = {2016}, note = {arXiv:1612.00889}}

@inproceedings{braverman2018,
  author = {Vladimir Braverman and Petros Drineas and Cameron Musco and Christopher Musco and Jalaj Upadhyay and David P. Woodruff and Samson Zhou},
  title  = {Near-optimal linear algebra in the online and sliding window models},
  booktitle = {IEEE Symposium on Foundations of Computer Science (FOCS)}, year = {2020}, eprint = {1805.03765}, eprinttype = {arxiv}}

@book{bv2004,
  author = {Stephen Boyd and Lieven Vandenberghe},
  title  = {Convex Optimization}, publisher = {Cambridge University Press}, year = {2004}}

@inproceedings{calandriello2017,
  author    = {Daniele Calandriello and Alessandro Lazaric and Michal Valko},
  title     = {Distributed adaptive sampling for kernel matrix approximation},
  booktitle = {International Conference on Artificial Intelligence and Statistics (AISTATS)}, year = {2017}, eprint = {1803.10172}, eprinttype = {arxiv}}

@inproceedings{cao2025,
  author    = {Yang Cao and Xiaoyu Li and Zhao Song and Xin Yang and Tianyi Zhou},
  title     = {Faster algorithm for structured {John} ellipsoid computation},
  booktitle = {Advances in Neural Information Processing Systems (NeurIPS)}, year = {2025}, eprint = {2211.14407}, eprinttype = {arxiv}}

@article{chen2018,
  author  = {Yuansi Chen and Raaz Dwivedi and Martin J. Wainwright and Bin Yu},
  title   = {Fast {MCMC} sampling algorithms on polytopes},
  journal = {Journal of Machine Learning Research}, volume = {19}, pages = {1--86}, year = {2018}}

@inproceedings{cherapanamjeri2022,
  author    = {Yeshwanth Cherapanamjeri and Sandeep Silwal and David P. Woodruff and Samson Zhou},
  title     = {Optimal algorithms for linear algebra in the current matrix multiplication time},
  booktitle = {ACM-SIAM Symposium on Discrete Algorithms (SODA)}, year = {2023}, eprint = {2211.09964}, eprinttype = {arxiv}}

@inproceedings{clarkson2013,
  author    = {Kenneth L. Clarkson and David P. Woodruff},
  title     = {Low-rank approximation and regression in input sparsity time},
  booktitle = {ACM Symposium on Theory of Computing (STOC)}, year = {2013}, eprint = {1207.6365}, eprinttype = {arxiv}}

@inproceedings{cohen2015,
  author    = {Michael B. Cohen and Richard Peng},
  title     = {$\ell_p$ row sampling by {Lewis} weights},
  booktitle = {ACM Symposium on Theory of Computing (STOC)}, year = {2015}}

@inproceedings{cohen2015uniform,
  author    = {Michael B. Cohen and Yin Tat Lee and Cameron Musco and Christopher Musco and Richard Peng and Aaron Sidford},
  title     = {Uniform sampling for matrix approximation},
  booktitle = {Innovations in Theoretical Computer Science (ITCS)}, year = {2015}, eprint = {1408.5099}, eprinttype = {arxiv}}

@inproceedings{cohen2016online,
  author    = {Michael B. Cohen and Cameron Musco and Jakub Pachocki},
  title     = {Online row sampling},
  booktitle = {International Conference on Approximation Algorithms for Combinatorial Optimization Problems (APPROX)}, year = {2016}, eprint = {1604.05448}, eprinttype = {arxiv}}

@inproceedings{cohen2019,
  author    = {Michael B. Cohen and Ben Cousins and Yin Tat Lee and Xin Yang},
  title     = {A near-optimal algorithm for approximating the {John} ellipsoid},
  booktitle = {Conference on Learning Theory (COLT)}, year = {2019}, eprint = {1905.11580}, eprinttype = {arxiv}}

@inproceedings{cohenmusco2017,
  author    = {Michael B. Cohen and Cameron Musco and Christopher Musco},
  title     = {Input sparsity time low-rank approximation via ridge leverage score sampling},
  booktitle = {ACM-SIAM Symposium on Discrete Algorithms (SODA)}, year = {2017}, eprint = {1511.07263}, eprinttype = {arxiv}}

@inproceedings{derezinski2018,
  author    = {Micha\l{} Derezi\'nski and Manfred K. Warmuth and Daniel Hsu},
  title     = {Leveraged volume sampling for linear regression},
  booktitle = {Advances in Neural Information Processing Systems (NeurIPS)}, year = {2018}, eprint = {1802.06749}, eprinttype = {arxiv}}

@article{derezinski2021,
  author  = {Micha\l{} Derezi\'nski and Michael W. Mahoney},
  title   = {Determinantal point processes in randomized numerical linear algebra},
  journal = {Notices of the American Mathematical Society}, volume = {68}, number = {1}, pages = {34--45}, year = {2021}, eprint = {2005.03185}, eprinttype = {arxiv}}

@article{dragomir2021,
  author  = {Radu-Alexandru Dragomir and Adrien Taylor and Alexandre d'Aspremont and J\'er\^ome Bolte},
  title   = {Optimal complexity and certification of {Bregman} first-order methods},
  journal = {Mathematical Programming}, volume = {194}, pages = {41--83}, year = {2022}, eprint = {1911.08510}, eprinttype = {arxiv}}

@inproceedings{drineas2006,
  author    = {Petros Drineas and Michael W. Mahoney and S. Muthukrishnan},
  title     = {Sampling algorithms for $\ell_2$ regression and applications},
  booktitle = {ACM-SIAM Symposium on Discrete Algorithms (SODA)}, year = {2006}}

@article{drineas2011,
  author  = {Petros Drineas and Michael W. Mahoney and S. Muthukrishnan and Tam\'as Sarl\'os},
  title   = {Faster least squares approximation},
  journal = {Numerische Mathematik}, volume = {117}, number = {2}, pages = {219--249}, year = {2011}, eprint = {0710.1435}, eprinttype = {arxiv}}

@article{drineas2012,
  author  = {Petros Drineas and Malik Magdon-Ismail and Michael W. Mahoney and David P. Woodruff},
  title   = {Fast approximation of matrix coherence and statistical leverage},
  journal = {Journal of Machine Learning Research}, volume = {13}, pages = {3475--3506}, year = {2012}, eprint = {1109.3843}, eprinttype = {arxiv}}

@article{drineasmahoney2016,
  author  = {Petros Drineas and Michael W. Mahoney},
  title   = {{RandNLA}: randomized numerical linear algebra},
  journal = {Communications of the ACM}, volume = {59}, number = {6}, pages = {80--90}, year = {2016}}

@inproceedings{fazel2021,
  author    = {Maryam Fazel and Yin Tat Lee and Swati Padmanabhan and Aaron Sidford},
  title     = {Computing {Lewis} weights to high precision},
  booktitle = {ACM-SIAM Symposium on Discrete Algorithms (SODA)}, year = {2022}, eprint = {2110.15563}, eprinttype = {arxiv}}

@inproceedings{feldman2011,
  author    = {Dan Feldman and Michael Langberg},
  title     = {A unified framework for approximating and clustering data},
  booktitle = {ACM Symposium on Theory of Computing (STOC)}, year = {2011}, eprint = {1106.1379}, eprinttype = {arxiv}}

@misc{gustafson2018,
  author = {Adam Gustafson and Hariharan Narayanan},
  title  = {{John}'s walk}, year = {2018}, note = {arXiv:1803.02032}}

@article{gutman2018,
  author  = {David H. Gutman and Javier F. Pe\~na},
  title   = {A unified framework for {Bregman} proximal methods: subgradient, gradient, and accelerated gradient schemes},
  journal = {Mathematical Programming}, year = {2018}, note = {Published 2023; arXiv v3, 2019}, eprint = {1812.10198}, eprinttype = {arxiv}}

@article{halko2011,
  author  = {Nathan Halko and Per-Gunnar Martinsson and Joel A. Tropp},
  title   = {Finding structure with randomness: probabilistic algorithms for constructing approximate matrix decompositions},
  journal = {SIAM Review}, volume = {53}, number = {2}, pages = {217--288}, year = {2011}}

@article{hanzely2021,
  author  = {Filip Hanzely and Peter Richt\'arik and Lin Xiao},
  title   = {Accelerated {Bregman} proximal gradient methods for relatively smooth convex optimization},
  journal = {Computational Optimization and Applications}, volume = {79}, pages = {405--440}, year = {2021}}

@misc{harris2024,
  author = {Elizabeth Harris and Ali Eshragh and Bishnu Lamichhane and Jordan Shaw-Carmody and Elizabeth Stojanovski},
  title  = {Efficient data-driven leverage score sampling algorithm for the minimum volume covering ellipsoid problem in big data},
  year   = {2024}, note = {arXiv:2411.03617}}

@article{khachiyan1996,
  author  = {Leonid G. Khachiyan},
  title   = {Rounding of polytopes in the real number model of computation},
  journal = {Mathematics of Operations Research}, volume = {21}, number = {2}, pages = {307--320}, year = {1996}}

@article{khachiyantodd1993,
  author  = {Leonid G. Khachiyan and Michael J. Todd},
  title   = {On the complexity of approximating the maximal inscribed ellipsoid for a polytope},
  journal = {Mathematical Programming}, volume = {61}, pages = {137--159}, year = {1993}}

@article{kiefer1960,
  author  = {Jack Kiefer and Jacob Wolfowitz},
  title   = {The equivalence of two extremum problems},
  journal = {Canadian Journal of Mathematics}, volume = {12}, pages = {363--366}, year = {1960}}

@article{koutis2016,
  author  = {Ioannis Koutis and Alex Levin and Richard Peng},
  title   = {Faster spectral sparsification and numerical algorithms for {SDD} matrices},
  journal = {ACM Transactions on Algorithms}, volume = {12}, number = {2}, pages = {17}, year = {2016}, eprint = {1209.5821}, eprinttype = {arxiv}}

@article{kumar2005,
  author  = {Piyush Kumar and E. Alper Y\i ld\i r\i m},
  title   = {Minimum-volume enclosing ellipsoids and core sets},
  journal = {Journal of Optimization Theory and Applications}, volume = {126}, number = {1}, pages = {1--21}, year = {2005}}

@inproceedings{lacoste2015,
  author    = {Simon Lacoste-Julien and Martin Jaggi},
  title     = {On the global linear convergence of {Frank--Wolfe} optimization variants},
  booktitle = {Advances in Neural Information Processing Systems (NeurIPS)}, year = {2015}}

@inproceedings{li2013,
  author    = {Mu Li and Gary L. Miller and Richard Peng},
  title     = {Iterative row sampling},
  booktitle = {IEEE Symposium on Foundations of Computer Science (FOCS)}, year = {2013}, eprint = {1211.2713}, eprinttype = {arxiv}}

@inproceedings{lsw2015,
  author    = {Yin Tat Lee and Aaron Sidford and Sam Chiu-wai Wong},
  title     = {A faster cutting plane method and its implications for combinatorial and convex optimization},
  booktitle = {IEEE Symposium on Foundations of Computer Science (FOCS)}, year = {2015}}

@article{lu2018,
  author  = {Haihao Lu and Robert M. Freund and Yurii Nesterov},
  title   = {Relatively smooth convex optimization by first-order methods, and applications},
  journal = {SIAM Journal on Optimization}, volume = {28}, number = {1}, pages = {333--354}, year = {2018}}

@article{lupong2013,
  author  = {Zhaosong Lu and Ting Kei Pong},
  title   = {Computing optimal experimental designs via interior point method},
  journal = {SIAM Journal on Matrix Analysis and Applications}, volume = {34}, number = {4}, pages = {1556--1580}, year = {2013}, eprint = {1009.1909}, eprinttype = {arxiv}}

@inproceedings{madan2019,
  author    = {Vivek Madan and Mohit Singh and Uthaipon Tantipongpipat and Weijun Xie},
  title     = {Combinatorial algorithms for optimal design},
  booktitle = {Conference on Learning Theory (COLT)}, year = {2019}}

@article{mahoney2009,
  author  = {Michael W. Mahoney and Petros Drineas},
  title   = {{CUR} matrix decompositions for improved data analysis},
  journal = {Proceedings of the National Academy of Sciences}, volume = {106}, number = {3}, pages = {697--702}, year = {2009}}

@article{mahoney2011,
  author  = {Michael W. Mahoney},
  title   = {Randomized algorithms for matrices and data},
  journal = {Foundations and Trends in Machine Learning}, volume = {3}, number = {2}, pages = {123--224}, year = {2011}, eprint = {1104.5557}, eprinttype = {arxiv}}

@article{martinsson2020,
  author  = {Per-Gunnar Martinsson and Joel A. Tropp},
  title   = {Randomized numerical linear algebra: foundations and algorithms},
  journal = {Acta Numerica}, volume = {29}, pages = {403--572}, year = {2020}}

@inproceedings{meng2013,
  author    = {Xiangrui Meng and Michael W. Mahoney},
  title     = {Low-distortion subspace embeddings in input-sparsity time and applications to robust linear regression},
  booktitle = {ACM Symposium on Theory of Computing (STOC)}, year = {2013}, eprint = {1210.3135}, eprinttype = {arxiv}}

@article{munteanu2018,
  author  = {Alexander Munteanu and Chris Schwiegelshohn},
  title   = {Coresets---methods and history: a theoretician's design pattern for approximation and streaming algorithms},
  journal = {KI---K\"unstliche Intelligenz}, volume = {32}, number = {1}, pages = {37--53}, year = {2018}}

@misc{murray2023,
  author = {Riley Murray and James Demmel and Michael W. Mahoney and N. Benjamin Erichson and Maksim Melnichenko and Osman Asif Malik and Laura Grigori and Piotr Luszczek and Micha\l{} Derezi\'nski and Miles E. Lopes and Tianyu Liang and Hengrui Luo and Jack Dongarra},
  title  = {Randomized numerical linear algebra: a perspective on the field with an eye to software},
  year   = {2023}, note = {arXiv:2302.11474}}

@inproceedings{musco2017,
  author    = {Cameron Musco and Christopher Musco},
  title     = {Recursive sampling for the {Nystr\"om} method},
  booktitle = {Advances in Neural Information Processing Systems (NeurIPS)}, year = {2017}, eprint = {1605.07583}, eprinttype = {arxiv}}

@inproceedings{nelson2013,
  author    = {Jelani Nelson and Huy L. Nguyen},
  title     = {{OSNAP}: faster numerical linear algebra algorithms via sparser subspace embeddings},
  booktitle = {IEEE Symposium on Foundations of Computer Science (FOCS)}, year = {2013}, eprint = {1211.1002}, eprinttype = {arxiv}}

@article{nemirovski1999,
  author  = {Arkadi Nemirovski},
  title   = {On self-concordant convex--concave functions},
  journal = {Optimization Methods and Software}, volume = {11/12}, pages = {303--384}, year = {1999}}

@book{nesterov1994,
  author = {Yurii Nesterov and Arkadii Nemirovskii},
  title  = {Interior-Point Polynomial Algorithms in Convex Programming}, publisher = {SIAM}, year = {1994}}

@book{nesterov2018,
  author = {Yurii Nesterov},
  title  = {Lectures on Convex Optimization},
  edition = {2}, series = {Springer Optimization and Its Applications 137}, publisher = {Springer}, year = {2018}}

@inproceedings{nikolov2019,
  author    = {Aleksandar Nikolov and Mohit Singh and Uthaipon Tao Tantipongpipat},
  title     = {Proportional volume sampling and approximation algorithms for $A$-optimal design},
  booktitle = {ACM-SIAM Symposium on Discrete Algorithms (SODA)}, year = {2019}, eprint = {1802.08318}, eprinttype = {arxiv}}

@book{ortega1970,
  author = {James M. Ortega and Werner C. Rheinboldt},
  title  = {Iterative Solution of Nonlinear Equations in Several Variables}, publisher = {Academic Press}, year = {1970},
  note   = {Reprinted as SIAM Classics Appl.\ Math.\ 30, 2000}}

@article{pena2019,
  author  = {Javier F. Pe\~na and Daniel Rodr\'iguez},
  title   = {Polytope conditioning and linear convergence of the {Frank--Wolfe} algorithm},
  journal = {Mathematics of Operations Research}, volume = {44}, number = {1}, pages = {1--18}, year = {2019}}

@book{pukelsheim2006,
  author = {Friedrich Pukelsheim},
  title  = {Optimal Design of Experiments}, publisher = {SIAM Classics}, year = {2006}}

@inproceedings{rudi2015,
  author    = {Alessandro Rudi and Raffaello Camoriano and Lorenzo Rosasco},
  title     = {Less is more: {Nystr\"om} computational regularization},
  booktitle = {Advances in Neural Information Processing Systems (NeurIPS)}, year = {2015}, eprint = {1507.04717}, eprinttype = {arxiv}}

@inproceedings{sarlos2006,
  author    = {Tam\'as Sarl\'os},
  title     = {Improved approximation algorithms for large matrices via random projections},
  booktitle = {IEEE Symposium on Foundations of Computer Science (FOCS)}, year = {2006}}

@inproceedings{song2024dp,
  author    = {Xiaoyu Li and Yingyu Liang and Zhenmei Shi and Zhao Song and Junwei Yu},
  title     = {Fast {John} ellipsoid computation with differential privacy optimization},
  booktitle = {Conference on Parsimony and Learning (CPAL)}, year = {2025}, eprint = {2408.06395}, eprinttype = {arxiv}}

@misc{song2024quantum,
  author = {Xiaoyu Li and Zhao Song and Junwei Yu},
  title  = {Quantum speedups for approximating the {John} ellipsoid},
  year   = {2024}, note = {arXiv:2408.14018}}

@misc{song2025,
  author = {Zhao Song and David P. Woodruff and Lichen Zhang},
  title  = {Sublinear time quantum sensitivity sampling},
  year   = {2025}, note = {arXiv:2509.16801}}

@inproceedings{song2025stream,
  author = {Zhao Song and Shenghao Xie and Samson Zhou},
  title  = {Towards sampling data structures for tensor products in turnstile streams},
  booktitle = {International Conference on Learning Representations (ICLR)}, year = {2026}, eprint = {2510.03678}, eprinttype = {arxiv}}

@inproceedings{spielman2008,
  author    = {Daniel A. Spielman and Nikhil Srivastava},
  title     = {Graph sparsification by effective resistances},
  booktitle = {ACM Symposium on Theory of Computing (STOC)}, year = {2008}, eprint = {0803.0929}, eprinttype = {arxiv}}

@inproceedings{titterington1976,
  author    = {D. Michael Titterington},
  title     = {Algorithms for computing {D}-optimal designs on a finite design space},
  booktitle = {Proc.\ 1976 Conference on Information Sciences and Systems},
  publisher = {Johns Hopkins University}, pages = {213--216}, year = {1976}}

@article{todd2007,
  author  = {Michael J. Todd and E. Alper Y\i ld\i r\i m},
  title   = {On {Khachiyan}'s algorithm for the computation of minimum-volume enclosing ellipsoids},
  journal = {Discrete Applied Mathematics}, volume = {155}, pages = {1731--1744}, year = {2007}}

@book{todd2016,
  author = {Michael J. Todd},
  title  = {Minimum-Volume Ellipsoids: Theory and Algorithms}, publisher = {MOS-SIAM}, year = {2016}}

@article{tropp2011,
  author  = {Joel A. Tropp},
  title   = {Improved analysis of the subsampled randomized {Hadamard} transform},
  journal = {Advances in Adaptive Data Analysis}, volume = {3}, number = {1--2}, pages = {115--126}, year = {2011}, eprint = {1011.1595}, eprinttype = {arxiv}}

@inproceedings{vvw2024,
  author    = {Virginia {Vassilevska Williams} and Yinzhan Xu and Zixuan Xu and Renfei Zhou},
  title     = {New bounds for matrix multiplication: from alpha to omega},
  booktitle = {ACM-SIAM Symposium on Discrete Algorithms (SODA)}, year = {2024}, eprint = {2307.07970}, eprinttype = {arxiv}}

@article{wang2017,
  author  = {Yining Wang and Adams Wei Yu and Aarti Singh},
  title   = {On computationally tractable selection of experiments in measurement-constrained regression models},
  journal = {Journal of Machine Learning Research}, volume = {18}, number = {143}, pages = {1--41}, year = {2017}, eprint = {1601.02068}, eprinttype = {arxiv}}

@article{woodruff2014,
  author  = {David P. Woodruff},
  title   = {Sketching as a tool for numerical linear algebra},
  journal = {Foundations and Trends in Theoretical Computer Science}, volume = {10}, number = {1--2}, pages = {1--157}, year = {2014}, eprint = {1411.4357}, eprinttype = {arxiv}}

@inproceedings{woodruff2022,
  author    = {David P. Woodruff and Taisuke Yasuda},
  title     = {Online {Lewis} weight sampling},
  booktitle = {ACM-SIAM Symposium on Discrete Algorithms (SODA)}, year = {2023}, eprint = {2207.08268}, eprinttype = {arxiv}}

@inproceedings{woodruff2025,
  author    = {David P. Woodruff and Taisuke Yasuda},
  title     = {{John} ellipsoids via lazy updates},
  booktitle = {Advances in Neural Information Processing Systems (NeurIPS)}, year = {2024}, eprint = {2501.01801}, eprinttype = {arxiv}}

@article{yu2010,
  author  = {Yaming Yu},
  title   = {Monotonic convergence of a general algorithm for computing optimal designs},
  journal = {The Annals of Statistics}, volume = {38}, number = {3}, pages = {1593--1606}, year = {2010}, eprint = {0905.2646}, eprinttype = {arxiv}}

@article{zhao2023,
  author  = {Renbo Zhao},
  title   = {New analysis of an away-step {Frank--Wolfe} method for minimizing logarithmically-homogeneous barriers},
  journal = {Mathematics of Operations Research}, year = {2023}, eprint = {2305.17808}, eprinttype = {arxiv}}
\bibliographystyle{alpha}

\newpage
\appendix
\begin{center}\textbf{\Huge Appendix}\end{center}
\vspace{0.6em}
\section{Notation and constants}
\label{app:notation}

For reference we collect the recurring symbols: \Cref{tab:notation-core} the core objects,
\Cref{tab:notation-const} the optimal-face quantities and condition numbers. All are defined where they
first appear; the table gives a pointer.

\begin{center}\small
\renewcommand{\arraystretch}{1.0}
\begin{tabular}{@{}l p{0.46\textwidth} l@{}}
\toprule
Symbol & Meaning & Reference\\
\midrule
$\mathbf{A}\in\R^{n\times d}$, rows $\mathbf{a}_i$ & data matrix; polytope $\{\mathbf{x}:|\mathbf{a}_i^\top\mathbf{x}|\le1\}$ & \Cref{sec:intro}\\
$\mathbf{M}(\mathbf{p})=\sum_ip_i\mathbf{a}_i\mathbf{a}_i^\top$ & information matrix & \Cref{sec:prelim}\\
$v_i(\mathbf{p})=\mathbf{a}_i^\top\mathbf{M}(\mathbf{p})^{-1}\mathbf{a}_i$ & leverage score, $=-\nabla f(\mathbf{p})_i$ & \Cref{prop:dict}\\
$f(\mathbf{p})=-\log\det\mathbf{M}(\mathbf{p})$ & $D$-optimal objective & \Cref{sec:prelim}\\
$g(\mathbf{p})=\max_iv_i(\mathbf{p})-d$ & Frank--Wolfe gap; John guarantee $g\le\eps d$ & \Cref{prop:dict}\\
$\Delta_n$ & probability simplex & \Cref{sec:prelim}\\
$\relint$ & relative interior & \Cref{sec:prelim}\\
$\supp$ & support & \Cref{sec:prelim}\\
$\mathbf{p}^\opt$ & optimal design & \Cref{sec:prelim}\\
$\mathbf{M}^\opt=\mathbf{M}(\mathbf{p}^\opt)$ & optimal information matrix & \Cref{sec:prelim}\\
$f^\opt=f(\mathbf{p}^\opt)$ & optimal objective value & \Cref{sec:prelim}\\
$\wsc(t)$ & $t-\log(1+t)$ & \Cref{lem:legendre}\\
$\wscs(t)$ & $-t-\log(1-t)$ & \Cref{lem:legendre}\\
$\omega<2.372$ & matrix-multiplication exponent & \Cref{sec:prelim}\\
$\nnz(\mathbf{A})$ & number of nonzeros of $\mathbf{A}$ & \Cref{sec:prelim}\\
\bottomrule
\end{tabular}
\captionof{table}{Core objects.}\label{tab:notation-core}
\end{center}

\begin{center}\small
\renewcommand{\arraystretch}{1.0}
\begin{tabular}{@{}l p{0.46\textwidth} l@{}}
\toprule
Symbol & Meaning & Reference\\
\midrule
$S^\opt$ & contact set $\{i:v_i(\mathbf{p}^\opt)=d\}$ & \Cref{sec:prelim}\\
$m=|S^\opt|\le\binom{d+1}{2}$ & contact-set size & \Cref{sec:prelim}\\
$F^\opt=\Delta_{S^\opt}$ & optimal face & \Cref{sec:face}\\
$T^\opt$ & tangent space of the face & \Cref{sec:face}\\
$\widehat{\mathbf{H}}(\mathbf{p})$ & facial Hessian block $\big(\nabla^2f(\mathbf{p})\big)_{S^\opt\times S^\opt}$ & \Cref{sec:face}\\
$\varrho(\mathbf{p})$ & Dikin radius (local-norm distance to $\mathbf{p}^\opt$) & \Cref{sec:face}\\
$\gsc=\min_{i\notin S^\opt}(d-v_i(\mathbf{p}^\opt))$ & strict-complementarity slack & \Cref{def:nondeg}\\
$\pmin=\min_{i\in S^\opt}p^\opt_i$ & smallest contact weight & \Cref{def:nondeg}\\
$\muopt$ & facial Hessian minimum eigenvalue & \Cref{def:kappa}\\
$\Lopt$ & facial Hessian maximum eigenvalue & \Cref{def:kappa}\\
$\kappa=\Lopt/\muopt$ & facial condition number & \Cref{def:kappa}\\
$\Phi$ & facial distance & \Cref{thm:zhao}\\
$\kappa_\Phi=1/(\mu_R\Phi^2)$ & away-step condition number & \Cref{thm:zhao}\\
$\mu_{\mathrm{loc}}$ & localized min.\ eigenvalue on $\bar{\mathcal D}$ & \Cref{lem:c0}\\
$L_{\mathrm{loc}}$ & localized max.\ eigenvalue on $\bar{\mathcal D}$ & \Cref{lem:c0}\\
$\kappa_{\mathrm{loc}}=9\kappa$ & localized condition number & \Cref{lem:c0}\\
$\bar r$ & localization radius & \Cref{def:c0}\\
$c_0=\wsc(\bar r)$ & sublevel value & \Cref{def:c0}\\
$C(\mathbf{A})$ & $\eps$-independent warm-start/identification cost & \Cref{lem:identify}\\
$\Theta_{\mathrm A}$ & instance polynomial (accelerated phase) & \Cref{thm:upper}\\
$\Theta_{\mathrm N}$ & instance polynomial (Newton phase) & \Cref{thm:newton}\\
\bottomrule
\end{tabular}
\captionof{table}{Optimal face and condition numbers.}\label{tab:notation-const}
\end{center}

\newpage
\section{Map of the results}
\label{app:map}

\Cref{fig:dag} traces the logical structure of the paper: an arrow $X\to Y$ means that the proof of $Y$
invokes $X$. The three pillars---the averaging barrier (\Cref{sec:avg}), the accelerated phase
(\Cref{sec:upper}), and the facial Newton phase (\Cref{sec:newton})---share the dictionary
(\Cref{prop:dict}); the two last-iterate phases additionally share the facial toolkit of \Cref{sec:face}.
Dashed nodes are imported; bold nodes are the three headline theorems.

\begin{figure}[ht]
\centering
\resizebox{\textwidth}{!}{%
\begin{tikzpicture}[
  >={Stealth[length=2mm]},
  res/.style={draw, rounded corners=2pt, align=center, font=\scriptsize, inner sep=3pt},
  main/.style={draw, rounded corners=2pt, align=center, font=\scriptsize, very thick, inner sep=3.5pt},
  imp/.style={draw, dashed, rounded corners=2pt, align=center, font=\scriptsize, inner sep=3pt},
  ed/.style={->, shorten >=1pt, shorten <=1pt, gray}
]
\node[res]  (dict)   at (0,0)      {Dictionary\\(\ref{prop:dict})};
\node[imp]  (ccly)   at (0,-2)     {CCLY fixed pt.\\(\ref{thm:ccly})};
\node[imp]  (zhao)   at (0,-4)     {Away-step\\(\ref{thm:zhao},\,\ref{thm:bomze})};
\node[res]  (tool)   at (3.6,-1)   {Facial toolkit\\(\S\ref{sec:face})};
\node[res]  (ident)  at (3.6,-3)   {Warm start \&\\identification (\ref{lem:identify})};
\node[res]  (smrec)  at (3.6,-5)   {Hessian recovery\\(\ref{prop:smrecover})};
\node[main] (avg)    at (7.2,0)    {Averaging\\barrier (\ref{thm:avg})};
\node[main] (upper)  at (7.2,-2)   {Accelerated\\(\ref{thm:upper})};
\node[res]  (dom)    at (7.2,-3.3) {Domination\\(\ref{thm:domination})};
\node[main] (newton) at (7.2,-5)   {Facial Newton\\(\ref{thm:newton})};
\node[res]  (sep)    at (10.8,-1)  {Separation\\(\ref{cor:sep})};
\node[res]  (total)  at (10.8,-3.3){Total cost\\(\ref{cor:total})};
\node[res]  (cost)   at (10.8,-5)  {Per-iter.\ cost\\(\ref{prop:cost})};
\draw[ed] (dict) -- (tool);
\draw[ed] (dict) -- (ident);
\draw[ed] (dict) -- (smrec);
\draw[ed] (dict) -- (avg);
\draw[ed] (ccly) to[bend right=20] (avg);
\draw[ed] (zhao) -- (ident);
\draw[ed] (zhao) to[bend right=10] (dom);
\draw[ed] (tool) -- (upper);
\draw[ed] (tool) -- (dom);
\draw[ed] (tool) -- (newton);
\draw[ed] (tool) -- (ident);
\draw[ed] (ident) -- (upper);
\draw[ed] (ident) to[out=-25,in=200] (total);
\draw[ed] (smrec) -- (newton);
\draw[ed] (smrec) to[out=-25,in=210] (cost);
\draw[ed] (avg) -- (sep);
\draw[ed] (upper) -- (sep);
\draw[ed] (newton) to[out=35,in=250] (sep);
\draw[ed] (newton) -- (total);
\draw[ed] (newton) -- (cost);
\end{tikzpicture}}
\caption{Dependency map of the results. An arrow $X\to Y$ indicates that the proof of $Y$ invokes $X$;
dashed nodes are imported, and the three bold nodes are the headline theorems.}
\label{fig:dag}
\end{figure}
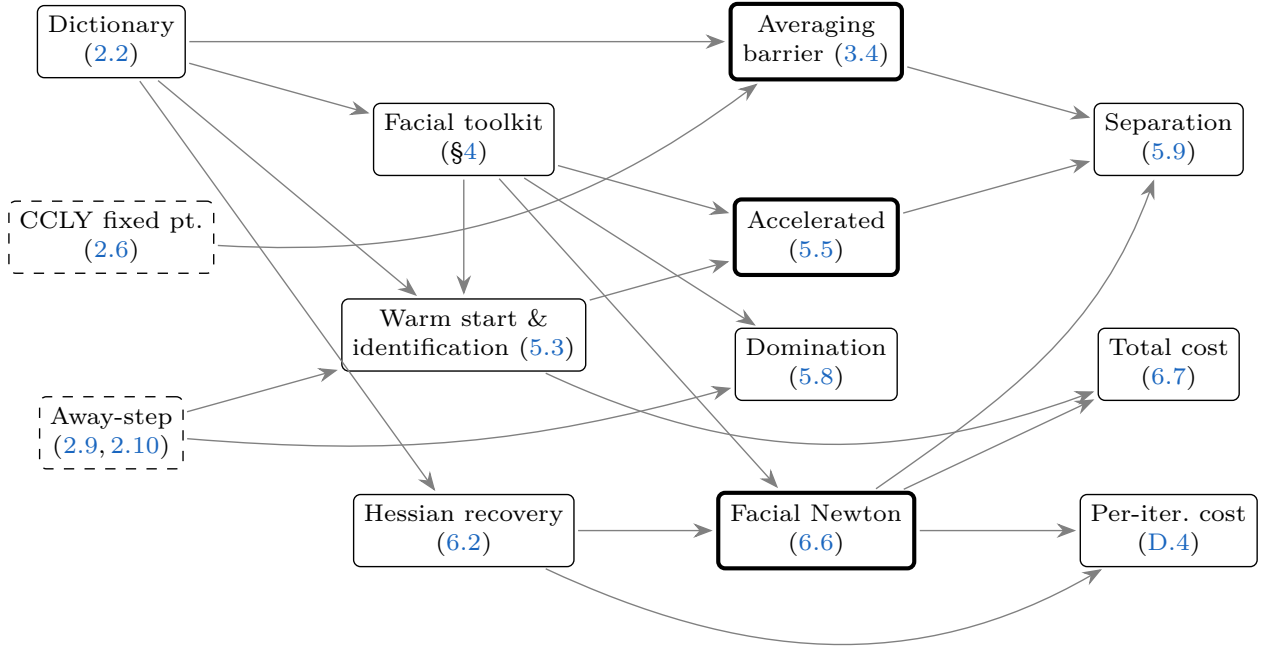

\section{Formulation: the John ellipsoid and \texorpdfstring{$D$/$G$}{D/G}-optimal design}
\label{app:formulation}

This appendix collects the classical equivalences that turn the geometric John-ellipsoid problem into the
convex program $\min_{\mathbf{p}\in\Delta_n}f(\mathbf{p})$ of the body~\cite{todd2016,pukelsheim2006,kiefer1960,cohen2019,cohen2015}.

\paragraph{The geometric problem.} For $\mathbf{A}\in\R^{n\times d}$ of rank $d$, the centrally symmetric polytope
is $P=\{\mathbf{x}\in\R^d:\norm{\mathbf{A}\mathbf{x}}_\infty\le1\}=\{\mathbf{x}:|\mathbf{a}_i^\top\mathbf{x}|\le1,\ i\in[n]\}$. Its \emph{John ellipsoid}
is the maximum-volume inscribed ellipsoid (\emph{MVIE}); by central symmetry it is centred at the origin,
and John's theorem gives the sandwiching $E\subseteq P\subseteq\sqrt d\,E$.

\paragraph{Ellipsoids as matrices.} A centred ellipsoid $E=\{\mathbf{x}:\mathbf{x}^\top\mathbf{S}\mathbf{x}\le1\}$ ($\mathbf{S}\succ0$) has
volume $\propto(\det\mathbf{S})^{-1/2}$ and support function $\max_{\mathbf{x}\in E}\mathbf{a}^\top\mathbf{x}=\sqrt{\mathbf{a}^\top\mathbf{S}^{-1}\mathbf{a}}$, so
$E\subseteq P\iff\mathbf{a}_i^\top\mathbf{S}^{-1}\mathbf{a}_i\le1$ for all $i$. Taking $\mathbf{S}=d\,\mathbf{M}(\mathbf{p})$ turns this inscribed
condition into $v_i(\mathbf{p})\le d$; by \Cref{prop:dict}\ref{it:trace} ($\max_iv_i\ge d$, equality iff $\mathbf{p}$ is
optimal) it is met only at the optimum, where $Q=\{\mathbf{x}:\mathbf{x}^\top(d\,\mathbf{M}(\mathbf{p}))\mathbf{x}\le1\}$ is exactly
inscribed---the John ellipsoid. For an approximate $(1+\eps)$-John design (\Cref{def:je}) one has instead
$\mathbf{a}_i^\top\mathbf{S}^{-1}\mathbf{a}_i=v_i/d\le1+\eps$, so the same $Q$ is inscribed only after shrinking by $\sqrt{1+\eps}$:
this is the left half of the sandwiching $\tfrac1{\sqrt{1+\eps}}Q\subseteq P\subseteq\sqrt d\,Q$.

\paragraph{Polar duality: MVIE $=$ MVEE of the rows.} The polar of $P$ is
$P^\circ=\mathrm{conv}\{\pm\mathbf{a}_i:i\in[n]\}$, and polarity acts on ellipsoids by
$\{\mathbf{x}^\top\mathbf{S}\mathbf{x}\le1\}^\circ=\{\mathbf{y}^\top\mathbf{S}^{-1}\mathbf{y}\le1\}$ with the product $\mathrm{vol}(E)\,\mathrm{vol}(E^\circ)$ constant;
hence maximizing inscribed volume in $P$ is the same as minimizing enclosing volume of $P^\circ$. The John
ellipsoid of $P$, the \emph{minimum-volume enclosing ellipsoid} (MVEE) of the rows $\{\pm\mathbf{a}_i\}$ (its
polar), $D$-optimal design, and $\ell_\infty$ Lewis weights are thus \emph{four facets of one optimal
weight $\mathbf{p}^\opt$}, related by polar duality and linear-time reductions~\cite{todd2016,cohen2015}.

\paragraph{$D$-optimal design.} The MVEE of $\{\pm\mathbf{a}_i\}$ is $\{\mathbf{y}:\mathbf{y}^\top\mathbf{M}(\mathbf{p}^\opt)^{-1}\mathbf{y}\le d\}$, where
\[
\mathbf{p}^\opt\in\argmax_{\mathbf{p}\in\Delta_n}\log\det\mathbf{M}(\mathbf{p})=\argmin_{\mathbf{p}\in\Delta_n}f(\mathbf{p})
\]
is the \emph{$D$-optimal design}, maximizing the determinant (``$D$'') of the information matrix
$\mathbf{M}(\mathbf{p})=\sum_ip_i\mathbf{a}_i\mathbf{a}_i^\top$. The containment $\mathbf{a}_i\in$ MVEE reads $v_i(\mathbf{p}^\opt)\le d$, with equality on
the contact set $S^\opt$; by \Cref{prop:dict}\ref{it:trace} this holds, among all $\mathbf{p}\in\Delta_n$, only at
the optimal designs (which all share the same $\mathbf{M}(\mathbf{p}^\opt)$; \Cref{lem:unique}). This is the program the
body solves, and
\Cref{prop:dict}\ref{it:grad} ($\nabla f=-\mathbf{v}$) is why a single leverage-score query is its first-order oracle.

\paragraph{$G$-optimal design and Kiefer--Wolfowitz.} In design terms $v_i(\mathbf{p})$ is the \emph{prediction
variance} at $\mathbf{a}_i$, and the \emph{$G$-optimal design} minimizes the worst one, $\min_{\mathbf{p}}\max_iv_i(\mathbf{p})$.
The Kiefer--Wolfowitz equivalence (\Cref{thm:kw};~\cite{kiefer1960,pukelsheim2006}) states that $D$- and
$G$-optimality coincide, with common optimum $\max_iv_i(\mathbf{p}^\opt)=d$. The $(1+\eps)$-John guarantee
$\max_iv_i\le(1+\eps)d$ (\Cref{def:je}) is therefore a $(1+\eps)$-approximate $G$-optimal certificate, equal
by \Cref{prop:dict}\ref{it:fw} to the Frank--Wolfe gap condition $g(\mathbf{p})\le\eps d$.

\paragraph{$\ell_\infty$ Lewis weights.} Finally, $\mathbf{p}^\opt$ is a rescaling of the \emph{$\ell_\infty$ Lewis
weights} $\mathbf{w}=d\,\mathbf{p}^\opt$ of $\mathbf{A}$~\cite{cohen2015}, and is the fixed point of the leverage map
$p_i\mapsto p_iv_i(\mathbf{p})/d$ (the multiplicative algorithm of \Cref{rem:titterington}, i.e.\ the CCLY
iteration, \Cref{thm:ccly}). This places
the John ellipsoid in the Lewis-weight family used for $\ell_p$ subspace embeddings.

\section{Deferred proofs}
\label{app:proofs}

\subsection{The dictionary\texorpdfstring{ (\Cref{prop:dict})}{}}
\label{app:dict}

\begin{proof}[Proof of \Cref{prop:dict}]
\ref{it:grad} By $\partial_{\mathbf{X}}\log\det \mathbf{X}=\mathbf{X}^{-1}$ and the chain rule, with $\partial_{p_i}\mathbf{M}(\mathbf{p})=\mathbf{a}_i\mathbf{a}_i^\top$,
\[
\partial_{p_i}f(\mathbf{p})=-\Tr\!\big(\mathbf{M}(\mathbf{p})^{-1}\partial_{p_i}\mathbf{M}(\mathbf{p})\big)=-\Tr\!\big(\mathbf{M}(\mathbf{p})^{-1}\mathbf{a}_i\mathbf{a}_i^\top\big)
=-\mathbf{a}_i^\top \mathbf{M}(\mathbf{p})^{-1}\mathbf{a}_i=-v_i(\mathbf{p}).
\]

\ref{it:trace} Summing the trace identity against $\mathbf{p}$,
\[
\sum_ip_iv_i(\mathbf{p})=\sum_ip_i\Tr(\mathbf{M}(\mathbf{p})^{-1}\mathbf{a}_i\mathbf{a}_i^\top)=\Tr\!\Big(\mathbf{M}(\mathbf{p})^{-1}\sum_ip_i\mathbf{a}_i\mathbf{a}_i^\top\Big)
=\Tr\!\big(\mathbf{M}(\mathbf{p})^{-1}\mathbf{M}(\mathbf{p})\big)=\Tr(\mathbf{I}_d)=d.
\]
As $\max_iv_i(\mathbf{p})$ is at least the $\mathbf{p}$-weighted average $d$, we get $\max_iv_i\ge d$; equality forces
$v_i(\mathbf{p})=d$ on $\supp(\mathbf{p})$, which is exactly the first-order condition $\nabla f(\mathbf{p})_i=-d$ for
$i\in\supp(\mathbf{p})$ for $\min_{\Delta_n}f$, i.e.\ optimality.

\ref{it:fw} A linear function on $\Delta_n$ is maximized at a vertex $\mathbf{e}_j$, so by \ref{it:grad},\ref{it:trace},
\[
g(\mathbf{p})=\max_j\inner{-\nabla f(\mathbf{p}),\mathbf{e}_j-\mathbf{p}}=\max_j\Big(v_j(\mathbf{p})-\sum_ip_iv_i(\mathbf{p})\Big)=\max_jv_j(\mathbf{p})-d.
\]
Convexity of $f$ gives
\[
f(\mathbf{p})-f^\opt\le\inner{\nabla f(\mathbf{p}),\mathbf{p}-\mathbf{p}^\opt}=\inner{-\nabla f(\mathbf{p}),\mathbf{p}^\opt-\mathbf{p}}\le
\max_{\mathbf{q}\in\Delta_n}\inner{-\nabla f(\mathbf{p}),\mathbf{q}-\mathbf{p}}=g(\mathbf{p}),
\]
which is the bound $g(\mathbf{p})\ge f(\mathbf{p})-f^\opt\ge0$. For the refined upper bound, let $\mathbf{p}^\opt$ be optimal with
$\mathbf{M}^\opt:=\mathbf{M}(\mathbf{p}^\opt)=\sum_ip^\opt_i\mathbf{a}_i\mathbf{a}_i^\top$. Applying the AM--GM inequality for the eigenvalues of the
positive-definite matrix $\mathbf{M}(\mathbf{p})^{-1/2}\mathbf{M}^\opt \mathbf{M}(\mathbf{p})^{-1/2}$ (i.e.\ $\log\det \mathbf{X}\le d\log(\tfrac1d\Tr \mathbf{X})$),
\begin{align*}
f(\mathbf{p})-f^\opt
&=\log\det \mathbf{M}^\opt-\log\det \mathbf{M}(\mathbf{p})=\log\det\!\big(\mathbf{M}(\mathbf{p})^{-1/2}\mathbf{M}^\opt \mathbf{M}(\mathbf{p})^{-1/2}\big)\\
&\le d\log\!\Big(\tfrac1d\Tr\!\big(\mathbf{M}(\mathbf{p})^{-1}\mathbf{M}^\opt\big)\Big)
 =d\log\!\Big(\tfrac1d\textstyle\sum_i p^\opt_i\,\mathbf{a}_i^\top \mathbf{M}(\mathbf{p})^{-1}\mathbf{a}_i\Big)
 =d\log\!\Big(\tfrac1d\textstyle\sum_i p^\opt_i v_i(\mathbf{p})\Big)\\
&\le d\log\!\Big(\tfrac1d\max_jv_j(\mathbf{p})\Big)=d\log\!\big(1+g(\mathbf{p})/d\big),
\end{align*}
using $\sum_ip^\opt_i=1$ and $\max_jv_j(\mathbf{p})=d+g(\mathbf{p})$; cf.\ \cite[\S2]{todd2016},~\cite{cohen2019}.
Finally $d\log(1+g/d)\le d\cdot(g/d)=g$.

\ref{it:hess} By the matrix-inverse derivative
\[
\partial_{p_j}\mathbf{M}(\mathbf{p})^{-1}=-\mathbf{M}(\mathbf{p})^{-1}\big(\partial_{p_j}\mathbf{M}(\mathbf{p})\big)\mathbf{M}(\mathbf{p})^{-1}=-\mathbf{M}(\mathbf{p})^{-1}\mathbf{a}_j\mathbf{a}_j^\top \mathbf{M}(\mathbf{p})^{-1},
\]
\[
\partial_{p_j}v_i(\mathbf{p})=\mathbf{a}_i^\top\big(\partial_{p_j}\mathbf{M}(\mathbf{p})^{-1}\big)\mathbf{a}_i
=-\mathbf{a}_i^\top \mathbf{M}(\mathbf{p})^{-1}\mathbf{a}_j\mathbf{a}_j^\top \mathbf{M}(\mathbf{p})^{-1}\mathbf{a}_i=-(\mathbf{a}_i^\top \mathbf{M}(\mathbf{p})^{-1}\mathbf{a}_j)^2,
\]
so $\nabla^2f(\mathbf{p})_{ij}=-\partial_{p_j}v_i(\mathbf{p})=(\mathbf{a}_i^\top \mathbf{M}(\mathbf{p})^{-1}\mathbf{a}_j)^2$. Write $\mathbf{b}_i:=\mathbf{M}(\mathbf{p})^{-1/2}\mathbf{a}_i$ and
use $\inner{\mathbf{X},\mathbf{Y}}=\Tr(\mathbf{X}^\top \mathbf{Y})$ on $\Sym^d$: then $(\mathbf{a}_i^\top \mathbf{M}(\mathbf{p})^{-1}\mathbf{a}_j)^2=(\mathbf{b}_i^\top \mathbf{b}_j)^2
=\inner{\mathbf{b}_i\mathbf{b}_i^\top,\mathbf{b}_j\mathbf{b}_j^\top}=\inner{\mathbf{b}_i^{\otimes2},\mathbf{b}_j^{\otimes2}}$, so $\nabla^2f(\mathbf{p})$ is the Gram
matrix of the vectors $\{\mathbf{b}_i^{\otimes2}\}_{i\in[n]}\subset\Sym^d$. A Gram matrix is positive
semidefinite, and its rank is $\dim\mathrm{span}\{\mathbf{b}_i^{\otimes2}\}\le\dim\Sym^d=\binom{d+1}{2}$. For the
kernel and for \eqref{eq:hessform}, for any $\boldsymbol{\delta}\in\R^n$,
\[
\begin{aligned}
\boldsymbol{\delta}^\top\nabla^2f(\mathbf{p})\,\boldsymbol{\delta}
&=\sum_{i,j}\delta_i\delta_j\inner{\mathbf{b}_i^{\otimes2},\mathbf{b}_j^{\otimes2}}
=\Big\langle\textstyle\sum_i\delta_i\mathbf{b}_i\mathbf{b}_i^\top,\sum_j\delta_j\mathbf{b}_j\mathbf{b}_j^\top\Big\rangle\\
&=\Big\|\textstyle\sum_i\delta_i\mathbf{b}_i\mathbf{b}_i^\top\Big\|_F^2
=\Big\|\mathbf{M}(\mathbf{p})^{-1/2}\big(\textstyle\sum_i\delta_i\mathbf{a}_i\mathbf{a}_i^\top\big)\mathbf{M}(\mathbf{p})^{-1/2}\Big\|_F^2,
\end{aligned}
\]
which vanishes iff $\sum_i\delta_i\mathbf{a}_i\mathbf{a}_i^\top=0$ (as $\mathbf{M}(\mathbf{p})^{-1/2}$ is invertible). This condition does
not involve $\mathbf{p}$, so $\ker\nabla^2f(\mathbf{p})$ is the same subspace for every $\mathbf{p}$ with $\mathbf{M}(\mathbf{p})\succ0$.
\end{proof}

\subsection{The self-concordance toolkit\texorpdfstring{ (\Cref{lem:Msand,lem:hesssand,lem:scineq,lem:levcomp,lem:legendre})}{}}
\label{app:sctoolkit}

We first record the derivatives of $F(\cdot)=-\log\det(\cdot)$ along matrix lines. For $\mathbf{X}\succ0$ and
$\mathbf{H}\in\Sym^d$, write $\mathbf{Y}_t:=\mathbf{X}+t\mathbf{H}$ and let $\varphi(t):=F(\mathbf{Y}_t)$ on the (open) set of $t$ with
$\mathbf{Y}_t\succ0$. Using
$\frac{d}{dt}\log\det\mathbf{Y}(t)=\Tr(\mathbf{Y}^{-1}\mathbf{Y}')$ and $\frac{d}{dt}\mathbf{Y}^{-1}=-\mathbf{Y}^{-1}\mathbf{Y}'\mathbf{Y}^{-1}$,
\begin{equation}
\label{eq:hessform-matrix}
\varphi'(t)=-\Tr\big(\mathbf{Y}_t^{-1}\mathbf{H}\big),\qquad
\varphi''(t)=\Tr\big(\mathbf{Y}_t^{-1}\mathbf{H}\mathbf{Y}_t^{-1}\mathbf{H}\big)
=\fnorm{\mathbf{Y}_t^{-1/2}\mathbf{H}\mathbf{Y}_t^{-1/2}}^2=\locnorm{\mathbf{H}}{\mathbf{Y}_t}^2,
\end{equation}
where the middle equality holds because $\Tr(\mathbf{Y}^{-1}\mathbf{H}\mathbf{Y}^{-1}\mathbf{H})
=\Tr\big((\mathbf{Y}^{-1/2}\mathbf{H}\mathbf{Y}^{-1/2})^2\big)$ and the trace of the square of a symmetric matrix is its
squared Frobenius norm. In particular $\varphi'(0)=\inner{\nabla F(\mathbf{X}),\mathbf{H}}$ with
$\nabla F(\mathbf{X})=-\mathbf{X}^{-1}$, and $\varphi''>0$ unless $\mathbf{H}=0$ (used in \Cref{lem:unique}).

\begin{proof}[Proof of \Cref{lem:Msand}]
Let $\mathbf{E}:=\mathbf{X}^{-1/2}\mathbf{H}\mathbf{X}^{-1/2}$, a symmetric matrix with
$\norm{\mathbf{E}}_{\mathrm{op}}\le\fnorm{\mathbf{E}}=r$ (the operator norm of a symmetric matrix is its largest
absolute eigenvalue, while the Frobenius norm is the $\ell_2$ norm of all eigenvalues). Hence
$-r\mathbf{I}\preceq\mathbf{E}\preceq r\mathbf{I}$; conjugating by $\mathbf{X}^{1/2}$ (congruence preserves the semidefinite
order) gives $-r\mathbf{X}\preceq\mathbf{H}\preceq r\mathbf{X}$. Adding $\mathbf{X}$ yields the sandwich for $\mathbf{X}+\mathbf{H}$; if $r<1$
the lower bound $(1-r)\mathbf{X}\succ0$ gives $\mathbf{X}+\mathbf{H}\succ0$.
\end{proof}

\begin{proof}[Proof of \Cref{lem:hesssand}]
Write $\mathbf{G}:=\mathbf{X}^{-1/2}\mathbf{H}\mathbf{X}^{-1/2}$ (so $\locnorm{\mathbf{H}}{\mathbf{X}}=\fnorm{\mathbf{G}}$) and
$\mathbf{S}:=\mathbf{Y}^{-1/2}\mathbf{X}^{1/2}$, so that
\[
\mathbf{Y}^{-1/2}\mathbf{H}\mathbf{Y}^{-1/2}=\mathbf{Y}^{-1/2}\mathbf{X}^{1/2}\,\mathbf{G}\,\mathbf{X}^{1/2}\mathbf{Y}^{-1/2}=\mathbf{S}\mathbf{G}\mathbf{S}^\top .
\]
The singular values of $\mathbf{S}$ satisfy
$\sigma_{\max}(\mathbf{S})^2=\norm{\mathbf{S}\mathbf{S}^\top}_{\mathrm{op}}
=\norm{\mathbf{X}^{1/2}\mathbf{Y}^{-1}\mathbf{X}^{1/2}}_{\mathrm{op}}$ (note $\mathbf{S}\mathbf{S}^\top=\mathbf{Y}^{-1/2}\mathbf{X}\mathbf{Y}^{-1/2}$ and
$\mathbf{X}^{1/2}\mathbf{Y}^{-1}\mathbf{X}^{1/2}$ have the same spectrum, being similar via $\mathbf{X}^{1/2}\mathbf{Y}^{1/2}$-type
conjugation; concretely both equal the spectrum of $\mathbf{Y}^{-1}\mathbf{X}$). From $\mathbf{Y}\succeq(1-r)\mathbf{X}$,
inversion (order-reversing on the positive-definite cone, see the proof of \Cref{lem:levcomp}) gives
$\mathbf{Y}^{-1}\preceq(1-r)^{-1}\mathbf{X}^{-1}$, hence
$\mathbf{X}^{1/2}\mathbf{Y}^{-1}\mathbf{X}^{1/2}\preceq(1-r)^{-1}\mathbf{I}$ and $\sigma_{\max}(\mathbf{S})^2\le(1-r)^{-1}$. Likewise
$\mathbf{Y}\preceq(1+r)\mathbf{X}$ gives $\sigma_{\min}(\mathbf{S})^2\ge(1+r)^{-1}$. Finally, for any symmetric $\mathbf{G}$ and
invertible $\mathbf{S}$,
\[
\sigma_{\min}(\mathbf{S})^2\,\fnorm{\mathbf{G}}\ \le\ \fnorm{\mathbf{S}\mathbf{G}\mathbf{S}^\top}\ \le\ \sigma_{\max}(\mathbf{S})^2\,\fnorm{\mathbf{G}}:
\]
the upper bound is $\fnorm{\mathbf{S}\mathbf{G}\mathbf{S}^\top}\le\sigma_{\max}(\mathbf{S})\fnorm{\mathbf{G}\mathbf{S}^\top}\le
\sigma_{\max}(\mathbf{S})^2\fnorm{\mathbf{G}}$ (multiplying by a matrix scales the Frobenius norm by at most the
operator norm); the lower bound follows by applying the upper one to
$\mathbf{G}=\mathbf{S}^{-1}(\mathbf{S}\mathbf{G}\mathbf{S}^\top)\mathbf{S}^{-\top}$ with the roles reversed, using
$\sigma_{\max}(\mathbf{S}^{-1})=1/\sigma_{\min}(\mathbf{S})$. Squaring and combining the displayed bounds proves
the lemma; tracking which sandwich half was used in each direction proves the one-sidedness claims.
\end{proof}

\begin{proof}[Proof of \Cref{lem:scineq}]
Let $\mathbf{H}:=\mathbf{Y}-\mathbf{X}$, $r=\locnorm{\mathbf{H}}{\mathbf{X}}$, and $\varphi(t)=F(\mathbf{X}+t\mathbf{H})$ as in
\eqref{eq:hessform-matrix}; the segment $\{\mathbf{X}+t\mathbf{H}:t\in[0,1]\}$ lies in the (convex) cone
$\Sym^d_{\succ0}$, so $\varphi$ is smooth on $[0,1]$ and Taylor's theorem with integral remainder
gives
\[
F(\mathbf{Y})=\varphi(1)=\varphi(0)+\varphi'(0)+\int_0^1(1-t)\,\varphi''(t)\,dt
=F(\mathbf{X})+\inner{\nabla F(\mathbf{X}),\mathbf{Y}-\mathbf{X}}+\int_0^1(1-t)\locnorm{\mathbf{H}}{\mathbf{Y}_t}^2dt .
\]
It remains to bound the integrand. Since $\locnorm{t\mathbf{H}}{\mathbf{X}}=tr$, \Cref{lem:Msand} gives
$\mathbf{Y}_t\preceq(1+tr)\mathbf{X}$ for every $t$ (and $\mathbf{Y}_t\succeq(1-tr)\mathbf{X}$, useful when $tr<1$). By the
one-sided halves of \Cref{lem:hesssand},
\[
\locnorm{\mathbf{H}}{\mathbf{Y}_t}^2\;\ge\;\frac{\locnorm{\mathbf{H}}{\mathbf{X}}^2}{(1+tr)^{2}}=\frac{r^2}{(1+tr)^2}
\qquad\text{always},
\qquad
\locnorm{\mathbf{H}}{\mathbf{Y}_t}^2\;\le\;\frac{r^2}{(1-tr)^2}\qquad\text{if }tr<1 .
\]
\ref{it:sclower} follows from
$\int_0^1(1-t)\,\frac{r^2}{(1+tr)^2}\,dt=\wsc(r)$: substituting $u=1+tr$ ($t=(u-1)/r$,
$dt=du/r$, $1-t=(r+1-u)/r$),
\[
\int_0^1(1-t)\frac{r^2}{(1+tr)^2}dt
=\int_1^{1+r}\frac{r+1-u}{u^2}\,du
=(r+1)\Big[\frac{-1}{u}\Big]_1^{1+r}-\big[\log u\big]_1^{1+r}
=r-\log(1+r)=\wsc(r).
\]
\ref{it:scupper} follows for $r<1$ from
$\int_0^1(1-t)\,\frac{r^2}{(1-tr)^2}\,dt=\wscs(r)$: substituting $u=1-tr$ (so $t=(1-u)/r$,
$dt=-du/r$, $1-t=(r-1+u)/r$),
\begin{align*}
\int_0^1(1-t)\frac{r^2}{(1-tr)^2}dt
&=\int_{1-r}^{1}\frac{r-1+u}{u^2}\,du
=(r-1)\Big[\frac{-1}{u}\Big]_{1-r}^{1}+\big[\log u\big]_{1-r}^{1}\\
&=(r-1)\cdot\frac{r}{1-r}-\log(1-r)
=-r-\log(1-r)=\wscs(r),
\end{align*}
the penultimate step using $(r-1)\big({-1}+\tfrac1{1-r}\big)=(r-1)\tfrac{r}{1-r}=-r$.
\end{proof}

\begin{proof}[Proof of \Cref{lem:levcomp}]
We first record that inversion reverses the order on $\Sym^d_{\succ0}$: if $\mathbf{A}\succeq\mathbf{B}\succ0$ then
$\mathbf{B}^{-1/2}\mathbf{A}\mathbf{B}^{-1/2}\succeq\mathbf{I}$, so its inverse satisfies
$\mathbf{B}^{1/2}\mathbf{A}^{-1}\mathbf{B}^{1/2}\preceq\mathbf{I}$, i.e.\ $\mathbf{A}^{-1}\preceq\mathbf{B}^{-1}$. Applying this to
$(1-r)\mathbf{X}\preceq\mathbf{Y}\preceq(1+r)\mathbf{X}$ gives
$(1+r)^{-1}\mathbf{X}^{-1}\preceq\mathbf{Y}^{-1}\preceq(1-r)^{-1}\mathbf{X}^{-1}$, and evaluating the quadratic forms at
$\mathbf{a}$ proves the claim.
\end{proof}

\begin{proof}[Proof of \Cref{lem:legendre}]
Monotonicity and strict convexity: $\wsc'(t)=1-\tfrac1{1+t}=\tfrac{t}{1+t}\ge0$ and
$\wsc''(t)=(1+t)^{-2}>0$; $\wscs'(t)=-1+\tfrac1{1-t}=\tfrac{t}{1-t}\ge0$ and
$\wscs''(t)=(1-t)^{-2}>0$; both vanish with their first derivatives at $0$. Conjugacy: for fixed
$t\in[0,1)$ the concave function $s\mapsto ts-\wsc(s)$ on $s\ge0$ has derivative
$t-\tfrac{s}{1+s}$, which vanishes at $s^\star=\tfrac{t}{1-t}\ge0$; since $ts-\wsc(s)\to-\infty$ as
$s\to\infty$ (because $\wsc(s)=s-\log(1+s)$ grows linearly with slope $1>t$), $s^\star$ is the
maximizer, and
\[
\sup_{s\ge0}\{ts-\wsc(s)\}
=\frac{t^2}{1-t}-\frac{t}{1-t}+\log\Big(1+\frac{t}{1-t}\Big)
=-t-\log(1-t)=\wscs(t).
\]
Quadratic bounds: from the power series $\wscs(t)=\sum_{k\ge2}t^k/k$ (valid for $|t|<1$) and
$1/k\le1/2$,
$\wscs(t)\le\tfrac12\sum_{k\ge2}t^k=\tfrac{t^2}{2(1-t)}$, which at $t\le\tfrac14$ is at most
$\tfrac{t^2}{2\cdot3/4}=\tfrac23t^2$. For the lower bound on $\wsc$: $g(t):=\wsc(t)-t^2/4$ has
$g(0)=0$ and $g'(t)=\tfrac{t}{1+t}-\tfrac t2=\tfrac{t(1-t)}{2(1+t)}\ge0$ on $[0,1]$, so
$\wsc(t)\ge t^2/4$ there; consequently, for $c\le\wsc(1)=1-\log2$, the inverse satisfies
$\wsc^{-1}(c)\le2\sqrt c$. Numerically $\wsc(\tfrac14)=\tfrac14-\log\tfrac54=0.02685\ldots>\tfrac1{38}$
and $\wsc(\tfrac12)=\tfrac12-\log\tfrac32=0.09453\ldots<0.0946$.
\end{proof}

\subsection{The localized FISTA phase\texorpdfstring{ (\Cref{thm:upper})}{}}
\label{app:fista}

Throughout this appendix \Cref{def:nondeg} is in force,
$Q:=\Delta_{S^\opt}$, $h=f|_D$, $\mathbf{x}^\opt:=\mathbf{p}^\opt$, $F:=h+\iota_Q$ (so $F^\opt=f^\opt$), and
$R:=\bar{\mathcal D}$ is the facial Dikin ball of \Cref{def:c0}, on which
\Cref{lem:c0}\ref{it:c0-cond} provides the constants
$L:=L_{\mathrm{loc}}=4\Lopt$ and $\mu:=\mu_{\mathrm{loc}}=\tfrac49\muopt$, with
$\kappa_{\mathrm{loc}}=L/\mu=9\kappa$. All norms are Euclidean on $V$; the gradient of $h$ at
$\mathbf{y}\in R$ is the projected vector $\nabla h(\mathbf{y})=P_{T^\opt}\big({-\mathbf{v}_{S^\opt}(\mathbf{y})}\big)$,
computable from one leverage-score query, and $\nabla h(\mathbf{x}^\opt)=0$ (\Cref{lem:min}(ii)).

\paragraph{The scheme.}
Given $\mathbf{x}_0\in Q$ with $h(\mathbf{x}_0)-f^\opt\le c$, one \emph{FISTA cycle} of length $N$
is~\cite{beck2009}: set $\mathbf{y}_1=\mathbf{x}_0$, $t_1=1$, and for $k=1,\dots,N$:
\[
\mathbf{x}_k=P_Q\Big(\mathbf{y}_k-\tfrac1L\nabla h(\mathbf{y}_k)\Big),\qquad
t_{k+1}=\tfrac{1+\sqrt{1+4t_k^2}}2,\qquad
\mathbf{y}_{k+1}=\mathbf{x}_k+\tfrac{t_k-1}{t_{k+1}}(\mathbf{x}_k-\mathbf{x}_{k-1}),
\]
with $P_Q$ the Euclidean projection onto the face simplex (computable in $O(m\log m)$ time). For
reference, the global guarantee of Beck--Teboulle reads:

\begin{theorem}[{FISTA~\cite[Thm.~4.4]{beck2009}}]
\label{thm:fista}
For the composite problem $\min\,\phi+\psi$ with $\phi$ convex and $L$-smooth \emph{on $\R^k$} and
$\psi$ proper closed convex, the scheme above (with the prox of $\psi$ in place of $P_Q$) satisfies,
for every minimizer $\mathbf{x}^\opt$,
$\ (\phi+\psi)(\mathbf{x}_k)-(\phi+\psi)(\mathbf{x}^\opt)\le\frac{2L\norm{\mathbf{x}_0-\mathbf{x}^\opt}^2}{(k+1)^2}$.
\end{theorem}

Our $h$ is neither finite nor $L$-smooth on all of $V$ (it blows up where $\mathbf{M}(\mathbf{p})$ degenerates), so
\Cref{thm:fista} does not apply as a black box. The next lemma reproves the Beck--Teboulle potential
argument under \emph{local} hypotheses, with the localization established by induction along the
trajectory. We use the standard facts $t_k\ge(k+1)/2$ and $t_{k+1}^2-t_{k+1}=t_k^2$
(immediate from the recursion), and $0\le\tfrac{t_k-1}{t_{k+1}}<1$.

\begin{lemma}[Localized FISTA cycle]
\label{lem:fistaloc}
Set $c_1:=(1/126)^2$. Suppose $\mathbf{x}_0\in Q$ with
$h(\mathbf{x}_0)-f^\opt\le c\le\min\{c_0,\ c_1/\kappa^{3}\}$. Then for every $k\ge1$ of the cycle:
\begin{enumerate}[label=\emph{(\alph*)},leftmargin=*,itemsep=2pt]
\item\label{it:floc-traj} $\mathbf{x}_k\in Q\cap R$ and $\mathbf{y}_k\in R$ (so every gradient evaluation and
every smoothness segment lies in $R$, and $\mathbf{M}(\mathbf{y}_k)\succ0$: all queries are well defined, at
weights supported on $S^\opt$);
\item\label{it:floc-rate} $h(\mathbf{x}_k)-f^\opt\ \le\ \dfrac{2L\norm{\mathbf{x}_0-\mathbf{x}^\opt}^2}{(k+1)^2}
\ \le\ \dfrac{4L\,c}{\mu\,(k+1)^2}$.
\end{enumerate}
In particular, after $N:=\lceil\sqrt{8\kappa_{\mathrm{loc}}}\rceil-1$ steps,
$h(\mathbf{x}_N)-f^\opt\le\tfrac12\big(h(\mathbf{x}_0)-f^\opt\big)$, and $\mathbf{x}_N\in\Cnull$.
\end{lemma}

\begin{proof}
\emph{Step 0: the one-step inequality.} Let $\mathbf{y}\in R$, and let
$\mathbf{x}:=P_Q(\mathbf{y}-\tfrac1L\nabla h(\mathbf{y}))$ satisfy $[\mathbf{y},\mathbf{x}]\subseteq R$. We claim that for every
$\mathbf{z}\in Q$,
\begin{equation}
\label{eq:onestep}
F(\mathbf{z})-F(\mathbf{x})\ \ge\ \tfrac L2\norm{\mathbf{x}-\mathbf{y}}^2+L\inner{\mathbf{y}-\mathbf{x},\mathbf{z}-\mathbf{y}} .
\end{equation}
If $h(\mathbf{z})=+\infty$ this is trivial. Otherwise: (i) the descent lemma
$h(\mathbf{x})\le h(\mathbf{y})+\inner{\nabla h(\mathbf{y}),\mathbf{x}-\mathbf{y}}+\tfrac L2\norm{\mathbf{x}-\mathbf{y}}^2$ holds because $\nabla h$
is $L$-Lipschitz on the convex set $R\supseteq[\mathbf{y},\mathbf{x}]$
(\Cref{lem:c0}\ref{it:c0-cond}); (ii) convexity of $h$ on $D$ gives
$h(\mathbf{z})\ge h(\mathbf{y})+\inner{\nabla h(\mathbf{y}),\mathbf{z}-\mathbf{y}}$; (iii) since $\mathbf{x}$ minimizes the
($L$-strongly convex) function $\mathbf{u}\mapsto\inner{\nabla h(\mathbf{y}),\mathbf{u}}+\tfrac L2\norm{\mathbf{u}-\mathbf{y}}^2$
over $Q$ and $\mathbf{z}\in Q$, first-order optimality gives
$\inner{\nabla h(\mathbf{y})+L(\mathbf{x}-\mathbf{y}),\,\mathbf{z}-\mathbf{x}}\ge0$. Subtracting (i) from (ii) and inserting (iii):
\[
F(\mathbf{z})-F(\mathbf{x})\ \ge\ \inner{\nabla h(\mathbf{y}),\mathbf{z}-\mathbf{x}}-\tfrac L2\norm{\mathbf{x}-\mathbf{y}}^2
\ \ge\ L\inner{\mathbf{y}-\mathbf{x},\mathbf{z}-\mathbf{x}}-\tfrac L2\norm{\mathbf{x}-\mathbf{y}}^2,
\]
and $\inner{\mathbf{y}-\mathbf{x},\mathbf{z}-\mathbf{x}}=\inner{\mathbf{y}-\mathbf{x},\mathbf{z}-\mathbf{y}}+\norm{\mathbf{y}-\mathbf{x}}^2$ rearranges this to
\eqref{eq:onestep}.

\emph{Step 1: the potential identity.} Write $v_k:=F(\mathbf{x}_k)-F^\opt$ and
$\mathbf{u}_k:=t_k\mathbf{x}_k-(t_k-1)\mathbf{x}_{k-1}-\mathbf{x}^\opt$. \emph{Provided the one-step inequality
\eqref{eq:onestep} is available at stage $k{+}1$} (i.e.\ $\mathbf{y}_{k+1}\in R$ and
$[\mathbf{y}_{k+1},\mathbf{x}_{k+1}]\subseteq R$), applying it at $\mathbf{z}=\mathbf{x}_k$ and at $\mathbf{z}=\mathbf{x}^\opt$, multiplying
the former by $(t_{k+1}-1)$, adding, multiplying by $t_{k+1}$, and using
$t_k^2=t_{k+1}^2-t_{k+1}$, yields after the standard completion of squares (using
$t_{k+1}\mathbf{y}_{k+1}-(t_{k+1}-1)\mathbf{x}_k=t_k\mathbf{x}_k-(t_k-1)\mathbf{x}_{k-1}$, which is the definition of
$\mathbf{y}_{k+1}$):
\begin{equation}
\label{eq:potential}
\tfrac2L\,t_k^2v_k-\tfrac2L\,t_{k+1}^2v_{k+1}\ \ge\ \norm{\mathbf{u}_{k+1}}^2-\norm{\mathbf{u}_k}^2 .
\end{equation}
Similarly, \eqref{eq:onestep} at stage $1$ ($\mathbf{y}_1=\mathbf{x}_0$, $\mathbf{z}=\mathbf{x}^\opt$) gives the base
\begin{equation}
\label{eq:base}
\tfrac2L\,t_1^2v_1+\norm{\mathbf{u}_1}^2\ \le\ \norm{\mathbf{x}_0-\mathbf{x}^\opt}^2=:B ,
\end{equation}
since with $t_1=1$, $\mathbf{u}_1=\mathbf{x}_1-\mathbf{x}^\opt$ and
$\tfrac L2\norm{\mathbf{x}_1-\mathbf{x}_0}^2+L\inner{\mathbf{x}_0-\mathbf{x}_1,\mathbf{x}^\opt-\mathbf{x}_0}
=\tfrac L2\big(\norm{\mathbf{x}_1-\mathbf{x}^\opt}^2-\norm{\mathbf{x}_0-\mathbf{x}^\opt}^2\big)$.
Chaining \eqref{eq:potential} from \eqref{eq:base}: as long as every stage up to $k$ was valid,
\begin{equation}
\label{eq:chain}
\tfrac2L\,t_k^2v_k+\norm{\mathbf{u}_k}^2\ \le\ B,
\qquad\text{hence}\qquad
v_k\le\frac{LB}{2t_k^2}\le\frac{2LB}{(k+1)^2} .
\end{equation}

\emph{Step 2: a priori bounds, assuming validity through stage $k$.} By strong convexity of $h$ on
$R$ with minimizer $\mathbf{x}^\opt$ (and $\mathbf{x}_0,\mathbf{x}^\opt\in R$, segment included),
$B=\norm{\mathbf{x}_0-\mathbf{x}^\opt}^2\le\tfrac{2c}{\mu}$. For $j\le k$ (using $\mathbf{x}_j\in R$, and
$v_j\le\tfrac L2B$ from \eqref{eq:chain} with $t_j\ge1$):
\[
\norm{\mathbf{x}_j-\mathbf{x}^\opt}^2\;\le\;\frac{2v_j}{\mu}\;\le\;\frac{LB}{\mu}
\;=\;\kappa_{\mathrm{loc}}\,B\;\le\;\frac{2\kappa_{\mathrm{loc}}\,c}{\mu}.
\]
Converting to the Dikin radius via
$\varrho(\cdot)\le\sqrt{\Lopt}\,\norm{\cdot-\mathbf{x}^\opt}_2$ (from \eqref{eq:rho-id} and
\Cref{def:kappa}) and $\Lopt=\tfrac L4$, $\mu=\tfrac49\muopt$, $\kappa_{\mathrm{loc}}=9\kappa$:
\[
\varrho(\mathbf{x}_j)\ \le\ \sqrt{\Lopt\cdot\frac{2\kappa_{\mathrm{loc}}c}{\mu}}
\ =\ \sqrt{\frac{2\cdot9\kappa\,\Lopt\,c}{(4/9)\muopt}}
\ =\ \sqrt{\frac{81}{2}\,\kappa^2c}\ \le\ 7\kappa\sqrt c\qquad(j\le k).
\]

\emph{Step 3: the induction.} We prove by induction on $k$ that all stages through $k$ are valid and
\ref{it:floc-traj} holds. \emph{Base.} $\mathbf{y}_1=\mathbf{x}_0$: $\varrho(\mathbf{x}_0)\le\sqrt{\Lopt B}\le
\sqrt{\Lopt\cdot2c/\mu}=\sqrt{\tfrac92\kappa c}\le\tfrac3{\sqrt2}\sqrt{\kappa c}\le\tfrac12$ (far
below, since $\kappa c\le c_1/\kappa^2\le c_1$), so $\mathbf{y}_1\in R$. \emph{Inductive step (and base
for $\mathbf{x}_1$).} Suppose $\mathbf{y}_{k+1}$'s ingredients $\mathbf{x}_k,\mathbf{x}_{k-1}$ satisfy
$\varrho\le7\kappa\sqrt c$ (Step 2; for $k=0$ read $\mathbf{x}_0=\mathbf{x}_{-1}$). Then, since
$0\le\tfrac{t_k-1}{t_{k+1}}<1$,
\[
\varrho(\mathbf{y}_{k+1})\ \le\ \varrho(\mathbf{x}_k)+\big(\varrho(\mathbf{x}_k)+\varrho(\mathbf{x}_{k-1})\big)
\ \le\ 21\kappa\sqrt c\ \le\ 21\sqrt{c_1}/\sqrt\kappa\ \le\ \tfrac16,
\]
using $c\le c_1/\kappa^3$ and $21\sqrt{c_1}=\tfrac{21}{126}=\tfrac16$. So $\mathbf{y}_{k+1}\in R$ and the
query at $\mathbf{y}_{k+1}$ is well defined. Next, the prox point: writing
$\mathbf{x}_{k+1}=P_Q(\mathbf{w})$ with $\mathbf{w}=\mathbf{y}_{k+1}-\tfrac1L\nabla h(\mathbf{y}_{k+1})$, the triangle inequality
through $P_Q(\mathbf{y}_{k+1})$ and nonexpansiveness of $P_Q$ give
\[
\norm{\mathbf{x}_{k+1}-\mathbf{y}_{k+1}}\ \le\ \dist(\mathbf{y}_{k+1},Q)+\tfrac1L\norm{\nabla h(\mathbf{y}_{k+1})}
\ \le\ 2\norm{\mathbf{y}_{k+1}-\mathbf{x}^\opt},
\]
where the second step uses $\dist(\mathbf{y}_{k+1},Q)\le\norm{\mathbf{y}_{k+1}-\mathbf{x}^\opt}$ (as $\mathbf{x}^\opt\in Q$)
together with the gradient bound
$\norm{\nabla h(\mathbf{y}_{k+1})}\le L\norm{\mathbf{y}_{k+1}-\mathbf{x}^\opt}$, valid because $\nabla h(\mathbf{x}^\opt)=0$
and $\nabla h$ is $L$-Lipschitz on $R$ by \Cref{lem:c0}\ref{it:c0-cond} (projected version; both
points and the segment lie in $R$). Hence, with $\norm{\mathbf{y}_{k+1}-\mathbf{x}^\opt}\le\varrho(\mathbf{y}_{k+1})/\sqrt{\muopt}$,
\[
\varrho(\mathbf{x}_{k+1})\ \le\ \varrho(\mathbf{y}_{k+1})+\sqrt{\Lopt}\norm{\mathbf{x}_{k+1}-\mathbf{y}_{k+1}}
\ \le\ \varrho(\mathbf{y}_{k+1})\big(1+2\sqrt\kappa\big)\ \le\ 3\sqrt{\kappa}\cdot21\kappa\sqrt c
\ =\ 63\,\kappa^{3/2}\sqrt c\ \le\ \tfrac12,
\]
using $1+2\sqrt\kappa\le3\sqrt\kappa$ ($\kappa\ge1$) and $63\sqrt{c_1}=\tfrac{63}{126}=\tfrac12$
with $c\le c_1/\kappa^3$. So $\mathbf{x}_{k+1}\in R$, the segment $[\mathbf{y}_{k+1},\mathbf{x}_{k+1}]\subseteq R$ by
convexity of $R$, stage $k{+}1$ is valid, the potential chain \eqref{eq:chain} extends to $k+1$, and
Step 2's sharper bound $\varrho(\mathbf{x}_{k+1})\le7\kappa\sqrt c$ is restored. This closes the induction
and proves \ref{it:floc-traj} and \ref{it:floc-rate} (the second inequality in \ref{it:floc-rate} is
$B\le2c/\mu$).

\emph{The halving conclusion.} With $N=\lceil\sqrt{8\kappa_{\mathrm{loc}}}\rceil-1$, i.e.\
$(N+1)^2\ge8\kappa_{\mathrm{loc}}$, \ref{it:floc-rate} and $B\le\tfrac2\mu(h(\mathbf{x}_0)-f^\opt)$ give
\[
h(\mathbf{x}_N)-f^\opt\ \le\ \frac{2L}{(N+1)^2}\cdot\frac{2\,(h(\mathbf{x}_0)-f^\opt)}{\mu}
\ =\ \frac{4\kappa_{\mathrm{loc}}}{(N+1)^2}\,\big(h(\mathbf{x}_0)-f^\opt\big)
\ \le\ \tfrac12\big(h(\mathbf{x}_0)-f^\opt\big).
\]
Finally $h(\mathbf{x}_N)-f^\opt\le c\le c_0$ and $\mathbf{x}_N\in Q\subseteq V$, so $\mathbf{x}_N\in\Cnull$.
\end{proof}

\begin{proof}[Completing the proof of \Cref{thm:upper}]
Run cycles $j=1,2,\dots$, each of length $N=O(\sqrt\kappa)$, restarting from the last iterate of the
previous cycle. By \Cref{lem:fistaloc} (whose hypothesis $h-f^\opt\le c$ at the cycle start holds for
every cycle, the gaps being halved), after $j$ cycles the gap is at most $2^{-j}c$, all queries are
well defined and supported on $S^\opt$, and cycle ends lie in $\Cnull$. The count of cycles needed
for $g\le\eps d$, via \Cref{lem:gapconv}, is as computed in the proof sketch in \Cref{sec:upper};
one extra query per cycle evaluates the certificate. Feasibility of the output point (a cycle-end
iterate in $\Cnull\subseteq\relint\Delta_{S^\opt}$) is \Cref{lem:c0}\ref{it:c0-interior}.
\end{proof}

\begin{remark}
The proof of \eqref{eq:potential} from \eqref{eq:onestep} is verbatim the argument
of~\cite[Lem.~4.1 and Thm.~4.4]{beck2009}; we reproduce the two applications of \eqref{eq:onestep}
and the algebra, since the only point of departure is \emph{where}
\eqref{eq:onestep} is valid. Apply \eqref{eq:onestep} at stage $k{+}1$ with $\mathbf{z}=\mathbf{x}_k$ (giving
$F(\mathbf{z})-F(\mathbf{x}_{k+1})=v_k-v_{k+1}$) and with $\mathbf{z}=\mathbf{x}^\opt$ (giving $-v_{k+1}$); multiply the former
by $(t_{k+1}-1)\ge0$ and add the latter:
\[
(t_{k+1}-1)v_k-t_{k+1}v_{k+1}\ \ge\ \tfrac L2\, t_{k+1}\norm{\mathbf{x}_{k+1}-\mathbf{y}_{k+1}}^2
+L\,\inner{\mathbf{x}_{k+1}-\mathbf{y}_{k+1},\,t_{k+1}\mathbf{y}_{k+1}-(t_{k+1}-1)\mathbf{x}_k-\mathbf{x}^\opt},
\]
where the inner product collects
$(t_{k+1}-1)(\mathbf{x}_k-\mathbf{y}_{k+1})+(\mathbf{x}^\opt-\mathbf{y}_{k+1})$ paired with $\mathbf{y}_{k+1}-\mathbf{x}_{k+1}$, with both
signs flipped. Multiplying by $t_{k+1}$, using $t_k^2=t_{k+1}^2-t_{k+1}$ and the elementary identity
$\norm{\mathbf{a}+\mathbf{b}}^2-\norm{\mathbf{b}}^2=\norm{\mathbf{a}}^2+2\inner{\mathbf{a},\mathbf{b}}$ with
$\mathbf{a}=t_{k+1}(\mathbf{x}_{k+1}-\mathbf{y}_{k+1})$ and $\mathbf{b}=t_{k+1}\mathbf{y}_{k+1}-(t_{k+1}-1)\mathbf{x}_k-\mathbf{x}^\opt$, and
finally the $\mathbf{y}$-update in the form $t_{k+1}\mathbf{y}_{k+1}-(t_{k+1}-1)\mathbf{x}_k=t_k\mathbf{x}_k-(t_k-1)\mathbf{x}_{k-1}$
(so that $\mathbf{b}=\mathbf{u}_k$ and $\mathbf{a}+\mathbf{b}=\mathbf{u}_{k+1}$), gives \eqref{eq:potential}.
\end{remark}

\subsection{The facial-coupling lemma\texorpdfstring{ (\Cref{lem:coupling})}{}}
\label{app:coupling}

We prove $\Phi\le2\,\sigma_{\min}$ (\Cref{lem:coupling}), where
$\mathbf{G}(\boldsymbol{\delta})=\sum_{i\in S^\opt}\delta_i\mathbf{c}_i$ acts on the face tangent
$T^\opt$, $\sigma_{\min}=\min\{\fnorm{\mathbf{G}(\boldsymbol{\delta})}:\boldsymbol{\delta}\in T^\opt,\norm{\boldsymbol{\delta}}_2=1\}$,
$\mathbf{c}_i=(\mathbf{M}^\opt)^{-1/2}\mathbf{a}_i\mathbf{a}_i^\top(\mathbf{M}^\opt)^{-1/2}$, and $\Phi$ is the facial distance of
$P:=\mathrm{conv}\{\mathbf{c}_i:i\in[n]\}$ in the Frobenius metric:
$\Phi=\min_{F}\mathrm{dist}\!\big(F,\mathrm{conv}(V\setminus F)\big)$ over proper nonempty faces $F$ of $P$, the face
entering as a \emph{convex body}~\cite[Thm.~1]{pena2019}. (After the isometry
$\mathbf{X}\mapsto(\mathbf{M}^\opt)^{-1/2}\mathbf{X}(\mathbf{M}^\opt)^{-1/2}$ this is exactly the facial distance of \Cref{thm:zhao} in the
local norm.)

\begin{proof}[Proof of \Cref{lem:coupling}]
Let $\sigma:=\sigma_{\min}$ and take a unit witness $\boldsymbol{\delta}\in T^\opt$, $\norm{\boldsymbol{\delta}}_2=1$, with
$\mathbf{E}:=\mathbf{G}(\boldsymbol{\delta})=\sum_{i\in S^\opt}\delta_i\mathbf{c}_i$ and $\fnorm{\mathbf{E}}=\sigma$.

\emph{Step 1 ($F^\opt:=\mathrm{conv}\{\mathbf{c}_i:i\in S^\opt\}$ is an exposed face of $P$).} On $\Sym^d$ consider the
functional $\mathbf{X}\mapsto\inner{\mathbf{X},\mathbf{I}}=\Tr\mathbf{X}$. For $i\in S^\opt$,
$\Tr\mathbf{c}_i=\norm{(\mathbf{M}^\opt)^{-1/2}\mathbf{a}_i}_2^2=\mathbf{a}_i^\top(\mathbf{M}^\opt)^{-1}\mathbf{a}_i=v_i(\mathbf{p}^\opt)=d$; and for every $i\in[n]$,
$v_i(\mathbf{p}^\opt)\le d$ with equality iff $i\in S^\opt$ (Kiefer--Wolfowitz, \Cref{thm:kw}, together with strict
complementarity, \Cref{def:nondeg}). Thus $\Tr\le d$ on all vertices of $P$, attained exactly on
$\{\mathbf{c}_i:i\in S^\opt\}$, so $\{\mathbf{X}:\Tr\mathbf{X}=d\}$ supports $P$ and $F^\opt=P\cap\{\Tr=d\}$ is an exposed face.

\emph{Step 2 ($F^\opt$ is a simplex).} By \Cref{def:nondeg} the matrices $\{\mathbf{c}_i\}_{i\in S^\opt}$ are linearly
(hence affinely) independent in $\Sym^d$, so $F^\opt$ is an $(m{-}1)$-simplex with $m=|S^\opt|$, and every
subset of its vertices spans a face of $F^\opt$, hence, a face of a face being a face, a face of $P$.

\emph{Step 3 (disjoint-support faces).} Split $\boldsymbol{\delta}=\boldsymbol{\delta}^+-\boldsymbol{\delta}^-$ into its nonnegative parts, with
disjoint supports $S^+,S^-\subseteq S^\opt$. As $\sum_i\delta_i=0$,
$s:=\textstyle\sum_i\delta_i^+=\sum_i\delta_i^-=\tfrac12\norm{\boldsymbol{\delta}}_1\ge\tfrac12\norm{\boldsymbol{\delta}}_2=\tfrac12$. The
points $\mathbf{q}^+:=s^{-1}\sum_{i\in S^+}\delta_i\mathbf{c}_i$ and $\mathbf{q}^-:=s^{-1}\sum_{i\in S^-}(-\delta_i)\mathbf{c}_i$ are genuine
convex combinations of $\{\mathbf{c}_i:i\in S^+\}$ and $\{\mathbf{c}_i:i\in S^-\}$ respectively, and
$\mathbf{q}^+-\mathbf{q}^-=s^{-1}\sum_{i\in S^\opt}\delta_i\mathbf{c}_i=s^{-1}\mathbf{E}$, so $\fnorm{\mathbf{q}^+-\mathbf{q}^-}=\sigma/s\le2\sigma$.

\emph{Step 4 (read off $\Phi$).} $F^+:=\mathrm{conv}\{\mathbf{c}_i:i\in S^+\}$ is a face of $F^\opt$ (Step 2), hence a
proper nonempty face of $P$ ($S^+\subsetneq S^\opt$ since $\boldsymbol{\delta}^-\ne0$, and $S^+\ne\emptyset$ since
$\boldsymbol{\delta}^+\ne0$). The vertices of $P$ not on $F^+$ are $\{\mathbf{c}_i:i\notin S^+\}\supseteq\{\mathbf{c}_i:i\in S^-\}\ni$
the constituents of $\mathbf{q}^-$, so $\mathbf{q}^-\in\mathrm{conv}(V\setminus F^+)$, while $\mathbf{q}^+\in F^+$. Hence
\[
\Phi\ \le\ \mathrm{dist}\!\big(F^+,\,\mathrm{conv}(V\setminus F^+)\big)\ \le\ \fnorm{\mathbf{q}^+-\mathbf{q}^-}\ \le\ 2\sigma .
\]
Squaring yields $\muopt=\sigma_{\min}^2\ge\Phi^2/4$. The exposed-face property of Step~1, which uses
\Cref{thm:kw} and fails for an arbitrary configuration of the $\mathbf{c}_i$, is exactly what makes an internal
near-dependence of the contacts ($\sigma$ small) register as a small facial distance of $P$.
\end{proof}

\subsection{The per-iteration arithmetic of the facial Newton phase\texorpdfstring{ (\Cref{thm:newton})}{}}
\label{app:arith}

The iteration \emph{count} of the facial Newton phase is \Cref{thm:newton}\ref{it:nwt-iters}; we record
here its per-iteration \emph{arithmetic}, deferred from \Cref{sec:newton}.

\begin{proposition}[Per-iteration cost of the facial Newton phase]
\label{prop:cost}
Fix an iterate $\mathbf{p}\in\relint\Delta_{S^\opt}$ and write $m=|S^\opt|\le\binom{d+1}{2}$. Computing the
gradient $\mathbf{v}_{S^\opt}(\mathbf{p})$, the certificate $\max_iv_i(\mathbf{p})$, the facial Hessian
$\widehat{\mathbf{H}}(\mathbf{p})$, and the Newton direction $\mathbf{n}_{\mathbf{p}}$ and decrement $\lambda(\mathbf{p})$ of one
iteration of \Cref{alg:newton} admits, for dense $\mathbf{A}$, two implementations:
\begin{enumerate}[label=\emph{(\roman*)},leftmargin=*,itemsep=2pt]
\item\label{it:cost-oracle} \emph{(Pure oracle.)} The $m+1\le\binom{d+1}{2}+1$ leverage-score queries of
\Cref{prop:smrecover}, costing $O(nd^{\omega+1})$ arithmetic, followed by an $O(m^3)$ linear solve.
\item\label{it:cost-direct} \emph{(One query plus direct linear algebra.)} A single leverage-score query
at $\mathbf{p}$, costing $O(nd^{\omega-1})$, followed by $O(md^2+d^\omega+m^2d+m^3)=O(d^6)$ further arithmetic.
\end{enumerate}
In particular, when $n=\Omega(d^{\,7-\omega})$ the further arithmetic of \ref{it:cost-direct} is within the
cost of its single query, so that implementation runs in $O(nd^{\omega-1})$ per iteration---a factor
$\Theta(d^2)$ below \ref{it:cost-oracle}.
\end{proposition}

\begin{proof}
Write $\mathbf{A}_{S^\opt}\in\R^{m\times d}$ for the contact-row submatrix and $\mathbf{M}=\mathbf{M}(\mathbf{p})\succ0$. One
leverage-score query returns $(v_i(\mathbf{w}))_{i\in[n]}$ at cost $O(nd^{\omega-1})$ for dense $\mathbf{A}$ (the cost
model of \Cref{sec:prelim}): $\mathbf{M}(\mathbf{w})=\mathbf{A}^\top\!\diag(\mathbf{w})\mathbf{A}$ is a $d\times n$ by $n\times d$ product,
$O(nd^{\omega-1})$; its inverse is $O(d^\omega)\le O(nd^{\omega-1})$; and the scores are the diagonal of
$\mathbf{A}\mathbf{M}(\mathbf{w})^{-1}\mathbf{A}^\top$, read off from $\mathbf{M}(\mathbf{w})^{-1}\mathbf{A}^\top$ ($O(nd^{\omega-1})$) by $n$ inner
products ($O(nd)$).

\ref{it:cost-oracle} By \Cref{prop:smrecover}, the leverage scores at $\mathbf{p}$ and at the $m$ designs
$\tfrac12(\mathbf{p}+\mathbf{e}_j)$, $j\in S^\opt$, determine $\widehat{\mathbf{H}}(\mathbf{p})$ exactly, and the query at $\mathbf{p}$ also
supplies $\mathbf{v}_{S^\opt}(\mathbf{p})$ and $\max_iv_i(\mathbf{p})$: these are $m+1$ queries, $O(nd^{\omega+1})$ in all. The
bordered system of \Cref{alg:newton} is $(m{+}1)$-dimensional and is solved in $O(m^3)$.

\ref{it:cost-direct} A single query at $\mathbf{p}$ gives $\mathbf{v}_{S^\opt}(\mathbf{p})=-\nabla h(\mathbf{p})|_{S^\opt}$, the
certificate, and a factorization of $\mathbf{M}$. With it, $\mathbf{B}:=\mathbf{M}^{-1}\mathbf{A}_{S^\opt}^\top\in\R^{d\times m}$ is $m$
back-solves, $O(md^2)$ (or recompute the factorization in $O(d^\omega)$); the Gram matrix
$\mathbf{G}=\mathbf{A}_{S^\opt}\mathbf{B}=\big(\mathbf{a}_i^\top\mathbf{M}^{-1}\mathbf{a}_j\big)_{i,j\in S^\opt}$ is an $m\times d$ by $d\times m$
product, $O(m^2d)$, and $\widehat{\mathbf{H}}(\mathbf{p})=\mathbf{G}\odot\mathbf{G}$ its entrywise square
(\Cref{prop:dict}\ref{it:hess}), $O(m^2)$. Solving the bordered system for $(\mathbf{n}_{\mathbf{p}},\nu)$ costs
$O(m^3)$ and $\lambda(\mathbf{p})^2=\mathbf{v}_{S^\opt}^\top\mathbf{n}_{\mathbf{p}}$ a further $O(m)$. The arithmetic beyond the
query is thus $O(md^2+d^\omega+m^2d+m^3)$; as $m\le\binom{d+1}{2}=O(d^2)$, the $m^3$ term dominates, giving
$O(d^6)$.

Finally $O(d^6)\le O(nd^{\omega-1})$ exactly when $n\ge d^{\,7-\omega}$, in which case \ref{it:cost-direct}
runs in $O(nd^{\omega-1})$, against $O(nd^{\omega+1})$ for \ref{it:cost-oracle}.
\end{proof}

\section{The explicit instance \texorpdfstring{$\mathbf{A}_\opt$}{A*}: exact data}
\label{sec:instance}

\paragraph{The instance.} Let $\mathbf{A}_\opt=\left(\begin{smallmatrix}2&0\\0&1\\1&1\\1&-1\end{smallmatrix}\right)$,
i.e.\ $\mathbf{a}_1=(2,0)$, $\mathbf{a}_2=(0,1)$, $\mathbf{a}_3=(1,1)$, $\mathbf{a}_4=(1,-1)$ ($n=4$, $d=2$).

\paragraph{Closed-form optimum.} $\mathbf{p}^\opt=(\tfrac13,0,\tfrac13,\tfrac13)$ gives
$\mathbf{M}(\mathbf{p}^\opt)=\diag(2,\tfrac23)$, $\mathbf{M}(\mathbf{p}^\opt)^{-1}=\diag(\tfrac12,\tfrac32)$, and
$\mathbf{v}(\mathbf{p}^\opt)=(2,\tfrac32,2,2)$. Thus $\max_iv_i=2=d$, so $\mathbf{p}^\opt$ is optimal (\Cref{thm:kw}), with contact
set $S^\opt=\{1,3,4\}$ and constraint $2$ strictly inactive ($v_2=\tfrac32<2$, $p^\opt_2=0$,
$\gsc=\tfrac12$). The contact matrices $\mathbf{a}_1\mathbf{a}_1^\top,\mathbf{a}_3\mathbf{a}_3^\top,\mathbf{a}_4\mathbf{a}_4^\top$ are
independent in $\Sym^2$ ($\det=-8\neq0$), so $\mathbf{p}^\opt$ is nondegenerate (\Cref{def:nondeg}), and unique
(\Cref{lem:unique}); $\pmin=\tfrac13$. The full Hessian matrix
$\mathbf{B}:=\big((\mathbf{a}_i^\top(\mathbf{M}^\opt)^{-1}\mathbf{a}_j)^2\big)_{i,j}$ is rational:
\[
\mathbf{B}=\begin{pmatrix}4&0&1&1\\0&\tfrac94&\tfrac94&\tfrac94\\1&\tfrac94&4&1\\1&\tfrac94&1&4\end{pmatrix},
\qquad\text{facial block}\quad
\widehat{\mathbf{H}}(\mathbf{p}^\opt)=\mathbf{B}_{S^\opt\times S^\opt}
=\begin{pmatrix}4&1&1\\1&4&1\\1&1&4\end{pmatrix}=3\,\mathbf{I}_3+\mathbf 1\mathbf 1^\top .
\]
Hence $\widehat{\mathbf{H}}(\mathbf{p}^\opt)$ has eigenvalues $\{6,3,3\}$ (eigenvector $\mathbf 1$ for $6$; the plane
$\mathbf 1^\perp\supseteq T^\opt$ for $3$), so in the convention of \Cref{def:kappa}:
$\muopt=3$ (the restriction to $T^\opt$ is exactly $3\,\mathbf{I}_2$), $\Lopt=6$, and $\kappa=2$. (On the
tangent of the \emph{full} simplex $\Delta_4$ the reduced Hessian has spectrum
$\approx\{0.073,\,3.0,\,3.87\}$---full rank, as $n-1=\binom{d+1}{2}=3$ leaves no flat directions here
(\Cref{lem:flat}). The \emph{facial} quantities are the ones that govern the rates.) We also record
$\sum_{i\in S^\opt}\mathbf{a}_i\mathbf{a}_i^\top=\diag(6,2)=3\,\mathbf{M}^\opt$, the equal-weights identity behind the
constancy of the facial column sums on this instance.

\paragraph{The multiplicative-algorithm Jacobian, exactly.}
The map $\Phi(\mathbf{p})=(p_iv_i(\mathbf{p})/d)_i$ (\Cref{rem:titterington}) has
$D\Phi(\mathbf{p}^\opt)_{ij}=\delta_{ij}v_j(\mathbf{p}^\opt)/d-p^\opt_iB_{ij}/d$, which evaluates to the rational
matrix
\[
D\Phi(\mathbf{p}^\opt)=
\begin{pmatrix}
\tfrac13&0&-\tfrac16&-\tfrac16\\[2pt]
0&\tfrac34&0&0\\[2pt]
-\tfrac16&-\tfrac38&\tfrac13&-\tfrac16\\[2pt]
-\tfrac16&-\tfrac38&-\tfrac16&\tfrac13
\end{pmatrix},
\]
whose columns sum to $0$ (the differential maps into the sum-zero subspace, as it must:
$\sum_ip_iv_i\equiv d$). Its exact eigenpairs are verified by four matrix--vector products:
\[
\begin{array}{llll}
\lambda=0:&(1,0,1,1)^\top,\qquad&
\lambda=\tfrac12:&(0,0,1,-1)^\top\ \text{and}\ (1,0,-\tfrac12,-\tfrac12)^\top,\\[2pt]
\lambda=\tfrac34:&(\tfrac23,1,-\tfrac56,-\tfrac56)^\top .
\end{array}
\]
The three eigenvectors with $\lambda\in\{\tfrac12,\tfrac12,\tfrac34\}$ are sum-zero, so the spectrum
on the simplex tangent is exactly $\{\tfrac12,\tfrac12,\tfrac34\}$ with spectral radius
$\rho=\tfrac34<1$, as used in \Cref{thm:avg} (Step 2); the slow mode is the inactive coordinate,
contracting at $v_2(\mathbf{p}^\opt)/d=\tfrac34$ exactly (visible in the second row of $D\Phi(\mathbf{p}^\opt)$).
The full spectrum $\{0,\tfrac12,\tfrac12,\tfrac34\}$ means Ostrowski's theorem applies in $\R^4$
directly.

\paragraph{The linearized gap on the feasible cone.} The strict positivity $F>0$ on
$K\setminus\{0\}$ is proved by hand in \Cref{lem:Fpos} (main text), using only the rows
\eqref{eq:Brows} of $\mathbf{B}$ above. Numerically, the unit-sphere minimum is
$\gamma_{\min}=\min\{F(\boldsymbol{\delta}):\boldsymbol{\delta}\in
K,\norm{\boldsymbol{\delta}}_2=1\}\approx0.386$, and the limiting direction of the averaged iterates is
$\mathbf{C}=\lim_T T(\bar{\mathbf{p}}^{(T)}-\mathbf{p}^\opt)\approx(0.165,0.988,-0.577,-0.577)$ (unnormalized,
$\norm{\mathbf{C}}_2\approx1.28$) with $\gamma=F(\mathbf{C})\approx0.495$. The two constants measure different
things ($\gamma_{\min}$ is the worst case over unit directions, $\gamma$ the value at the actual
summed displacement, consistently $\gamma\ge\gamma_{\min}\norm{\mathbf{C}}_2$) and both corroborate
\Cref{lem:Fpos} and \Cref{thm:avg} ($C_2>0$ as proved there).

\section{The leverage-score model in context}
\label{app:lev-model}

The oracle studied in this paper (one query at a weighting $\mathbf{p}$ returns the vector
$(v_i(\mathbf{p}))_{i\in[n]}$, $v_i(\mathbf{p})=\mathbf{a}_i^\top\mathbf{M}(\mathbf{p})^{-1}\mathbf{a}_i$) is the central object of
\emph{randomized numerical linear algebra} (RandNLA), surveyed in~\cite{drineasmahoney2016,martinsson2020} with a software-oriented modern account in~\cite{murray2023}. The quantity $p_iv_i(\mathbf{p})$ is the $i$-th
\emph{statistical leverage score}, the $i$-th diagonal entry of the weighted projection matrix
$\mathbf{P}^{1/2}\mathbf{A}\,\mathbf{M}(\mathbf{p})^{-1}\mathbf{A}^\top\mathbf{P}^{1/2}$ ($\mathbf{P}=\diag(\mathbf{p})$), measuring how much row $i$ influences
the column space; $\sum_ip_iv_i(\mathbf{p})=d$ (\Cref{prop:dict}\ref{it:trace}). We collect here the
algorithmic context, largely the work of others, that motivates measuring cost in leverage-score queries.

\paragraph{Leverage scores as importance weights.}
Sampling rows with probability proportional to their leverage scores yields relative-error solutions to
$\ell_2$ regression and low-rank approximation~\cite{drineas2006,drineas2011} (an idea introduced as an
algorithmic primitive, sampling \emph{actual} rows and columns, by the CUR decomposition of Mahoney and
Drineas~\cite{mahoney2009}), and is the engine of the sketching paradigm for matrix
computation~\cite{sarlos2006,mahoney2011,woodruff2014,halko2011}. In the graph setting the leverage scores of
the incidence matrix are exactly the \emph{effective resistances}, and sampling by them is spectral
sparsification~\cite{spielman2008}, computed by iterative leverage refinement~\cite{li2013} or combinatorial
resistance estimates~\cite{koutis2016}---the same principle that, in the $D$-optimal-design form, makes a
near-optimal $\mathbf{p}$ a near-isotropic reweighting of the rows. Closest to the present setting,
leverage-score sampling also serves as a data-reduction front-end for the \emph{minimum-volume covering
ellipsoid}, the enclosing counterpart of the John ellipsoid, in the big-data regime~\cite{harris2024}.

\paragraph{Ridge leverage scores and the kernel setting.}
Regularization replaces $\mathbf{M}(\mathbf{p})$ by $\mathbf{M}(\mathbf{p})+\lambda\mathbf{I}$ and the leverage score by the
\emph{ridge} leverage score, whose sum is the effective dimension; sampling by it gives statistically optimal
kernel ridge regression and Nystr\"om approximation~\cite{alaoui2015,rudi2015}, input-sparsity-time low-rank
approximation~\cite{cohenmusco2017}, recursive landmark selection~\cite{musco2017}, and distributed or
adaptive dictionaries~\cite{calandriello2017}. The leverage-sampling principle also has an \emph{online} form,
in which rows are kept or discarded irrevocably as they stream by~\cite{cohen2016online}.

\paragraph{Volume sampling, determinantal processes, and coresets.}
Beyond independent leverage sampling, \emph{volume sampling} and determinantal point processes select rows
jointly with probability proportional to a squared volume, giving unbiased estimators with
leverage-score-quality tails~\cite{derezinski2018,derezinski2021} and near-optimal subset selection for the
pseudoinverse norm~\cite{avron2013}. Generalizing leverage scores from least squares to arbitrary objectives
yields the \emph{sensitivity} of a datum, the basis of the coreset
framework~\cite{feldman2011,braverman2016,munteanu2018}.

\paragraph{Computing the oracle.}
For dense $\mathbf{A}$ one query costs $O(nd^{\omega-1})$ exactly. What makes the model practical is
\emph{approximation}: relative-error leverage scores in $O(nd\log n)$ time by Johnson--Lindenstrauss
sketching of the pseudoinverse~\cite{drineas2012}, in input-sparsity time $O(\nnz(\mathbf{A}))$ by sparse
subspace embeddings~\cite{clarkson2013,nelson2013,meng2013}, through structured fast
transforms~\cite{tropp2011,drineas2011}, and in current-matrix-multiplication time with optimal
dependence~\cite{cherapanamjeri2022}; recursive (ridge-)leverage sampling refines a uniform sample into a
leverage sample in nearly input-sparsity time~\cite{cohen2015uniform}. It is precisely this
\emph{per-query} cost that the input-sparsity~\cite{cao2025} and lazy-update~\cite{woodruff2025}
John-ellipsoid algorithms reduce---an axis orthogonal to the \emph{query count} studied here.

\paragraph{Optimal design as a tractable program.}
The John ellipsoid is the $D$-optimal instance of a broad family of experimental-design criteria, for which a
rich algorithmic literature obtains near-optimal designs: by convex relaxation and rounding under a
measurement budget~\cite{wang2017}, by regret minimization across the $A/D/E/V/G$ criteria~\cite{allenzhu2017},
by proportional volume sampling for $A$-optimality~\cite{nikolov2019}, and by combinatorial local
search~\cite{madan2019}. These works optimize the same objective measured in \emph{design size}; our averaging
barrier and acceleration phases measure it in \emph{oracle queries}.

\paragraph{Lewis weights and the John ellipsoid.}
The John ellipsoid is the $\ell_\infty$ endpoint of a family: a John design is the $\ell_\infty$
\emph{Lewis-weight} distribution~\cite{cohen2015}, and Lewis weights are the $\ell_p$ analogue of leverage
scores, themselves computed by a short sequence of reweighted leverage-score
evaluations~\cite{fazel2021}. Sampling by Lewis weights gives $\ell_p$ subspace embeddings, including in
the online, streaming, and sliding-window models~\cite{braverman2018,woodruff2022}, with \emph{sensitivity}
sampling~\cite{song2025} and turnstile-stream sampling data structures~\cite{song2025stream} the further
generalizations. All of this optimizes or approximates the \emph{cost}
of a query; our averaging barrier (\Cref{thm:avg}) and doubly logarithmic accuracy phase
(\Cref{thm:newton}) concern the \emph{number} of queries, hold however cheap each query is made, and
extending them to these approximate oracles is the noise-stability question of \Cref{sec:discussion}.

\section{Numerical experiments}
\label{app:experiments}

We validate each of the three results with a small, self-contained reference implementation, on the
explicit instance $\mathbf{A}_\opt$ above and on random instances
$\mathbf{A}=\mathbf{G}\,\diag(\boldsymbol{\sigma})^{1/2}$ ($\mathbf{G}$ standard Gaussian, $\boldsymbol{\sigma}$ geometrically
spaced---its spread controls the facial conditioning). The metric is the number of \emph{leverage-score
evaluations} (oracle queries) to reach a $(1+\eps)$-John ellipsoid $\max_iv_i(\mathbf{p})\le(1+\eps)d$; for the
facial phases we also report the \emph{face-phase iteration count}.

\paragraph{Dictionary and oracle (\Cref{prop:dict,prop:smrecover}).}
The identities of \Cref{prop:dict} hold to finite-difference accuracy ($\nabla f=-\mathbf{v}$ to
$3\times10^{-5}$; $\nabla^2f_{ij}=(\mathbf{a}_i^\top \mathbf{M}^{-1}\mathbf{a}_j)^2$ to $5\times10^{-4}$), the
multiplicative-map Jacobian spectrum to $10^{-6}$, and the Sherman--Morrison Hessian recovery
\eqref{eq:smrecover} to machine precision ($10^{-12}$) on random instances.

\paragraph{Averaging barrier and separation (\Cref{thm:avg,cor:sep}).}
On $\mathbf{A}_\opt$, $T\cdot g(\bar{\mathbf{p}}^{(T)})\to0.495=\gamma$ with measured log--log slope $-1.00$
(\Cref{fig:avglaw}), while the last iterate reaches the double-precision floor in under $100$ steps.
The separation is then stark (\Cref{fig:separation}): the uniform averaged certificate needs
$\{4,25,248,2476,24752,247513\}$ leverage evaluations for $\eps\in\{10^{-1},\dots,10^{-6}\}$ (exactly
$\Theta(\eps^{-1})$, a factor $10$ per decade), whereas CCLY's own last iterate, away-step Frank--Wolfe,
and both facial phases of \Cref{alg:main} are polylogarithmic.

\paragraph{Acceleration, and $\log$ versus $\log\log$ (\Cref{thm:upper,thm:newton}).}
Across random instances of growing conditioning, the face-phase count scales as $\kappa^{0.87}$ for
unaccelerated projected gradient versus $\kappa^{0.43}$ for the accelerated method (\Cref{fig:facial}, right),
the $\sqrt\kappa$ rate of \Cref{thm:upper}. The two facial phases part ways along every axis
(\Cref{fig:loglog}): in \emph{accuracy} the accelerated count grows linearly in $\log(1/\eps)$ while the
facial Newton phase stays doubly-logarithmically flat (\Cref{thm:newton}); in \emph{conditioning} the Newton
count is moreover \emph{condition-free}; and although each Newton step spends $O(d^2)$ leverage queries to
the accelerated phase's one, in \emph{total} queries the Newton phase still overtakes once $\eps$ is small
enough.

\paragraph{End-to-end cost, counting the warm start.}
The face-phase plots above isolate the accuracy term; the honest total also pays the $\eps$-independent
warm-start $C(\mathbf{A})$ of \Cref{lem:identify}. \Cref{fig:global} counts it, all from the uniform start
$\mathbf{1}/n$. \emph{(a)} The uniform-averaging certificate is $\Theta(1/\eps)$, whereas both of our methods
sit at a flat $C(\mathbf{A})$ plus a slow tail---so they overtake averaging only past a crossover, and at
\emph{crude} accuracy the simple average is in fact cheaper, having paid no setup. \emph{(b)} The total is
dominated by the setup: the warm start is a fixed floor and the Newton accuracy phase adds only a
doubly-logarithmic sliver. \emph{(c)} The conditioning lives in the warm start: across instances
$C(\mathbf{A})$ and the accelerated accuracy phase grow with $\kappa$, while the Newton accuracy phase is
flat (condition-free), as \Cref{thm:newton} predicts. \emph{(d)} End-to-end, then, the total still scales
with conditioning---through $C(\mathbf{A})$, not the accuracy term; closing that dependence is the
identification question of \Cref{prob:uniform}.

\paragraph{Reproducibility.}
Every quantity is reproduced by iterating the maps defined in the body and forming the displayed rational
matrices. The claims on $\mathbf{A}_\opt$ are deterministic---the closed forms of $\mathbf{p}^\opt$, $\mathbf{B}$,
$\widehat{\mathbf{H}}(\mathbf{p}^\opt)$, and $D\Phi(\mathbf{p}^\opt)$ are exact, and the separation counts follow from
iterating the CCLY map and the methods of \Cref{alg:main}; only the conditioning sweeps draw random
instances.

\begin{figure}[H]
\centering
\includegraphics[width=\textwidth]{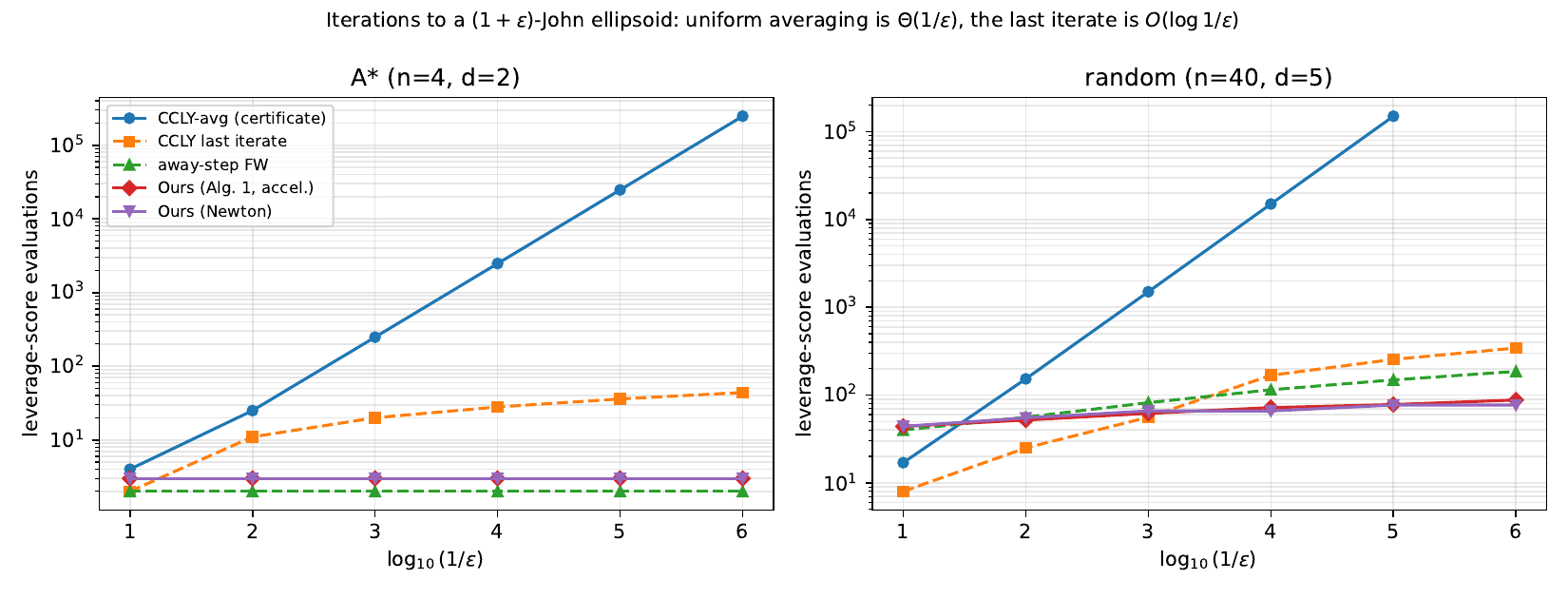}
\caption{The averaged-versus-last-iterate separation of \Cref{cor:sep}, on $\mathbf{A}_\opt$ (left) and a
random instance (right): the uniform-averaging certificate (CCLY-avg) needs $\Theta(\eps^{-1})$ leverage
evaluations, whereas CCLY's own last iterate, away-step Frank--Wolfe, and both facial phases of
\Cref{alg:main} (accelerated and Newton) are polylogarithmic in $1/\eps$.}
\label{fig:separation}
\end{figure}

\begin{figure}[H]
\centering
\includegraphics[width=\textwidth]{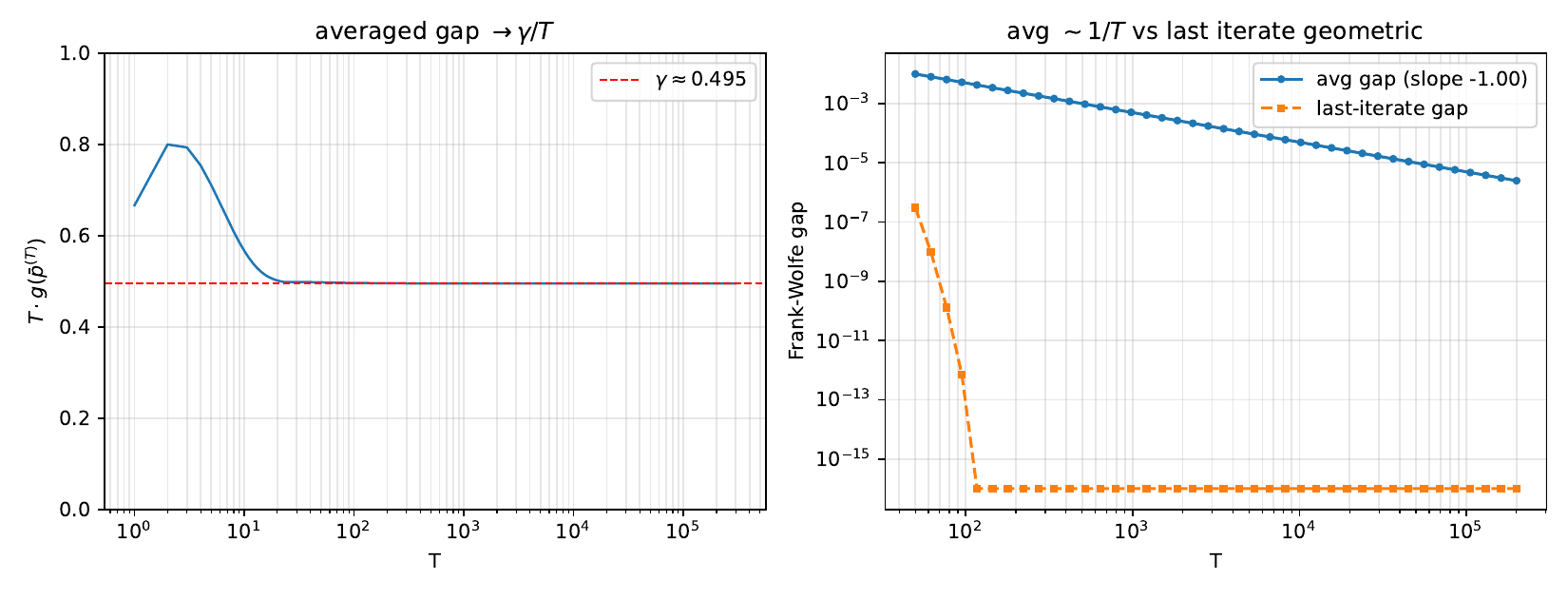}
\caption{The averaging law on $\mathbf{A}_\opt$ (\Cref{thm:avg}): the gap of the uniform running average obeys
$T\cdot g(\bar{\mathbf{p}}^{(T)})\to\gamma\approx0.495$ (so $g\propto1/T$), while the last iterate falls to the
double-precision floor.}
\label{fig:avglaw}
\end{figure}

\begin{figure}[H]
\centering
\begin{minipage}[b]{0.49\textwidth}\centering
\includegraphics[width=\textwidth]{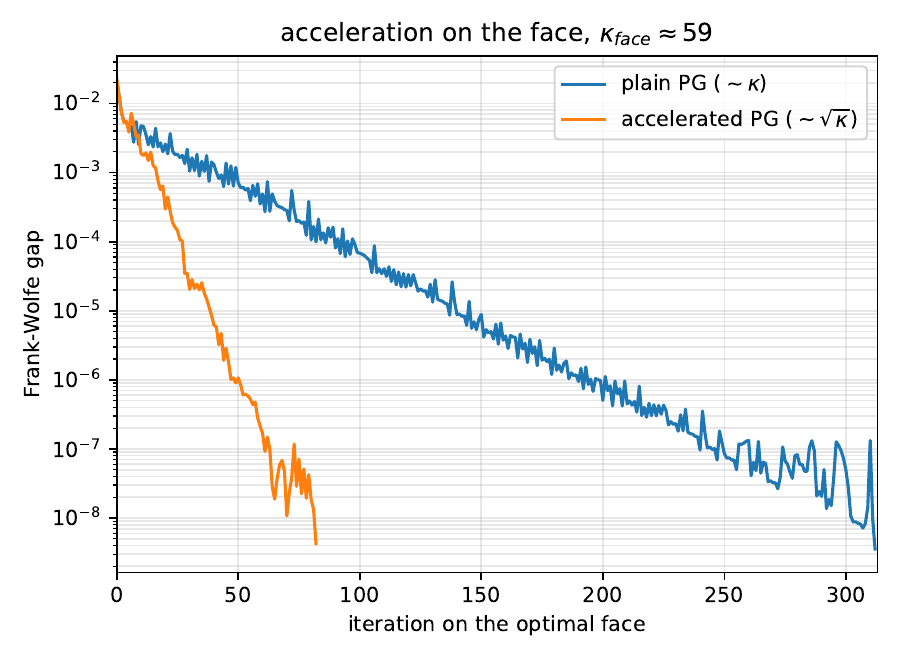}
\end{minipage}\hfill
\begin{minipage}[b]{0.49\textwidth}\centering
\includegraphics[width=\textwidth]{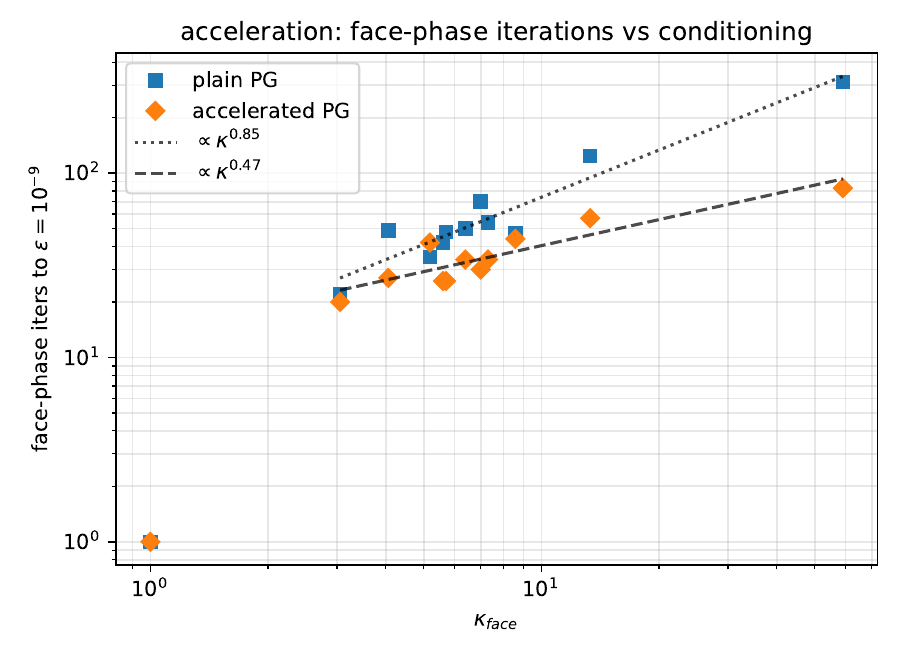}
\end{minipage}
\caption{The accelerated facial phase. \emph{Left:} accelerated versus plain projected gradient on an
instance with $\kappa\approx59$ (gap versus face-phase iteration). \emph{Right:} the face-phase count
versus $\kappa$ across random instances, scaling as $\kappa^{0.43}$ (accelerated) versus $\kappa^{0.87}$
(plain)---the $\sqrt\kappa$ rate of \Cref{thm:upper}.}
\label{fig:facial}
\end{figure}

\begin{figure}[H]
\centering
\includegraphics[width=\textwidth]{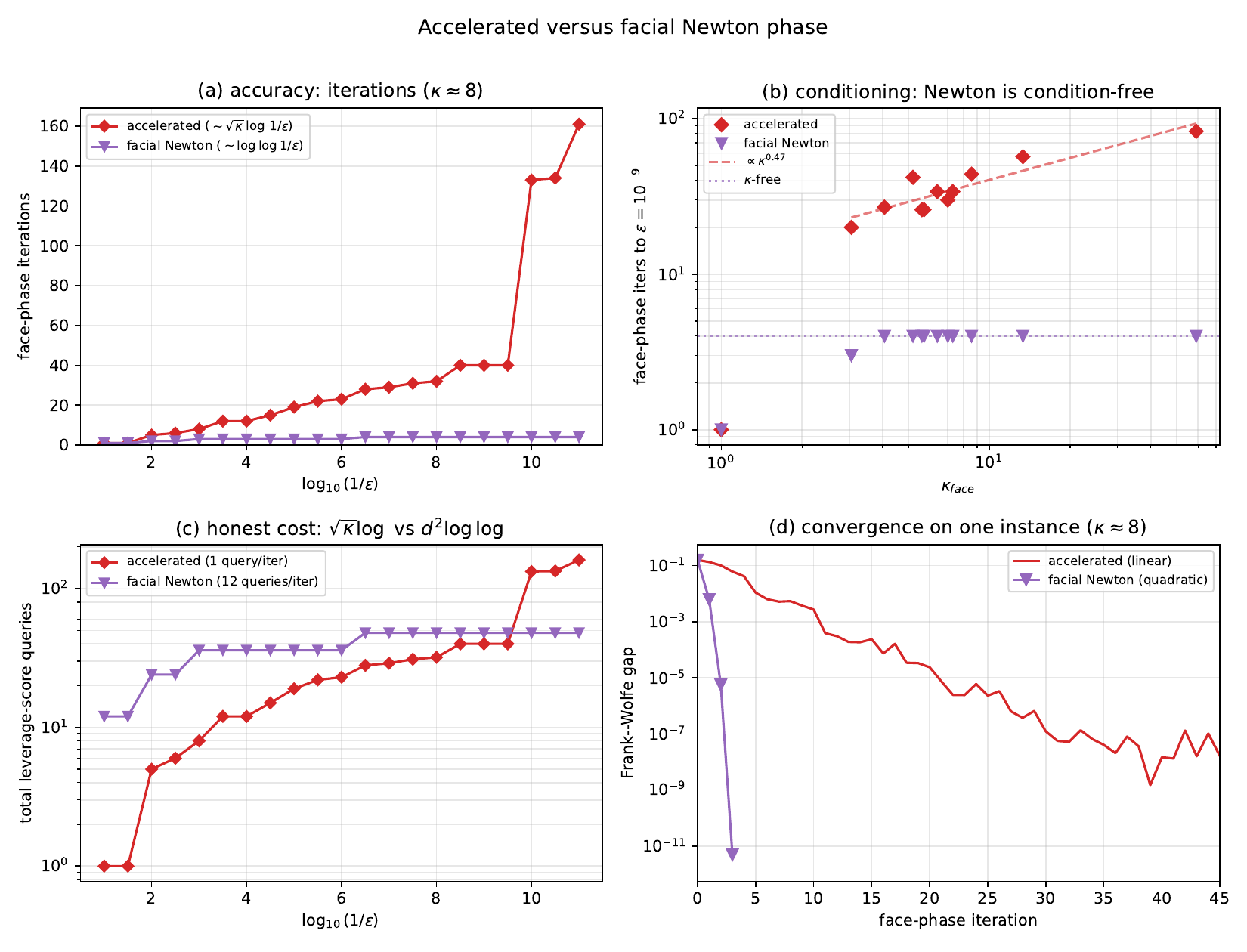}
\caption{The accelerated (\Cref{thm:upper}) versus facial Newton (\Cref{thm:newton}) phase, four views.
\emph{(a)} \emph{Accuracy:} face-phase iterations grow linearly in $\log(1/\eps)$ for the accelerated phase
and stay essentially flat (doubly logarithmic, $\log\log(1/\eps)$) for the Newton phase ($\kappa\approx8$).
\emph{(b)} \emph{Conditioning:} across random instances the accelerated count scales as $\sqrt\kappa$
(fit $\propto\kappa^{0.47}$), whereas the Newton count is \emph{condition-free} (flat), confirming that no
condition number multiplies the $\eps$-dependent term of \Cref{thm:newton}. \emph{(c)} \emph{Honest cost:}
the Newton phase spends $m{+}1=O(d^2)$ leverage queries per iteration against the accelerated phase's one,
so in \emph{total} queries ($\sqrt\kappa\log(1/\eps)$ versus $d^2\log\log(1/\eps)$) the accelerated phase is
cheaper at moderate accuracy and the Newton phase wins only once $\eps$ is small enough. \emph{(d)} The
convergence mechanism on one instance: the accelerated gap decays linearly (geometrically), the Newton gap
quadratically, hitting the floor in a handful of steps.}
\label{fig:loglog}
\end{figure}

\begin{figure}[H]
\centering
\includegraphics[width=\textwidth]{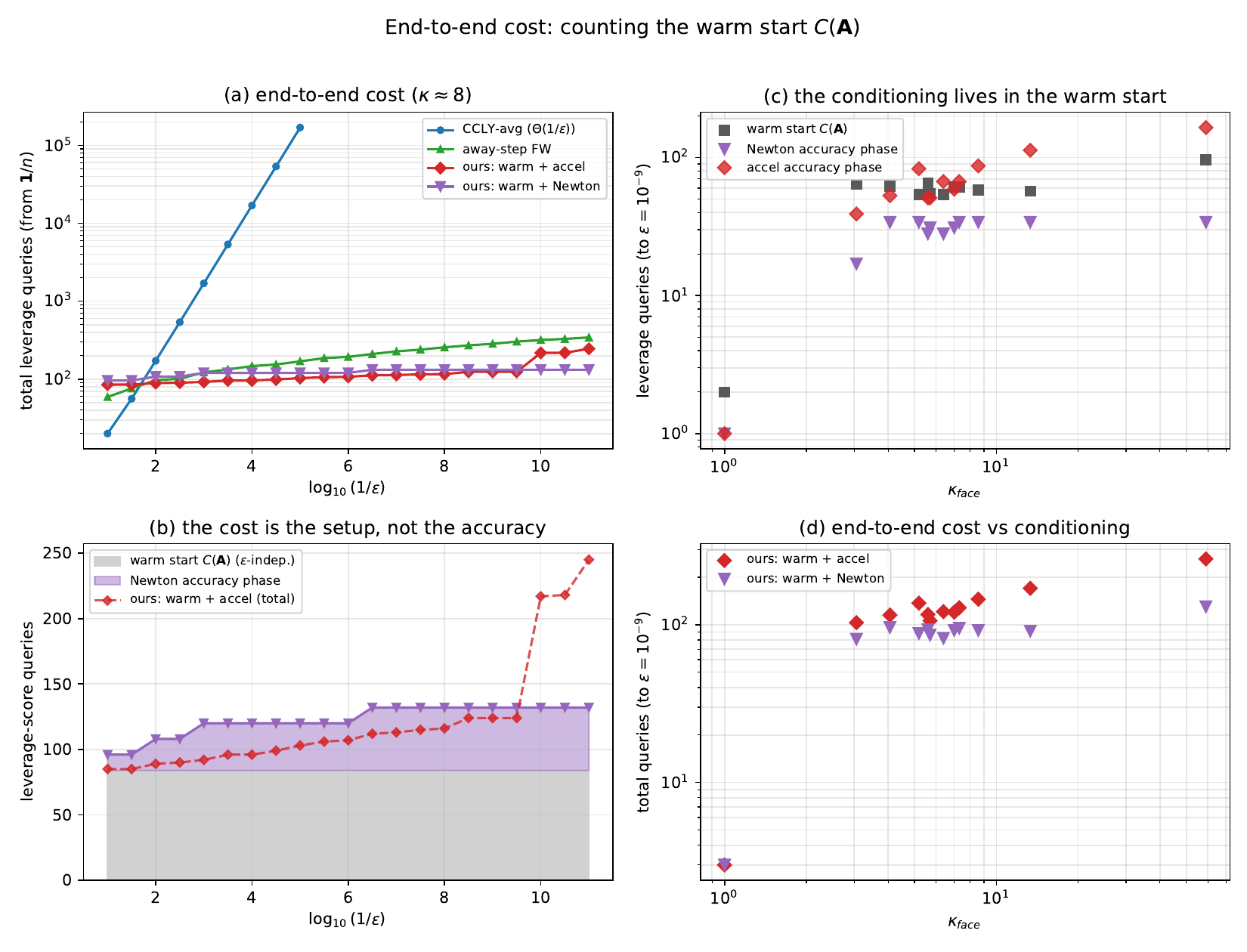}
\caption{End-to-end cost, including the $\eps$-independent warm start $C(\mathbf{A})$ (\Cref{lem:identify}),
all from the uniform start $\mathbf{1}/n$. \emph{(a)} Total leverage queries versus accuracy: CCLY's uniform
average is $\Theta(1/\eps)$ (cheaper only at crude accuracy, before our setup amortizes), whereas
warm\,$+$\,accel and warm\,$+$\,Newton are $C(\mathbf{A})$ plus a $\sqrt\kappa\log(1/\eps)$ resp.\
$d^2\log\log(1/\eps)$ tail. \emph{(b)} The same total decomposed: the warm start is a fixed floor, the Newton
accuracy phase a thin doubly-logarithmic band on top. \emph{(c)} Across random instances the warm start and
the accelerated accuracy phase grow with $\kappa$, while the \emph{Newton} accuracy phase is condition-free
(flat). \emph{(d)} End-to-end totals still grow with $\kappa$---the condition dependence sits in the warm
start, not the accuracy term.}
\label{fig:global}
\end{figure}

\end{document}